\documentclass[]{article}
\usepackage{ifpdf}

\newcommand{\Title}{VogrincKendall1}
\newcommand{\Subject}{\Title}

\usepackage[british]{babel}
\newcommand{\Date}{\today}

\usepackage{calc, xspace}
\usepackage{amsmath, amsthm, amsfonts, euscript}
\usepackage{subcaption}

 \usepackage[utf8x]{inputenc}
\usepackage{ccfonts,eulervm}

  \usepackage[no-comment]{WSKtracking}

\definecolor{linkcolor}{named}{Maroon}
\definecolor{citecolor}{named}{PineGreen}
\definecolor{urlcolor}{named}{RoyalPurple}
\definecolor{okcolor}{named}{OliveGreen}
\definecolor{alertcolor}{named}{BrickRed}

\ifpdf
 \DeclareGraphicsExtensions{.png,.jpg}%
\else
 \DeclareGraphicsExtensions{.eps,.ps,.png,.jpg}%
\fi

\usepackage[totalwidth=480pt,totalheight=680pt]{geometry}
 \usepackage[
 \ifpdf
   pdftex,
   unicode,
 \fi
 ]{hyperref}
 \hypersetup{
 pdfauthor={Jure Vogrinc and Wilfrid S. Kendall},
 pdftitle={\Title},
 pdfsubject={\Subject},
 }
 \hypersetup{citecolor=citecolor}
 \hypersetup{linkcolor=linkcolor}
 \hypersetup{urlcolor=urlcolor}
 \hypersetup{colorlinks=true}

\RequirePackage{WSKmaths}
\graphicspath{{image/}}

\usepackage{amssymb}
\newtheorem{appendixlem}{Lemma}[section]

\newcommand{\LMHR}{\rho}
\newcommand{\fBM}{B^{(H)}}
\newcommand{\qu}{\varphi_c}

\newcommand{\V}[1]{V(#1\,|\fBM;c)}
\newcommand{\Vd}[1]{\dot{V}(#1\,|\fBM;c)}
\newcommand{\Vdd}[1]{\ddot{V}(#1\,|\fBM;c)}

\newcommand{\auxset}{\mathcal{S}}
\newcommand{\SigmaField}{\mathcal{F}}
\newcommand{\X}{\mathcal{X}}
\newcommand{\upi}{\xi}
\newcommand{\Region}{\mathcal{R}}
\newcommand{\Expectp}[2]{\operatorname{\mathbb{E}}_{#1}\left[#2\right]}

\newcommand{\sign}{\operatorname{sign}}
\newcommand{\ESJD}{\operatorname{\text{ESJD}}}

\theoremstyle{plain}
 \newtheorem{AssumptionFramework}[thm]{Anomalous Scaling Framework}

\begin{document}
 \TrackingPreamble

 \title{Counterexamples for optimal scaling of\\
 Metropolis-Hastings chains  with rough target densities}
 \author{%
Jure Vogrinc and Wilfrid S. Kendall
\\
                 \href{mailto:jure.vogrinc@warwick.ac.uk}%
                 {\scriptsize jure.vogrinc@warwick.ac.uk}, %
                 \href{mailto:w.s.kendall@warwick.ac.uk}%
                 {\scriptsize w.s.kendall@warwick.ac.uk}
}
 \date{\Date}
 \maketitle

\begin{abstract}\noindent

For sufficiently smooth targets of product form it is known that the variance of a single coordinate of the proposal in RWM (Random walk Metropolis) and MALA (Metropolis adjusted Langevin algorithm) should optimally scale as $n^{-1}$ and as $n^{-\frac{1}{3}}$ with dimension \(n\),
and that the acceptance rates should be tuned to $0.234$ and $0.574$. We establish counterexamples
to demonstrate that smoothness assumptions of the order of $\mathcal{C}^1(\Reals)$ for RWM and $\mathcal{C}^3(\Reals)$ for MALA are indeed required if these scaling rates are to hold. 
The counterexamples 
identify classes of marginal targets for which these guidelines
are violated, obtained by perturbing a standard Normal density (at the level of the potential for RWM and  the second derivative of the potential for MALA) using roughness generated by a path of fractional Brownian motion with Hurst exponent $H$. 
For such targets there is strong evidence that RWM and MALA proposal variances should optimally be scaled as $n^{-\frac{1}{H}}$ and as $n^{-\frac{1}{2+H}}$ and will then
obey anomalous acceptance rate guidelines. 
Useful heuristics resulting from this theory are discussed.
The paper develops a framework capable of tackling optimal scaling results for quite general Metropolis-Hastings algorithms (possibly depending on a random environment).
\end{abstract}



%
\section{Introduction}\label{sec:introduction}
Probabilistic computation and optimisation are tools of widespread importance in
applied mathematical science,
and are widely used in order to facilitate
the use of Bayesian statistics, especially in machine learning contexts.
In particular, the use of Markov chain Monte Carlo (MCMC) methods 
is now wide-spread. 
This greatly increases the value of mathematical theory underlying these methods;
many significant theoretical advances have 
indeed been made but theory still lags behind the explosive growth 
of many varieties of applications. 
Theory typically provides significant help and guidance by studying 
basic building blocks of these algorithms,
applied to toy examples which are nevertheless representative of applications
\citep{Diaconis-2013,RobertsRosenthal-2001b}.

A remarkable example of theoretical guidance is provided by
results on
optimal scaling of MCMC
(\citet{RobertsGelmanGilks-1997,RobertsRosenthal-1998}, see also \citet{Gelfand-1991} who establish diffusion limits for MCMC algorithms). 
Here the toy examples are high-dimensional ``product targets''
(multivariate probability densities which render all 
\(n\) coordinates
independent and identically distributed). 
Optimal scaling results show that
(under suitable regularity conditions)
as dimension \(n\) increases
so the proposal variances of each coordinate 
of the Random walk Metropolis (RWM) 
and the Metropolis adjusted Langevin algorithm (MALA) proposals 
should respectively be chosen proportional to
$n^{-1}$ and $n^{-1/3}$. 
Furthermore, it proves optimal as \(n\to\infty\) to choose the constant of proportionality so as to obtain average acceptance rates of
\(0.234\) of the proposed moves for RWM and \(0.574\) for MALA. 
These results were originally proved only for the toy example of product targets 
given above; 
nevertheless simulation evidence suggests that they 
should hold in much greater generality 
and notable progress has been made to generalise the theory
towards more general targets, 
especially in the case of RWM: see for example 
\citet{YangRobertsRosenthal-2019}. 
Consequently the theory does indeed provide very practical and useful 
guidelines for practitioners,
and additionally provides an important context for motivating and assessing
adaptive MCMC methods.

The theoretical results require smoothness assumptions 
for the underlying marginal target density function. 
\citet{RobertsGelmanGilks-1997}
actually
required \(3\) continuous derivatives
for their RWM result 
while \citet{RobertsRosenthal-1998}
needed \(8\) continuous derivatives
for their approach to MALA.
These assumptions were necessitated by the methods of proof but did not
otherwise seem particularly natural 
and it was unclear to what extent they were actually necessary.
Recent work has used different methods of proof
to establish, at least in the case of RWM, that 
the original smoothness assumptions were
indeed much stricter than is really
required 
\citep{DurmusLeCorffMoulinesRoberts-2016,ZanellaKendallBedard-2016}.
The main focus of our paper is to develop a class of counterexamples to
demonstrate
the extent to which some kinds of
smoothness assumption
are genuinely necessary for both RWM and MALA.

In summary, we show that a certain level of smoothness of 
the marginal target density
function
is indeed required in order to deliver the original optimal scaling guidelines.
To be specific,
RWM essentially requires
\(1\) continuous derivative almost everywhere 
while
MALA requires
\(3\) continuous derivatives
almost everywhere.
Note that 
no derivatives are required
in order 
for RWM to deliver the prescribed target probability measure as a large-time equilibrium,
while MALA requires just one derivative.
Nevertheless we show that some 
higher order smoothness is indeed necessary
if the algorithms are to scale in a standard way.
In the following it is shown that, in
the absence of suitable smoothness, 
there exist classes of targets for which
the above optimality results do not apply,
and indeed
different, anomalous, tuning guidelines 
appear to be optimal. 
Note 
in particular
that failure of smoothness
at isolated points
(as often occurs in applications)
need
not 
be
sufficient to destroy standard smoothing
\citep{DurmusLeCorffMoulinesRoberts-2016}:
our counterexamples 
are necessarily non-smooth over a substantial range. 
However, the counterexamples tell us something fundamental about the way in which RWM and MALA
really do depend on regularity and are thus methodologically interesting. They quantifiably exhibit another, often overlooked, way in which MCMC can perform badly, different for instance from the target having multiple modes or being zero in large parts of space. A bottleneck in MCMC mixing can also be caused by local roughness or oscillations and we believe the results presented below  
do indicate useful
aspects of scaling behaviour for MCMC methods in such cases (see Section~\ref{sec:UsefullHeuristics}).

For RWM, for each \(0<H<1\) we use 
a probabilistic approach
to construct a 
class of product targets
which
lie in $\mathcal{C}^\gamma(\Reals)$ (for \(\gamma<H\)) but not in $\mathcal{C}^H(\Reals)$,
and for
which the RWM algorithm does \emph{not} scale optimally in the way
indicated by the theory in \cite{RobertsGelmanGilks-1997}.
Indeed an
``Expected Squared Jump Distance'' (ESJD) approach
indicates a different and anomalous manner of optimal scaling.
For MALA, for each \(0<H<1\) we similarly
use 
a probabilistic approach
to construct a class of product targets
which
lie in $\mathcal{C}^{2+\gamma}(\Reals)$ (for \(\gamma<H\))
but not in $\mathcal{C}^{2+H}(\Reals)$ 
and for which 
again the MALA algorithm does not
scale optimally
according to the
regular-case
theory of \citet{RobertsRosenthal-1998};
here an ESJD approach again indicates
a different and anomalous manner
of optimal scaling.

Our method of approach is to generate random targets
-- in effect, random environments --
based on a random realisation of a two-sided \(H\)-fractional Brownian motion path.
Indeed, bearing in mind appropriate density theorems for Gaussian measures,
in some sense our counterexamples are generic!
We use the generated path to construct a 
marginal probability density function
such that 
any possibility of optimal 
scaling could only arise by
tuning the coordinate variance of proposals for the associated $n$-dimensional product targets 
to be proportional to $n^{-1/H}$ for RWM
and to $n^{-1/(2+H)}$ for MALA (instead of $n^{-1}$ for RWM and $n^{-1/3}$ for MALA).

In addition the method of proof may be of independent interest. Section~\ref{sec:CLTgeneral} provides a  very suitable framework for addressing optimal scaling questions for Metropolis-class MCMC methods, particularly for identifying minimal required smoothness conditions. It is plausible that similar frameworks can be obtained for other classes of MCMC algorithms. Independently of that, an approach involving random targets, similar to Section~\ref{sec:anomalous_scaling}, could be used to construct other kinds of counterexamples in MCMC.

The rest of the paper is organized as follows. 
Section \ref{sec:results} states and discusses the main results of the paper. 
Section~\ref{sec:CLTgeneral} establishes conditions,
nearly as general as possible,
in the setting of product targets 
under which an associated Central Limit Theorem holds for the log Metropolis-Hastings ratio and 
a non-trivial limiting acceptance rate exists.
Section~\ref{sec:isserlis} states and proves consequences
of the celebrated Isserlis theorem which will later be used to control distributions of important quantities expressed 
in the context of a random environment.
Section~\ref{sec:anomalous_scaling} introduces a general framework for showing when anomalous scaling
can occur for general Metropolis-Hastings algorithms applied to product targets
in which the marginal product target density depends on a random environment which is a continuous Gaussian process.
Sections~\ref{sec:RWM} and Sections~\ref{sec:MALA} respectively verify that that the general framework of Section~\ref{sec:anomalous_scaling} 
is satisfied in cases of anomalously scaled RWM and MALA. Targets used for RWM (respectively MALA) are perturbations, on the level of potential (respectively second derivative of the potential), of the standard Normal density. 
Finally, Section~\ref{sec:conclusion} discusses 
considerations concerning Expected Square Jump Distance,
open questions, potential extensions and 
how heuristics 
suggested by these theoretical results could be useful in applications.

%
\section{Main results of the paper}\label{sec:results}

This section
presents our main results in more detail.
First of all, recall the mathematical framework of optimal scaling for MCMC.
The marginal probability density function for the product target measure
(assumed here to be strictly positive)
is denoted by \(\pi\).
Thus the product target measure on \(\Reals^n\) is given by
\[
\Pi_n({\d}x)\quad=\quad\Pi_n({\d}x_1,\dots {\d}x_n)
\quad=\quad
\prod_{i=1}^n\Big(\pi(x_i){\d}x_i\Big)
\,.
\]
Our results concern asymptotic behaviour 
 (as the dimension \(n\) grows to \(\infty\))
 of MCMC algorithms
 which deliver this target measure as large-time equilibrium.
The relevant algorithms, RWM and MALA, give rise to Markov chains
 $\left(X^{\text{RWM},(n)}_{k} \;:\; k=1, 2, \ldots\right)$  for RWM
 and 
 $\left(X^{\text{MALA},(n)}_{k}\;:\; k=1, 2, \ldots\right)$ for MALA
 (here the dummy index \(k\) is the discrete time variable for the Markov chains).
 The chains are Metropolis-Hastings (MH) algorithms with 
 target probability measures \(\Pi_n\)
 based on multivariate normal proposals
$Q^{\text{RWM},(n)}(x,{\d}y)\sim N\left(x,\frac{\ell^2}{n}\cdot I_n\right)$  for RWM
and 
$Q^{\text{MALA},(n)}(x,{\d}y)\sim N\left(x+\frac{\ell^2}{2n^{1/3}}\;\nabla (\log\Pi_n(x)),\frac{\ell^2}{n^{1/3}}\cdot I_n\right)$ for MALA,
and we consider the stationary versions of all these chains 
(so initial distribution is always \(\Pi_n\)).
Here \(\ell>0\) is a parameter determining the asymptotic scale of the proposal.


The classic results of \citet{RobertsGelmanGilks-1997} and \cite{RobertsRosenthal-1998} state 
(respectively for
$\pi\in\mathcal{C}^3(\Reals)$
for RWM,
and 
 $\pi\in\mathcal{C}^8(\Reals)$
for MALA)
that
as \(n\to\infty\) there is weak convergence 
of the first coordinate of the chain
(under certain conditions on the decay of the tails and the regularity of the
marginal probability density of \(\Pi_n\))
 \begin{equation}\label{eq:Langevin}
 X^{\text{RWM},(n)}_{\lfloor n\cdot t\rfloor,1}
\quad \xrightarrow{w} \quad U_t
\qquad\text{and}\qquad 
 X^{\text{MALA},(n)}_{\lfloor n^{1/3}\cdot t\rfloor,1}
\quad \xrightarrow{w} \quad U_t
 \end{equation}
to a Langevin diffusion \(U\), a solution of the \emph{continuous time}
stochastic 
differential equation
\[
{\d}U_t
\quad=\quad
h(\ell)^{1/2}{\d}B_t+\frac{h(\ell)}{2}\;\nabla(\log(\pi(U_t)){\d}t\,.
\]
That is to say, 
the accelerated first coordinates $X^{\text{RWM},(n)}_{\lfloor nt\rfloor,1}$ 
and 
$X^{\text{MALA},(n)}_{\lfloor n^{1/3}t\rfloor,1}$, 
when considered as piece-wise constant continuous time processes,
converge weakly to the Langevin diffusion $U$ as the dimension $n$ increases.
The expressions for diffusion speeds $h(\ell)$ are different in RWM and MALA cases 
and optimizing over the choice of $\ell$ then leads 
to different (but appealingly simple) acceptance ratio guidelines.

The computational heart of these results lies in
the task of showing 
that the acceptance ratio converges to a constant different to zero or one,
and this follows by application of 
a version of the Central limit theorem (CLT) 
that applies to the coordinate-wise logarithms of MH acceptance ratios
for these algorithms. 
For instance, if $X_1,\dots X_n$ are the 
independent and identically distributed 
(IID) coordinates of $X^{RWM,(n)}\sim\Pi_n$ 
and $Y_1,\dots Y_n$ are the IID coordinates of the RWM proposal
$Y^{RWM,(n)}\sim Q^{\text{RWM},(n)}(X^{RWM,(n)},{\d}y)$ 
then the following CLT
\[
\sum_{i=1}^n \log(\pi(Y_i))-\log(\pi(X_i))
\quad\xrightarrow{w}\quad
N\left(-\half\sigma^2,\sigma^2\right)
\] 
holds for an appropriate constant 
$\sigma^2=\ell^2\int_{\Reals}\left((\log\pi)'(x))\right)^2\pi(x){\d}x$. 
This then identifies the limiting average acceptance ratio via
\[
\alpha\left(X^{RWM,(n)},Y^{RWM,(n)}\right)
\;=\;
(1\wedge \exp)\left(\sum_{i=1}^n \log(\pi(Y_i))-\log(\pi(X_i))\right)
\quad\xrightarrow{w}\quad
(1\wedge\exp)\left(N\left(-\half\sigma^2,\sigma^2\right)\right)
\,,
\]
where $(1\wedge \exp)(x)$ denotes $\min(1,e^x)$ for $x\in\Reals$.
Note that this identifies 
the optimal scaling rate for the coordinates of the proposal:
if a scaling rate is not asymptotic to the rate giving a CLT 
($n^{-1}$ for RWM and $n^{-1/3}$ for MALA)
then either there is no limiting average acceptance rate
or the limit 
is necessarily $0$ or $1$.

We now construct classes of marginal probability density functions for which 
\emph{anomalous} scaling occurs at least at the level of ESJD.
We do this 
by using a randomized construction based on fractional Brownian motions. 
Recall that $\{\fBM_x,~x\in\Reals\}$ is a two-sided 
\emph{fractional Brownian motion} (fBM) with Hurst parameter $H\in(0,1)$
if it is a centred zero-mean Gaussian process with 
covariance defined for \emph{arbitrary}
$x,y\in\Reals$ by
\begin{equation}\label{eq:fBM}
\Gamma^{(H)}(x,y)
\quad=\quad
\Expect{\fBM_x \fBM_y}
\quad=\quad
\tfrac{1}{2}|x|^{2H}+\tfrac{1}{2}|y|^{2H}-\tfrac{1}{2}|x-y|^{2H}\,.
\end{equation}
Here we refer to \citet[Chapter 5]{Nualart-2006}
for fBM theory.
This reference covers the single-sided fBM
with \(x\geq0\): however extension to the double-sided
case is immediate if one notes that \eqref{eq:fBM}
remains non-negative definite for all \(x,y\).
\citet[Chapter I, Exercise (3.9)]{RevuzYor-1991}
gives an explicit and succinct
construction for all \(x,y\)
(see also \citealp{MandelbrotVanNess-1968}).
The sample paths of fBM with Hurst parameter \(H\)  
are almost surely H\"older continuous of exponent \(\gamma\) 
whenever \(0<\gamma<H\) (though not for \(\gamma=H\)). Let $\Omega^{(H)}$ denote the space of all two-sided paths that are zero at time zero and are in $\mathcal{C}^{\gamma}(\Reals)$ for all $0<\gamma<H$, so $\Omega^{(H)}$ is in fact a probability space equipped with a measure provided by two-sided fBM with Hurst parameter \(H\).

The main result concerning RWM counterexamples can be summarised as follows 
(where $I_n$ denotes the $n$-dimensional identity matrix):

\begin{thm}[Anomalous scaling for RWM]\label{thm:CLT_RWMvanilla}
Consider the random function \(\upi_\fBM\) depending on 
the fractional Brownian motion \(\fBM\)
and defined by
\[
\upi(x|\fBM)\quad= \quad \frac{1}{\sqrt{2\pi}}\exp\left(\fBM_x-\frac{x^2}{2}\right)\,.
\]
Almost surely $\int_{-\infty}^\infty \upi(x|\fBM){d}x<\infty$,
so
\(\upi(\cdot\,|\fBM)\) can be renormalized to provide a (random) target density
\[
\pi(x|\fBM)\quad=\quad
\frac{\upi(x|\fBM)}{\int_{-\infty}^\infty \upi(u|\fBM){d}u}\,.
\]
Condition on \(\fBM\) and consider a stationary RWM chain with 
target $\Pi_n(~\cdot~|\fBM)=\prod_{i=1}^n \pi(~\cdot~|\fBM)$ and 
proposal $Q^{\text{RWM},(n)}(x,{\d}y)\sim N\left(x,{\ell^2}{n^{-1/H}}\cdot I_n\right)$.
Then there is a constant 
\(\sigma^2=\ell^{2H}~\frac{2^H}{\sqrt{\pi}}\Gamma(H+\frac{1}{2})\) 
such that, as \(n\to\infty\),
the probability of acceptance of the proposal (conditional on the underlying \(\fBM\))
satisfies
\[
\alpha\left(X^{RWM,(n)},Y^{RWM,(n)}\right)
\quad\xrightarrow{w} \quad
(1\wedge\exp)\left(N\left(-\half\sigma^2,\sigma^2\right)\right).
\]
almost surely (for almost all realisations of the fBM $\fBM$).
\end{thm}
In effect \(\fBM\) is providing a random environment, \(X^{RWM,(n)}\)
is a Markov chain using this random environment, and Theorem \ref{thm:CLT_RWMvanilla}
refers to the quenched behaviour of this Markov chain in a random environment.

We draw attention to the 
anomalous rate of scaling of the proposal variances.
The reader should 
keep in mind that 
substantially
different rates of proposal scaling will yield either a trivial limit
or no limiting behaviour at all. 
In particular, the optimal scaling of \citet{RobertsGelmanGilks-1997} cannot here apply. It is also possible to see that this rate of proposal variance decay is optimal in terms of the ESJD. However, we have to pose it in a slightly different way than classically: for any decay rate of proposal variances the ESJD (random, because it depends on the environment) divided by the optimal
ESJD rate converges to zero in probability.
We outline the proof and discuss this further in Section~\ref{sec:ESJD}.

Given
the Hurst parameter $H$ and the rate $\ell n^{-1/H}$
of optimal proposal variance decay,
one can then optimise the ESJD decay rate over the choice of $\ell$. This gives us an optimal acceptance rate for each $H$. 
The function cannot be expressed in closed form 
but can be plotted numerically, see the left panel of Figure~\ref{fig:figure1}.
Note that the optimal acceptance rate converges to zero as $H\to 0$, and for example it is optimal to accept approximately $7\%$ of the proposals for $H=1/2$ and only $0.7\%$ for $H=1/4$.

The analogous result concerning MALA requires
definition of a localisation function 
$\qu(x)\colon\Reals\to[0,1]$ 
depending on a parameter $c>0$ and defined for $x\in\Reals$ (with $\qu(0)=1$) 
as follows
\begin{equation}\label{eq:localisation}
\qu(x)\quad=\quad 1\wedge \left(c^{\frac{3}{2H}}\;|x|^{-3}\right)
\quad=\quad\min\left\{1,~ c^{\frac{3}{2H}}\;|x|^{-3}\right\}\,.
\end{equation}
We will consider perturbations of a normal density by a fBM path at the level of the second derivative of the potential and the localisation function is introduced to control fBM fluctuations and ensure the resulting random target is integrable.

Anomalous scaling of MALA can then occur as follows.
\begin{thm}[Anomalous scaling for MALA]\label{thm:CLT_MALAvanilla}
Consider the random function \(\upi(\cdot\,|\fBM;c)\) depending on 
the fractional Brownian motion \(\fBM\)
and defined by
\[
\upi\left(x|\fBM;c\right)\quad=\quad 
\frac{1}{\sqrt{2\pi}}\exp\left(-\frac{x^2}{2}+x^2\int_{0}^1\fBM_{xs}\varphi_c(xs)(1-s){\d}s\right)\,.
\] 
For every Hurst index $H\in(0,1)$ there exists a small enough $c>0$, such that 
almost surely
$\int_{-\infty}^\infty \upi(x|\fBM;c){d}x$
is finite,
so
\(\upi(\cdot\,|\fBM;c)\) can be renormalized to provide a (random) target density
\[
\pi(x|\fBM;c)\quad=\quad
\frac{\upi(x|\fBM;c)}{\int_{-\infty}^\infty \upi(u|\fBM;c){d}u}\,.
\]
Condition on \(\fBM\) and consider a stationary MALA chain with 
target $\Pi_n(~\cdot~|\fBM;c)=\prod_{i=1}^n \pi(~\cdot~|\fBM;c)$ and 
proposal 
$Q^{\text{MALA},(n)}(x,{\d}y)\sim N\left(x+\frac{1}{2}\ell^2n^{-1/(2+H)}\;\nabla (\log\Pi_n(x)),\,\ell^2n^{-1/(2+H)}\cdot I_n\right)$. Then there is a constant $\sigma^2>0$
such that, as \(n\to\infty\),
the probability of acceptance of the proposal (conditional on the underlying \(\fBM\))
satisfies
\[
\alpha\left(X^{MALA,(n)},Y^{MALA,(n)}\right)
\quad\xrightarrow{w} \quad
(1\wedge\exp)\left(N\left(-\half\sigma^2,\sigma^2\right)\right)
\]
almost surely (for almost all realisations of the fBM $\fBM$).
We may take
\[
\sigma^2\quad=\quad
\ell^{4+2H}\times \frac{2^{1+H}\Gamma(H+\tfrac52)}{\sqrt{\pi}}\times\frac{H}{2+7H+7H^2+2H^3}\times\int_{-\infty}^\infty \qu(x)^2\pi(x|\fBM;c)dx\,.
\]
\end{thm}
We emphasize
that here the log marginal target density is twice differentiable
and \(H\) measures the roughness of the second derivative as noted below.
(In the RWM case \(H\)  measures the roughness 
of the log marginal target density itself.)

Again, the anomalous rate
means that the optimal scaling of \cite{RobertsRosenthal-1998} cannot here apply, 
moreover that the rate of proposal variance decay is optimal. Again optimizing over the choice of $\ell$ leads to different optimal acceptance rates for different values of $H$. 
As exhibited in the right panel of Figure~\ref{fig:figure1},
and as in the RWM case, the optimal acceptance rate increases
as \(H\) increases. In the MALA case $H$ measures the roughness of the second derivative of the target (so $\pi$ is "$2+H$ smooth") and the optimal acceptance rate does not decay to zero as $H$ does. In fact both plots of Figure~\ref{fig:figure1} are obtained by numerically solving the same equation (see Section~\ref{sec:ESJD})) over different ranges of smoothness parameter, $(0,1)$ for RWM and $(2,3)$ MALA.
The choice of optimal acceptance rate also seems to be much more robust in case of MALA; this is supported by the numerical examples in Section~\ref{sec:UsefullHeuristics}.

\begin{Figure}
\begin{subfigure}{0.5\textwidth}
\begin{center}
\includegraphics[width=0.85\linewidth]{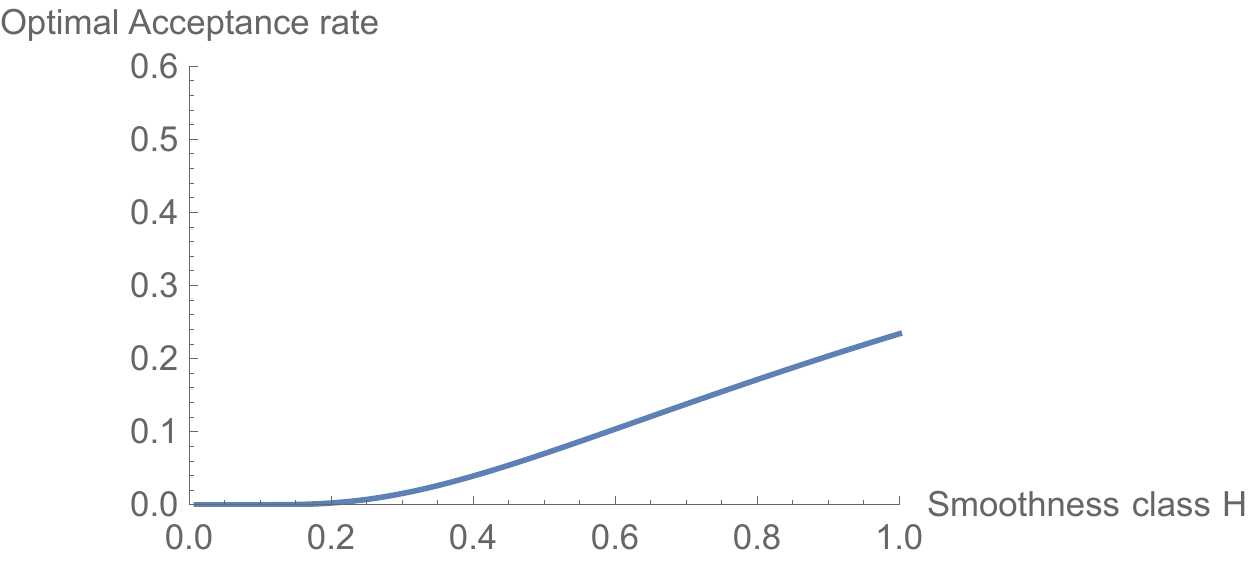}
\end{center}
\end{subfigure}%
\begin{subfigure}{0.5\textwidth}
\begin{center}
\includegraphics[width=0.85\linewidth]{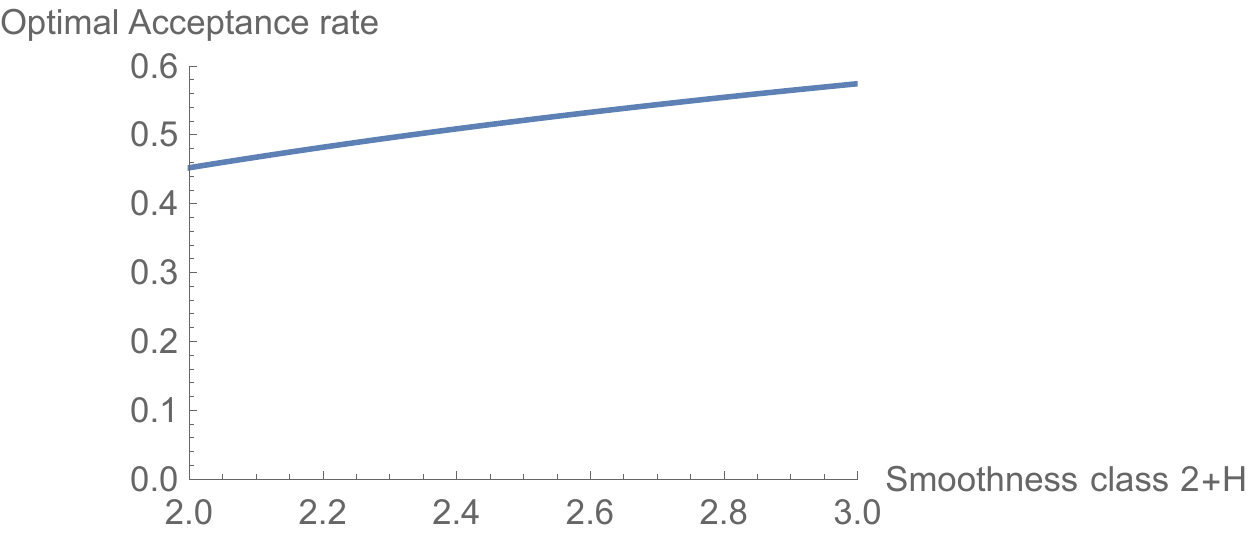}
\end{center}
\end{subfigure}
\caption{Optimal acceptance rates for RWM (left) for smoothness classes $H\in(0,1)$ and MALA (right) for smoothness classes $2+H$ with $H\in(0,1)$.
The optimal acceptance rates are obtained by numerically solving the same
equation (in terms of smoothness class) in both cases.}\label{fig:figure1}
\end{Figure}

The marginal target probability densities 
for these counterexamples are chosen to facilitate simple proofs; many other constructions work equally well.
The RWM choice is a fractional Brownian perturbation of a Normal density;
the MALA choice is based on a fractional Brownian perturbation 
at the level of the second derivative of the log-density, so
\[
\left(\frac{d}{{\d}x}\right)^2\left(\log(\pi(x|\fBM;c))\right) \quad=\quad-1+\qu(x)\fBM\,.
\]
The proofs will work for other kinds of perturbation,
and indeed it is an interesting question
what exactly are the analytical features of a marginal target probability density
that would lead to anomalous scaling.

In this paper we do not proceed to establish weak convergence to Langevin diffusion
limits, because the current results are sufficient to establish counterexamples. This and other related questions are further discussed in Section~\ref{sec:D&D}. Section~\ref{sec:UsefullHeuristics} discussed the question of what useful heuristics can be learned from these results.

%
\section{Generalities concerning Metropolis-Hastings log acceptance rates}\label{sec:CLTgeneral}
In this section we describe a 
general framework for proving CLT-type results such as Theorems~\ref{thm:CLT_RWMvanilla} and \ref{thm:CLT_MALAvanilla}. 
The treatment follows the style of 
\citet{Tierney-1998}, and applies to rather general Metropolis-Hastings (MH) samplers. 

Let $(\X,\SigmaField)$ be a measurable space,
supporting a probability measure \(\pi\) (the ``target probability distribution'') 
and a Markov kernel \((Q(x,\cdot):x\in\X)\) (the ``proposal mechanism'').
Recall that \((Q(x,\cdot):x\in\X)\) is a Markov kernel if 
(i) \(x\mapsto Q(x,A)\) is measurable for any \(A\in\SigmaField\)
and
(ii) \(B\mapsto Q(x,B)\) is a probability measure on $(\X,\SigmaField)$ 
for any \(x\in\X\).
As described by \citet[Proposition 1]{Tierney-1998}, 
let \(\Region\in\SigmaField\otimes\SigmaField\) be the symmetric set 
such that \(\pi(\d{x})Q(x,\d{y})\) and \(\pi(\d{y})Q(y,\d{x})\) are mutually absolutely continuous on $\Region$ and mutually singular off $\Region$.
\citeauthor{Tierney-1998} notes that
$\Region$ is unique up to differences of sets which are null with respect to both these measures.

We define the log MH acceptance ratio (log-MH-ratio) \(\rho\) by%
\begin{equation}\label{eq:LMHRgeneral}
\LMHR(x,y)\quad=\quad 
 \begin{cases}
\log\left(\frac{\pi(\d{y})Q(y,\d{x})}{\pi(\d{x})Q(x,\d{y})}\right)
           & \text{if } (x,y)\in \Region\,,\\
0          & \text{otherwise.}  
 \end{cases}
\end{equation}
We write \(\LMHR\) for the random variable \(\LMHR(X,Y)\),
where \((X,Y)\) has 
distribution given by \(\pi(\d{x})Q(x,\d{y})\).

It is straightforward to verify that the definition of \(\LMHR\)
and the reversibility of the MH algorithm under \(\pi\)
together imply the following computational relationships.

\begin{proposition}\label{prop:LMHR_properties}
With \(\LMHR(x,y)\) defined as above,
\begin{enumerate}[label=\emph{\ref{prop:LMHR_properties}.(\alph*)}]
\item\label{subprop:antisymmetry}
$\LMHR (x,y)=-\LMHR (y,x)$ almost everywhere with respect to $\pi(\d x)Q(x,\d y)$.
\item\label{subprop:de-exponentiation}
Let \(f: \Reals\to\Reals\) be a measurable function such that
\(\Expect{|f(\LMHR)|}<\infty\), so that \(f\circ\LMHR\) is integrable with respect to
\(\pi(\d{x})Q(x,\d{y})\). Then 
$\Expect{f(-\LMHR )e^{\LMHR }}=\Expect{f(\LMHR )}$.
\end{enumerate}
\end{proposition}

\begin{proof} 
By Radon-Nikodym theorem \ref{subprop:antisymmetry}
is immediate from \eqref{eq:LMHRgeneral} and the symmetry of \(\Region\).

To establish Proposition \ref{subprop:de-exponentiation},
argue as follows. 
Using \ref{subprop:antisymmetry},
we know
that \(f(-\LMHR)e^{\LMHR}\Indicator{\LMHR\leq \kappa}\) will be integrable
against the probability measure \(\pi(\d{x})Q(x,\d{y})\)
for each positive constant \(\kappa\), hence
\begin{multline*}
 \Expect{f(-\LMHR)e^{\LMHR}\;;\;\LMHR\leq \kappa}\quad=\quad
 {\int\int}_{\LMHR(x,y)\leq \kappa} f(-\LMHR(x,y))e^{\LMHR(x,y)} \pi(\d{x})Q(x,\d{y})
 \\
 \quad=\quad
 {\int\int}_{(x.y)\in\Region \text{ and }\LMHR(x,y)\leq \kappa} f(-\LMHR(x,y))e^{\LMHR(x,y)} \pi(\d{x})Q(x,\d{y})
  +f(0)\iint_{\Region^c}\pi(\d{x})Q(x,\d{y})
 \\
 \;=\;
 {\int\int}_{(x.y)\in\Region \text{ and }\LMHR(x,y)\leq \kappa} f(-\LMHR(x,y)) 
 \frac{\pi(\d{y})Q(y,\d{x})}{\pi(\d{x})Q(x,\d{y})} \pi(\d{x})Q(x,\d{y})
  +f(0)\iint_{\Region^c}\pi(\d{x})Q(x,\d{y})
  \\
 \;=\;
 {\int\int}_{(x.y)\in\Region \text{ and }\LMHR(x,y)\leq \kappa} f(-\LMHR(x,y))\pi(\d{y})Q(y,\d{x})
 +f(0)\iint_{\Region^c}\pi(\d{y})Q(y,\d{x}) 
 \\
 \quad=\quad
 {\int\int}_{\LMHR(y,x)\geq -\kappa} f(\LMHR(y,x))\pi(\d{y})Q(y,\d{x})
 \quad=\quad 
 \Expect{f(\LMHR)\;;\;\LMHR\geq -\kappa}\,.
\end{multline*}
The fact that $\Region$ is a symmetric set is used for the fourth equality, 
while Proposition \ref{subprop:antisymmetry} is used for the fifth.

Applying the above to the function \(|f|\), and letting \(\kappa\to\infty\),
yields \(\Expect{|f(-\LMHR )|e^{\LMHR }}=\Expect{|f(\LMHR )|}<\infty\) by monotone convergence.
Hence \(\Expect{f(-\LMHR)e^{\LMHR}\;;\;\LMHR\leq \kappa}\to\Expect{f(-\LMHR)e^{\LMHR}}\), by dominated convergence.
Proposition \ref{subprop:de-exponentiation} then follows by letting $\kappa\to\infty$ in the equation above.
\end{proof}

\begin{rem}\label{rem:identities}
Useful identities following from Proposition~\ref{prop:LMHR_properties} include:
\begin{enumerate}[label=\emph{\ref{rem:identities}.(\alph*)}]
 \item \label{item:unit}$\Expect{e^\LMHR}=1$; 
 \item\label{item:swap} 
 $\Expect{f(\LMHR)e^{\LMHR/2}}=\Expect{f(-\LMHR)e^{\LMHR/2}}$ whenever 
 \(f(\LMHR)e^{\LMHR/2}\) is an integrable random variable;
 \item\label{item:odd} 
 $\Expect{f(\LMHR)e^{\LMHR/2}}=0$ whenever $f$ is an odd function and
 \(f(\LMHR)e^{\LMHR/2}\) is an integrable random variable.
  \item\label{item:even} 
 If $f$ is an even function and $f(\LMHR)$ is an integrable random variable then $f(\LMHR)e^{t\LMHR}$ is an integrable random variable for all $t\in(0,1)$.
\end{enumerate}
\end{rem}

Typically, when establishing optimal scaling results, a key task
is to determine
when it is valid to assert asymptotically negligibility
of
half the variance of \(\rho\)
plus its mean.
The next few results establish when this asymptotically negligibility holds for a rather general context.
\begin{proposition}\label{prop:expectation identity}
Suppose that
\(\Expect{\LMHR^2}<\infty\).
Then 
$\Expect{\LMHR} +{\half}\Expect{\LMHR ^2}
=\Expect{\LMHR ^2\int_0^{{\half}}(1-e^{t\LMHR })\d{t}}$.
\end{proposition}

\begin{proof}
Consider the exact Taylor expansion (valid for all values of $\LMHR $)
\[
e^{\LMHR } \quad=\quad 
1+\LMHR +{\half}\LMHR ^2
+\LMHR ^2\int_0^1(1-t)(e^{t\LMHR }-1){\d}t \,.
\]
%
%
Taking expectations and using Remark \ref{item:unit},
\[
\Expect{\LMHR +\LMHR ^2/2}
\quad=\quad
\Expect{\LMHR ^2\int_0^1(1-t)(1-e^{t\LMHR }){\d}t}\,.
\]
The proof is concluded by the following sequence of equalities. They are justified respectively
by applying the Fubini-Tonelli theorem for exchanging order of integral
when the integrand is bounded above, 
using Proposition \ref{subprop:de-exponentiation} for the function $f(\rho)=\rho^2e^{t\rho}$ (which is integrable by Remark~\ref{item:even}), 
changing variables using $u=1-t$, and finally 
applying the Fubini-Tonelli theorem once more:
\begin{multline*}
\Expect{\LMHR ^2\int_{{\half}}^1(1-t)(1-e^{t\LMHR }){\d}t}
\quad=\quad
\int_{{\half}}^1(1-t)\Expect{\LMHR ^2(1-e^{t\LMHR })}{\d}t
\quad=\quad
\int_{{\half}}^1(1-t)\Expect{\LMHR ^2(1-e^{(1-t)\LMHR })}{\d}t\\
\quad=\quad
\int_{0}^{{\half}}u\Expect{\LMHR ^2(1-e^{u\LMHR })}{\d}u
\quad=\quad
\Expect{\LMHR ^2\int_0^{{\half}}u(1-e^{u\LMHR }){\d}u}\,.
\end{multline*}
\end{proof}

We now establish a bound on the right-hand side 
\(\Expect{\LMHR ^2\int_0^{{\half}}(1-e^{t\LMHR })\d{t}}\)
of Proposition \ref{prop:expectation identity}
which will be sufficient for our purposes.

\begin{lem}\label{lem:expansion remainder bound} 
Suppose that \(\Expect{\rho^2}<\infty\).
For every $\kappa>0$ the following bound holds:
\[
\left|\Expect{\LMHR ^2\int_0^{{\half}}(1-e^{t\LMHR }){\d}t}\right|
\quad<\quad
\sinh\left({\half}\kappa\right)\Expect{\LMHR ^2}+\Expect{\LMHR ^2\;;\;\LMHR<-\kappa}\,.
\]
\end{lem}

\begin{proof} 
Fix $\kappa>0$ and split the above integral into parts 
according to whether $\LMHR\in [-\kappa,\kappa]$, $\LMHR<-\kappa$ or $\LMHR>\kappa$. 
If $\LMHR \in [-\kappa,\kappa]$, then 
$\left|\int_0^{{\half}}(1-e^{t\LMHR }){\d}t\right|\leq 
{\half}\max\{1-e^{-{\half} \kappa},e^{{\half} \kappa}-1\}
=
{\half} (e^{{\half} \kappa}-1)$.
Using \({\half} (e^{{\half} \kappa}-1)<\sinh\left({\half}\kappa\right)\)
for the sake of simplicity, we obtain
\[
\left|\Expect{\LMHR ^2\int_0^{{\half}}(1-e^{t\LMHR }){\d}t
\;;\; -\kappa\leq \LMHR^2\leq \kappa}\right|
\quad<\quad
\sinh\left({\half}\kappa\right)\Expect{\LMHR ^2}\,.
\]

If $\LMHR <-\kappa$, then $1-e^{t\LMHR }<1$ so 
$\left|\Expect{\LMHR ^2\int_0^{{\half}}(1-e^{t\LMHR})\d{t}\;;\; \LMHR<-\kappa}\right|<{\half}\Expect{\LMHR ^2\;;\;\LMHR < -\kappa}$. 
Similarly, if $\LMHR >\kappa$, then $e^{t\LMHR }-1<e^{\LMHR }$.
Applying Proposition~\ref{subprop:de-exponentiation} 
to the function $f(\rho)=\rho^2 \Indicator{\kappa<\rho}$, this implies
\begin{equation*}
\left|\Expect{\LMHR ^2 \int_0^{{\half}}(1-e^{t\LMHR }){\d}t\;;\; \kappa< \LMHR}\right|
\quad<\quad {\half} \Expect{\LMHR ^2 e^{\LMHR}\;;\; \kappa< \LMHR}
\quad=\quad
{\half} \Expect{\LMHR ^2\;;\; \LMHR<-\kappa}\,.
\end{equation*}
The result follows by adding these three terms.
\end{proof}

Now consider a sequence of log-MH-ratios \(\LMHR_n\), 
possibly defined on different probability spaces
and associated with different target probability distributions \(\pi_n\),
proposal kernels \(Q_n\) and derived log-MH-ratios \(\rho_n\). 
Asymptotic negligibility of the second moment of
\(\LMHR_n\) and a technical condition weaker than the uniform integrability of
the scaled random variables
\(\LMHR_n^2/\Expect{\LMHR_n^2}\) on \(\rho_n<0\),
imply that
mean plus half variance of \(\LMHR_n\) is asymptotically negligible:

\begin{thm}[``half variance plus mean is asymptotically negligible'']\label{thm:mean+half-variance-is-small}
Suppose that \(\Expect{\LMHR_n^2}\to0\) as \(n\to\infty\),
and suppose moreover we can find positive constants \(\kappa_n\to0\) such that 
$\Expect{\LMHR_n^2 \;;\; \LMHR_n<-\kappa_n}=o(\Expect{\LMHR_n^2})$ as \(n\to\infty\). 
Then 
$\Expect{\LMHR_n}+{\half} \Expect{\LMHR_n^2}=o(\Expect{\LMHR_n^2})$\
as \(n\to\infty\); indeed
\[
 \left|\Expect{\LMHR_n}+{\half} \Expect{\LMHR_n^2}\right|
 \quad\leq\quad
  \sinh({\tfrac12}\kappa_n)\Expect{\LMHR_n^2}+\Expect{\LMHR_n^2\;;\;\LMHR_n<-\kappa_n}
 \quad=\quad o\left(\Expect{\LMHR_n^2}\right)\,.
\]
Moreover
\[
 \Expect{\LMHR_n}+{\half} \Var{\LMHR_n}\quad=\quad o\left(\Expect{\LMHR_n^2}\right)\,.
\]
\end{thm}

\begin{proof}
The inequality follows directly from
Proposition~\ref{prop:expectation identity} and Lemma~\ref{lem:expansion remainder bound}.
The asymptotic negligibility of mean plus half variance follows from
the observation that \(\Var{\LMHR_n}=\Expect{\LMHR_n^2}-\Expect{\LMHR_n}^2\):
we control \(\Expect{\LMHR_n}^2\) by iterating the argument, since
\(\left(\Expect{\LMHR_n}\right)^2=
\left(o(\Expect{\LMHR_n^2})-{\half} \Expect{\LMHR_n^2}\right)^2=o(\Expect{\LMHR_n^2})\)
(expanding the quadratic and using \(\Expect{\LMHR_n^2}\to0\)).
\end{proof}

We now establish a central limit theorem for suitable sums of independent log-Metropolis-Hastings ratios (as would arise when considering 
suitable product-distribution targets). The main requirement is
simply a
particular uniform integrability condition on the sequence of 
scaled squares of the 
log-Metropolis-Hastings ratios, corresponding to a Lindeberg condition.
\begin{thm}\label{thm:unconditional-clt}
Consider a triangular array formed by
\(\LMHR_{n,j}\) (for \(j=1, \ldots, m_n\), \(n=1,2,\ldots\)),
built out of row-wise independent 
log-Metropolis-Hastings-ratio random variables for targets 
$\pi_{nj}$ and proposals $q_{nj}$.
Suppose that there exists a sequence of positive numbers $\kappa_{n,j}$
such that $\lim_{n\to\infty}\sup_{j\leq m_n}\kappa_{n,j}=0$, 
and 
\(\Expect{\LMHR_{n,j}^2\;;\; \LMHR_{n,j} < -\kappa_{n,j}}=o\left(\Expect{\LMHR_{n,j}^2}\right)\)
uniformly, in the sense that
\begin{equation}\label{eq:control}
\lim_{n\to\infty}\sup_{j\leq m_n}
\frac{\Expect{\LMHR_{n,j}^2\;;\;\LMHR_{n,j}<-\kappa_{n,j}}}{\Expect{\LMHR_{n,j}^2}}
\quad=\quad
0\,.
\end{equation}
Suppose further that there exists a constant $\sigma^2<\infty$ such that 
$\lim_{n\to\infty}\sum_{j=1}^{m_n}\Expect{\LMHR^2_{nj}}=\sigma^2$. 
Then 
\[
\sum_{j=1}^{m_n}\LMHR_{n,j}
\quad\xrightarrow{w}\quad N\left(-\half\sigma^2,\sigma^2\right)\,. 
\]
\end{thm} 
 
\begin{proof}
The Lindeberg Central Limit Theorem
(see for example \citealp[Theorem 4.15]{Kallenberg-2010}) follows from the following conditions:  
\begin{enumerate}[label=(\roman*)]
\item \label{item:point1}
$\sum_{j=1}^{m_n}\Prob{|\LMHR_{n,j}|>\eps}\to 0$ for any $\eps>0$,
\item \label{item:point2}
 $\sum_{j=1}^{m_n}\Expect{\LMHR_{n,j}\text{; }|\LMHR_{n,j}|\leq 1}\to -\half\sigma^2$,
\item \label{item:point3}
 $\sum_{j=1}^{m_n}\Var{\LMHR_{n,j}\text{; }|\LMHR_{n,j}|\leq 1}\to \sigma^2$.
\end{enumerate}

Note that
Proposition~\ref{subprop:de-exponentiation}
implies that, 
for any non-negative function \(f\)
and
for any $n,j$,
\[
\Expect{f(|\LMHR_{n,j}|) \;;\; \LMHR_{n,j} > \kappa_{n,j}}
\quad\leq\quad 
\Expect{f(|\LMHR_{n,j}|)e^{\LMHR_{n,j}} \;;\; \LMHR_{n,j} > \kappa_{n,j}}
\quad=\quad
\Expect{f(|\LMHR_{n,j}|)\;;\; \LMHR_{n,j} < - \kappa_{n,j}}\,,
\]
and hence
\begin{equation}\label{eq:aux3}
\Expect{f(|\LMHR_{n,j}|) \;;\; |\LMHR_{n,j}| > \kappa_{n,j}}
\quad\leq\quad
2\Expect{f(|\LMHR_{n,j}|)\;;\; \LMHR_{n,j} < -\kappa_{n,j}}\,.
\end{equation}

Taking $f(x)=x^2$ in \eqref{eq:aux3} yields
the following 
for fixed \(\eps>0\) 
and for all sufficiently large integers $n$:
\begin{multline*}
\sum_{j=1}^{m_n}\Prob{\LMHR_{n,j}^2>\eps^2}
\quad\leq\quad
\frac{1}{\eps^2}\sum_{j=1}^{m_n}
 \Expect{\LMHR_{n,j}^2\;;\;|\LMHR_{n,j}|>\eps}
\quad\leq\quad \frac{1}{\eps^2}\sum_{j=1}^{m_n}\Expect{\LMHR_{n,j}^2\;;\;|\LMHR_{n,j}|>\kappa_{n,j}}
\\
\quad\leq\quad 
\frac{2}{\eps^2}\sum_{j=1}^{m_n}
\Expect{\LMHR_{n,j}^2\;;\; \LMHR_{n,j} < -\kappa_{n,j}}
\quad\leq\quad
\frac{2}{\eps^2} 
\left(\sum_{j=1}^{m_n}\Expect{\LMHR_{n,j}^2}\right)
\times
\sup_{j\leq m_n}\frac{\Expect{\LMHR_{n,j}^2\;;\;  \LMHR_{n,j} < -\kappa_{n,j}}}{\Expect{\LMHR_{n,j}^2}}\,.
\end{multline*} 
But we have supposed that 
\(\sum_{j=1}^{m_n}\Expect{\LMHR^2_{nj}}\to\sigma^2<\infty\), 
so CLT requirement \ref{item:point1} follows from \eqref{eq:control}.

Note \(\Expect{\LMHR_{n,j}\;;\;|\LMHR_{n,j}|>1}\leq \Expect{\LMHR_{n,j}^2\;;\;|\LMHR_{n,j}|>1}\).
It follows that
$\sum_{j=1}^{m_n}\Expect{\LMHR_{n,j}\;;\;|\LMHR_{n,j}|>1}\to 0$.

Note also
\(\Var{\LMHR_{n,j}\;;\;|\LMHR_{n,j}|> 1}\leq
\Expect{\LMHR_{n,j}^2\;;\;|\LMHR_{n,j}|> 1}\)
and therefore it also follows that
\begin{equation}\label{eq:negligible-excess-variance}
\sum_{j=1}^{m_n}\Var{\LMHR_{n,j}\;;\;|\LMHR_{n,j}|> 1}\quad\to\quad 0\,.
\end{equation}

Proposition~\ref{prop:expectation identity} and Lemma~\ref{lem:expansion remainder bound} imply
the asymptotic relationship 
\begin{multline}\label{eq:asymp-relation}
\sum_{j=1}^{m_n}
\left|\Expect{\LMHR_{n,j}}+\frac{1}{2}\Expect{\LMHR_{n,j}^2}
\right|
\quad\leq\quad
\sum_{j=1}^{m_n}\left(
\sinh\left({\half}\kappa_{n,j}\right) \Expect{\LMHR_{n,j}^2}
+
\Expect{\LMHR_{n,j}^2\;;\; \LMHR_{n,j}<-\kappa_{n,j}}
\right)
\\
\quad\leq\quad
\left( 
\sinh\left({\half}\sup_{j\leq m_n}\kappa_{n,j}\right)
+\sup_{j\leq m_n}\frac{\Expect{\LMHR_{n,j}^2\;;\; \LMHR_{n,j}<-\kappa_{n,j}}}{\Expect{\LMHR_{n,j}^2}}
\right)\sum_{j=1}^{m_n}\Expect{\LMHR_{n,j}^2}
\quad\to\quad0\,.
\end{multline}
The combination of
\(\sup_{j\leq m_n}\kappa_{n,j}\to0\) and
\(\sum_{j=1}^{m_n}\Expect{\LMHR^2_{nj}}\to\sigma^2\)
and \eqref{eq:control} together imply convergence to \(0\).

Hence CLT requirement \ref{item:point2} follows, 
since
$\sum_{j=1}^{m_n}\Expect{\LMHR_{n,j}\;;\;|\LMHR_{n,j}|>1}\to 0$,
together with the asymptotic relationship \eqref{eq:asymp-relation},
implies that
\[
\lim_{n\to\infty}
\sum_{j=1}^{m_n}\Expect{\LMHR_{n,j}\;;\;|\LMHR_{n,j}|\leq 1}
\quad=\quad
\lim_{n\to\infty}\sum_{j=1}^{m_n}\Expect{\LMHR_{n,j}}
\quad=\quad
-\frac{1}{2}\lim_{n\to\infty}\sum_{j=1}^{m_n}\Expect{\LMHR_{n,j}^2}
\quad=\quad
-\frac{\sigma^2}{2}\,.
\]

Finally, we deal with CLT requirement \ref{item:point3}.
First note that inequality \eqref{eq:aux3} 
implies $\Expect{\LMHR_{n,j}^2}\leq \kappa^2_{n,j}+2\Expect{\LMHR_{n,j}^2\;;\;\LMHR_{n,j}<-\kappa_{n,j}}$ for every $n,j$.
Hence $\lim_{n\to\infty}\sup_{j\leq m_n}\kappa_{n,j}=0$
and the uniform integrability
together imply that 
\begin{equation}\label{eq:uniform-negligibility}
\lim_{n\to\infty}\sup_{j\leq m_n}\Expect{\LMHR_{n,j}^2}=0\,. 
\end{equation}
Thus the asymptotic relationship \eqref{eq:asymp-relation} allows us to deduce
that as \(n\to\infty\) so
\begin{multline*}
\lim_{n\to\infty}\sum_{j=1}^{m_n}\Expect{\LMHR_{n,j}}^2
\quad=\quad
\lim_{n\to\infty}\frac{1}{4}\sum_{j=1}^{m_n}\Expect{\LMHR_{n,j}^2}^2
\quad\leq\quad 
\lim_{n\to\infty}\frac{1}{4}\sum_{j'=1}^{m_n}\Expect{\LMHR_{n,j'}^2}\;\cdot\;\sup_{j\leq m_n}\Expect{\LMHR_{n,j}^2}
\\
\quad=\quad
\frac{\sigma^2}{4} \lim_{n\to\infty}\sup_{j\leq m_n}\Expect{\LMHR_{n,j}^2}
\quad=\quad0\,.
\end{multline*}
This implies CLT requirement \ref{item:point3} as follows:
using \eqref{eq:negligible-excess-variance},
\begin{multline*}
\lim_{n\to\infty}\sum_{j=1}^{m_n}\Var{\LMHR_{n,j}\;;\;|\LMHR_{n,j}|\leq 1}
\quad=\quad
\lim_{n\to\infty}\sum_{j=1}^{m_n}\Var{\LMHR_{n,j}}
\\
\quad=\quad
\lim_{n\to\infty}\left(\left(\sum_{j=1}^{m_n}\Expect{\LMHR_{n,j}^2}\right)
-\left(\sum_{j=1}^{m_n}\Expect{\LMHR_{n,j}}^2\right)
\right)
\quad=\quad
\lim_{n\to\infty}\sum_{j=1}^{m_n}\Expect{\LMHR_{n,j}^2}
\quad=\quad\sigma^2
\,.
\end{multline*}
\end{proof}
 
\begin{rem}\label{rem:uniform-negligibility}
 Note that asymptotic as \(n\to\infty\) the second moment \(\Expect{\LMHR_{n,j}^2}\) is uniformly negligible (in \(j\)). 
 This is established by 
 \eqref{eq:uniform-negligibility} as a consequence of the assumptions of Theorem
 \ref{thm:unconditional-clt}.
\end{rem}

\begin{rem}\label{rem:CLTconditional}
The fundamental difference between this central limit theorem and those established in \citet{RobertsGelmanGilks-1997} and all subsequent optimal scaling results is as follows. Our result is not conditional on a specific target location. It concerns random variables that are simultaneously dependent on the target and the proposal draw, as opposed to showing that for most fixed target draws the log Metropolis Hastings ratios (viewed only as functions of the proposal) satisfy a central limit theorem. This subtle difference allows the use of weaker smoothness conditions.
\end{rem}

%
\section{Variations on Isserlis Theorem}\label{sec:isserlis}
We seek an analysis of optimal scaling
for RWM and MALA when the marginal
target probability density function depends on the Gaussian
random process given by the two-sided fBM \(\fBM\), as prescribed in 
Theorems \ref{thm:CLT_RWMvanilla} and \ref{thm:CLT_MALAvanilla}.
This analysis requires a variation on the
classical result 
of \citet{Isserlis-1918},
and consequent estimates and computations, 
which we now describe.

First we introduce some preliminary combinatorial notation.
Given a multiset $S$,
a \emph{pairing} is a partition of $S$ into pairs 
(each pair possibly containing the same element twice). 
A pairing is called \emph{proper} 
if each of its pairs contains two distinct elements.
Let $\mathcal{P}(S)$ denote
the set of all pairings of the multiset $S$, 
and let $\mathcal{P}^*(S)\subseteq\mathcal{P}(S)$ 
denote the set of all proper pairings. 
\citeauthor{Isserlis-1918}' theorem, sometimes also called Wick's theorem 
runs as follows:

\begin{thm}[\citealp{Isserlis-1918}]\label{thm:Isserlis}
Let $X=(X_1,X_2,\dots X_n)$ be centred multivariate Normal random variable
and consider a general multi\-set $S=\{s_1,s_2,\dots, s_m\}$, 
with $s_i\in\{1,2,\dots n\}$ for $i=1,\ldots, m$. Then
\[
\Expect{X_{s_1}X_{s_2}\cdots X_{s_m}}=
\begin{cases}
\sum_{p\in\mathcal{P}(S)}\prod_{\Lambda\in p}\Expect{X_{\lambda_1}X_{\lambda_2}},	
                &\text{ if $m$ is even,}\\
0,		&\text{ if $m$ is odd,}\\
\end{cases}
\]
where the product is taken
over all pairs $\Lambda=\{\lambda_1,\lambda_2\}$ of a pairing $p$.
\end{thm}
A proof can be found, for example, in \citet[Theorem~1.28]{Janson-1997}.

\citeauthor{Isserlis-1918}' theorem leads to the following proposition.
\begin{proposition}\label{prop:Isserlis1}
Let $X=(X_1,X_2,\dots X_{n+k})$ be a 
centred multivariate Normal random variable with covariance matrix $R$ 
and let $S=2\times\{1,2,\dots,n\}\cup\{n+1,n+2,\dots,n+k\}$,
using a compact multiset notation to signal that
elements $1,\dots n$ appear twice while elements 
$n+1,\dots,n+k$ appear once only. Then
\begin{equation*}
\Expect{(X_1^2 -R _{11})(X_2^2-R_{22})\cdots (X_n^2-R_{nn})X_{n+1}X_{n+2}\cdots X_{n+k}}
\,=\,
\begin{cases}
\sum_{p\in\mathcal{P}^*(S)}\prod_{\Lambda\in p}\Expect{X_{\lambda_1}X_{\lambda_2}}
               &\text{ for even $k$,}\\
0 	       &\text{ for odd $k$.}\\
\end{cases}
\end{equation*}
\end{proposition}
\noindent
Note the crucial difference between expansions
in Theorem \ref{thm:Isserlis}
and Proposition \ref{prop:Isserlis1}:
in the proposition the sum is taken over the set \(\mathcal{P}^*\)
of \emph{proper} pairings.

\begin{proof}
It suffices to consider the result when the covariance matrix \(R\)
lies in the interior of the set of all valid covariance matrices:
the general result then follows by a continuity argument.
This allows us to argue algebraically,
viewing relevant expectations as multivariate
polynomials in the entries of \(R\).

First note that the result follows trivially if \(k\) is odd:
apply inclusion-exclusion
of \(R_{11}, R_{22}, \ldots, R_{nn}\) to
\[
\Expect{(X_1^2 -R _{11})(X_2^2-R^2_{22})\cdots (X_n^2-R_{nn})X_{n+1}X_{n+2}\cdots X_{n+k}}\,.
\] 
In case of odd $k$, each term in the inclusion-exclusion expansion must vanish
by Theorem \ref{thm:Isserlis}. So we need consider only the case of even \(k\).

Consider the \citeauthor{Isserlis-1918} expansion
of 
$I_1=\Expect{X_1^2X_2^2\cdots X_n^2\cdot X_{n+1}X_{n+2}\cdots X_{n+k}}$,
viewed as a sum of monomials in the entries of \(R\).
According to the combinatorial expression for this 
given in Theorem \ref{thm:Isserlis}, 
if we remove all monomials involving any of \(R_{11}, R_{22}, \ldots, R_{nn}\)
then the remaining sum is exactly
\(\sum_{p\in\mathcal{P}^*(S)}\prod_{\Lambda\in p}\Expect{X_{\lambda_1}X_{\lambda_2}}\).

Now consider the \citeauthor{Isserlis-1918} expansion of 
\(I_2=\Expect{(X_1^2-R_{11})(X_2^2-R_{22})\cdots (X_n^2-R_{nn})X_{n+1}\cdots X_{n+k}}\),
viewed as a linear combination of monomials of \(X_1, X_2, \ldots, X_{n+k}\).
This agrees with the \citeauthor{Isserlis-1918} expansion
of 
$I_1$
up to a difference of a linear combination of monomials involving non-empty
selections of \(R_{11}, R_{22}, \ldots, R_{nn}\).

The result is therefore proved if we can establish that, after cancellation,
the expansion of
\(I_2\) contains no terms involving any of the diagonal entries $R_{11},R_{22},\dots,R_{nn}$.

The joint moment generating function of the multivariate Normal variable $X$ equals $\exp\left(\frac{1}{2}t^{\top}Rt\right)$, for $t\in\Reals^{n+k}$. Hence \begin{equation*}
I_2\quad=\quad
\left[\left(\frac{\partial^2}{\partial t_1^2}-R_{11}\right)\dots \left(\frac{\partial^2}{\partial t_{n}^2}-R_{nn}\right)\frac{\partial}{\partial t_{n+1}}\dots\frac{\partial}{\partial t_{n+k}}\exp\left({\frac{1}{2}t^{\top}Rt}\right)\right]_{t_1=0,\dots,t_{n+k}=0}\,.
\end{equation*}
Viewing this as a smooth function of the vector \(t\) and the entries of \(R\) and denoting
the differential operator 
\(D=\left(\frac{\partial^2}{\partial t_2^2}-R_{22}\right)\dots \left(\frac{\partial^2}{\partial t_{n}^2}- R_{nn}\right) \frac{\partial}{\partial t_{n+1}}\dots\frac{\partial}{\partial t_{n+k}}\)
observe that
\begin{multline*}
\frac{\partial}{\partial R_{11}}\Expect{(X_1^2-R_{11})\cdots (X_n^2-R_{nn})X_{n+1}\cdots X_{n+k}}
\\
\quad=\quad
\frac{\partial}{\partial R_{11}}\left[\left(\frac{\partial^2}{\partial t_1^2}-R_{11}\right)D
~~\exp\left({\frac{1}{2}t^{\top}Rt}\right)\right]_{t_1=0,\dots,t_{n+k}=0}
\\
\;=\;
\left[D\frac{\partial}{\partial R_{11}}\left(\frac{\partial^2}{\partial t_1^2}-R_{11}\right)
~~\exp\left({\frac{1}{2}t^{\top}Rt}\right)\right]_{t_1=0,\dots,t_{n+k}=0}
\;=\;
\left[
 D
 \frac{\partial}{\partial R_{11}}
 ~~\left(\sum_{i=1}^{n+k} R_{1i}t_i\right)^2
 \exp\left({\frac{1}{2} t^{\top}Rt}\right)
\right]_{t_1=0,\dots,t_{n+k}=0}
\\
\quad=\quad
\left[
t_1\;D
 ~~\exp\left({\frac{1}{2} t^{\top}Rt}\right)
\left(t_1\left(\sum_{i=1}^{n+k}R_{1i}t_i\right)^2+2\sum_{i=1}^{n+k}R_{1i}t_i\right)
\right]_{t_1=0,\dots,t_{n+k}=0}
\quad=\quad 0\,.
\end{multline*}
The second identity holds because we can swap the order of differentiation
and interchange differentiation with taking the limit $t_i\to 0$. 
This is justified since $\exp\left(\frac{1}{2}t^{\top}Rt\right)$ 
and all its derivatives are smooth 
and for a smooth function $f$ 
the functions $R_{11}\mapsto f(R_{11},t)$ converge uniformly to 
$R_{11}\mapsto f(R_{11},0)$ (as $t\to 0$) 
in some compact neighbourhood of $(R_{11},0)$.

The same argument applies for differentiation
with respect to $R_{22},\dots,R_{nn}$.
Thus it follows that
that \(I_2\) is free of all terms involving $R_{11}, R_{22},\dots,R_{nn}$,
hence must equal
\(\sum_{p\in\mathcal{P}^*(S)}\prod_{\Lambda\in p}\Expect{X_{\lambda_1}X_{\lambda_2}}\)
as required.
\end{proof}



\begin{lem}\label{lem:Isserlis_exponential}
Let $X=(X_0,X_1,X_2,\dots X_n)$ be a 
centred multivariate Normal random variable with covariance matrix $R$. 
Then
\begin{multline*}
\Expect{\exp\left(X_0\right) \times \prod_{i=1}^n(X_i^2 -R _{ii})}
\quad=\quad
\exp\left(\tfrac12 {R_{00}}\right)	
\times \Expect{\prod_{i=1}^n\left((X_i+R_{0i})^2 -R _{ii}\right) }
\\
\quad=\quad\exp\left(\tfrac12{R_{00}}\right)
\sum_{\substack{{A_1\sqcup A_2\sqcup A_3=\{1,2,\dots,n\}}\\ |A_2|\text{ is even}}}
\quad
\sum_{p\in\mathcal{P}^*(2A_1\cup A_2)}
2^{|A_2|} \prod_{i\in A_2}R_{0i}\times\prod_{i\in A_3}R_{0i}^2	\times\prod_{\Lambda\in p}R_{\lambda_1\lambda_2}
\,,
\end{multline*}
where the inner sum ranges over $\mathcal{P}^*(2A_1\cup A_2)$, the set of all proper pairings of the multiset $2A_1\cup A_2$
in which elements of $A_1$ appear twice and 
elements of $A_2$ appear once. The outer sum ranges over all three-fold partitions \(A_1, A_2, A_3\) 
of the set $\{1,2,\dots,n\}$ with $A_2$ containing evenly many elements.
\end{lem}

\begin{proof}
 We may suppose that \(X = C^\top Z\),
 where 
 \(Z\) is a \(k\)-dimensional
 standard Normal random variable
 and \(C\) is a \((k\times (n+1)\) matrix.
 Thus the covariance matrix of \(X\) is
 \(R=\Expect{X X^\top}=\Expect{C^\top Z Z^\top C}
 =C^\top C\),
 and
 \(X_i=e_i^\top C^\top Z\)
 where
 \(e_i\) is an \((n+1)\)-vector 
with \(1\) as the \(i^\text{th}\) entry and \(0\) elsewhere.
Hence, using \(z\) for a \((n+1)\)-vector
of integration corresponding to \(Z\),
and the translation invariance of Lebesgue measure,
\begin{multline*}
\Expect{\exp\left(X_0\right) \times \prod_{i=1}^n(X_i^2 -R _{ii})}
\quad=\quad
\Expect{
\exp\left(e_0^\top C^\top Z\right)
\prod_{i=1}^n\left(|e_i^\top C^\top Z|^2-R_{ii}\right)
}
\\
\quad=\quad
\frac{1}{(2\pi)^{(n+1)/2}}\int
\exp\left(e_0^\top C^\top z\right)
\prod_{i=1}^n\left(|e_i^\top C^\top z|^2-R_{ii}\right)
\exp\left(-\tfrac12 |z|^2\right)
{\d}z
\\
\quad=\quad
\frac{1}{(2\pi)^{(n+1)/2}}\int
\exp\left(e_0^\top C^\top (z+Ce_0)\right)
\prod_{i=1}^n\left(|e_i^\top C^\top (z+Ce_0)|^2-R_{ii}\right)
\exp\left(-\tfrac12 |z+Ce_0|^2\right)
{\d}z
\\
\,=\,
\frac{1}{(2\pi)^{(n+1)/2}}\int
\exp\left(e_0^\top C^\top (z+Ce_0)-\tfrac12 |z+Ce_0|^2+
\tfrac12 |z|^2\right)
\prod_{i=1}^n\left(|e_i^\top C^\top (z+Ce_0)|^2-R_{ii}\right)
\exp\left(-\tfrac12 |z|^2\right)
{\d}z
\\
\quad=\quad
\frac{1}{(2\pi)^{(n+1)/2}}\int
\exp\left(\tfrac12 e_0^\top C^\top Ce_0\right)
\prod_{i=1}^n\left(|e_i^\top C^\top z +R_{i0}|^2-R_{ii}\right)
\exp\left(-\tfrac12 |z|^2\right)
{\d}z
\\
\quad=\quad
\Expect{\exp(\tfrac12 R_{00})
\prod_{i=1}^n\left(\left(X_i +R_{i0}\right)^2-R_{ii}\right)}
\,.
\end{multline*}

Finally,
rewrite each factor \(\left((X_i+R_{0i})^2 -R _{ii}\right)\)
as $(X^2_i-R_{ii})+2R_{0i}X_i+R^2_{0i}$. 
Expanding the product accordingly, we obtain
\begin{multline*}
\Expect{\exp\left(X_0\right) \times \prod_{i=1}^n(X_i^2 -R _{ii})}
\\
\quad=\quad
\exp\left(\frac{R_{00}}{2}\right)
\sum_{A_1\sqcup A_2\sqcup A_3=\{1,2,\dots,n\}}
\quad
2^{|A_2|}		
\Expect{\prod_{i\in A_1}(X_i^2-R_{ii})\times \prod_{i\in A_2}X_i}\times \prod_{i\in A_2}R_{0i}\times\prod_{i\in A_3}R_{0i}^2
\,.
\end{multline*}

The result follows by applying Proposition~\ref{prop:Isserlis1}.
\end{proof}

We will also require the following combinatorial lemma 
in order to separate out groups of integration variables.

\begin{lem}\label{lem:Isserlis_product_bound}
Suppose that $S=2\times\{1,2,\dots,n\}\cup\{n+1,n+2,\dots,n+2k\}$ and $p\in\mathcal{P}^*(S)$ is a proper pairing. 
Then it is possible to partition $p$ into three disjoint sets of pairs
$p_1,p_2,p_3$ such that 
the pairs in each $p_i$ ($i=1,2,3$) are pairwise disjoint 
and 
$\lfloor \frac{n+k}{3}\rfloor\leq|p_i|\leq \lceil \frac{n+k}{3}\rceil$.

Furthermore,
if $X=(X_1,X_2,\dots X_{n+2k})$ 
is a centred multivariate Normal random variable, then
\[
\prod_{\Lambda\in p}	\Expect{X_{\lambda_1}X_{\lambda_2}}^2
\leq 
\frac{1}{3}	\prod_{\Lambda\in p_1}	\Expect{X_{\lambda_1}X_{\lambda_2}}^6
+\frac{1}{3}	\prod_{\Lambda\in p_2}	|\Expect{X_{\lambda_1}X_{\lambda_2}}^6
+\frac{1}{3}	\prod_{\Lambda\in p_3}	|\Expect{X_{\lambda_1}X_{\lambda_2}}^6.
\]
\end{lem}

\begin{proof}
The pairing $p$ defines a graph 
on the set of its pairs $\{\Lambda_1,\dots \Lambda_{n+k}\}$, 
where $\Lambda_i$ and $\Lambda_j$ are connected 
if and only if $\Lambda_i\cap\Lambda_j\neq\emptyset$. 
The maximal degree of this graph is two, 
hence it is a disjoint union of paths and cycles (and isolated points). 
Each path or cycle can be coloured with three colours (red, green and blue)
so that no neighbouring vertexes are of the same colour and the numbers of vertexes of different colours differ by at most one. 
Finally colours are interchangeable within each cycle or path, 
so by careful selection of excess colours for each cycle or path
we can ensure that the numbers of vertexes of different colours 
in the entire graph also differ by at most one.

The colors give us the partition. 
By definition all the pairs in each $p_i$ are disjoint,
and by construction 
$\lfloor \frac{n+k}{3}\rfloor\leq|p_i|\leq \lceil \frac{n+k}{3}\rceil$. 
The bound follows by the inequality between geometric and arithmetic means.
\end{proof}

%
\section{Anomalous scaling for Metropolis-Hastings algorithms in random environment}\label{sec:anomalous_scaling}

In this section we develop a framework 
for proving anomalous scaling results for Metropolis-Hastings algorithm for product targets which depend on random environments. The aim is to establish sufficient conditions under which the algorithm will exhibit anomalous scaling behaviour
for almost all realisations of the random environment. 
Sections~\ref{sec:RWM}~and~\ref{sec:MALA} will
then use this framework in the contexts of RWM and MALA algorithms to produce proofs of Theorems~\ref{thm:CLT_RWMvanilla}~and~\ref{thm:CLT_MALAvanilla}.

For the sake of definiteness and computational convenience,
we denote the random environment by $B$,
and suppose this to be determined by
a stationary continuous Gaussian process.
A (non-normalised) random marginal target density is then
produced by a map
$\upi\colon\Reals\times\Omega\to [0,\infty)$,
required to deliver
$\int_{\Reals}\upi(x\,|B){\d}x<\infty$ 
for almost all realisations of 
the random environment $B$. 
The normalization of the random marginal target density
is denoted by
$\pi\colon\Reals\times\Omega\to \Reals$,
so that
\[
 \int_\Reals \pi(x\,|B) {\d}{x}
 \quad=\quad
 \int_\Reals\left(\frac{
 \upi(x\,|B)
 }{
 \int_\Reals \upi(u\,|B){\d}{u}
 }\right) {\d}{x}
 \quad=\quad1\,.
\]
Finally,
let
$\rho_n(x,y\,|B)=\log\left(\frac{\pi(y|B)q_n(y,x\,|B)}{\pi(x\,|B)q_n(x,y\,|B)}\right)=\log\left(\frac{\upi(y|B)q_n(y,x\,|B)}{\upi(x\,|B)q_n(x,y\,|B)}\right)$ 
denote
the logarithm of the acceptance ratio of the Metropolis-Hastings algorithm with 
marginal target density $\pi(\cdot\,|B)$ 
(equivalently $\upi(\cdot\,|B)$) 
and proposal density 
$q_n\colon\Reals^2\times\Omega\to\Reals$.
Note that \(q_n(x,y\,|B)\) may
also depend on the random environment.

In light of Section~\ref{sec:CLTgeneral}, particularly the Central Limit Theorem~\ref{thm:unconditional-clt}, the crucial step is to identify the decay rate of the second moments of the log acceptance ratio $\Expectp{\pi(\cdot\,|B), q_n(\cdot\,|B)}{\rho^2_n\,|B}$ or equivalently the decay rate of functionals (differing only by a normalising constant that does not depend on $n$)
\begin{align}\label{eq:I_n}
\mathcal{I}_n(B)\quad &=\quad
\iint_{\Reals^2}\rho_n^2(x,y\,|B)
\,\upi(x\,|B)
\;q_n(x,y\,|B){\d}x{\d}y\nonumber\\
&=\quad	\iint_{\Reals^2}\rho_n^2(x,x+\sigma_nz\,|B)\,
\upi(x\,|B)\
\;
\sigma_nq_n(x,x+\sigma_nz\,|B){\d}x{\d}z\,,
\end{align} 
for some positive sequence 
\(\sigma_1 > \sigma_2 > \sigma_3 > \ldots > 0\).

Throughout the remainder of the paper, 
for two sequences $a_1, a_2, \ldots$ 
and $b_1, b_2, \ldots$ of positive real numbers,
the notation $a_n\precsim b_n$ indicates 
that there is a positive constant $C>0$ 
such that $a_n\leq Cb_n$ holds for all $n$.

We consider the situation in which there is a product target with
marginal target density depending on a random environment 
and a product Metropolis-Hasting proposal. 
In this section we consider the implications for optimal scaling
if the following framework of assumptions
is valid.
\begin{AssumptionFramework}\label{theassumption}
Let $\nu_1,\nu_2$ be probability density functions on $\Reals$ with all polynomial moments finite.  
Fix positive constants $\beta ,\gamma$ and $\ell$,
and choose a positive integer $m$
such that $m>3+\frac{144\beta}{\min(24\gamma,1)}$.
Finally, set $\nu=\nu_1\times\nu_2$ to be a joint density function,
and set $\sigma_n=\ell n^{-\frac{1}{2\beta}}$.
The sequence of assumptions (depending implicitly
on \(\beta , \gamma, \ell, m\)) are as follows:

\begin{enumerate}[label=\Alph*), ref=\ref{theassumption}.\Alph*]
\item \label{assumption:NormalPotentialFluctuations}
\textbf{Mixed Gaussian perturbation of log marginal target density:}

For every real $x$, the (un-normalized)
marginal target density is given by
\[
\upi(x\,|B)\quad=\quad\exp(K(x\,|B))\nu_1(x)\,,
\]
where $(K(x\,|B):x\in\Reals)$
is a centred Gaussian process such that $K(x\,|B)$
has variance $k(x)$. Furthermore,
we suppose \(K(x,B)\) has a particular
unconditional exponential moment that is finite:
\[
\Expect{\int_{\Reals}\exp(2m^2 K(x,B))\nu_1(x){\d}x}
\quad=\quad
 \int_{\Reals}\exp(2m^2k(x))\nu_1(x){\d}x
 \quad<\quad\infty
\]
(with \(m\) chosen as above). Particularly, this moment condition implies $\upi(x\,|B)$ is indeed a target density for almost every realisation of the random environment $B$.

\item\label{assumption:AsymptoticProposal}
\textbf{Asymptotic behaviour of perturbation of marginal proposal:}

For every real $x,z$ and positive integer $n$,
the marginal proposal density \(q_n\) satisfies
\[
\sigma_nq_n(x,x+\sigma_n z\,|B)\quad=\quad L_n(x,z\,|B)\nu_2(z)\,,
\]
where the random variable $L_n(x,z\,|B)$ is controlled by
\[
\iint_{\Reals^2}
\Expect{\left|L_n(x,z\,|B)-1\right|^{4m}}\nu(x,z){\d}x{\d}z
\quad\precsim\quad
\sigma_n^{4m\gamma}\,.
\]

\item \label{assumption:ApproximateLMHRNormality}
\textbf{Approximate Normality of log Metropolis Hastings ratio (LMHR):}

For every real $x,z$ and positive integer $n$
\[
\rho_n(x,x+\sigma_nz\,|B)\quad=\quad M_n(x,z\,|B)\quad+\quad\Delta_n(x,z\,|B)\,,
\]
where (for each \(n\)) the random process $(M_n(x,z\,|B):x,z\in\Reals)$ 
is a centred Gaussian process such that processes $K$ and $M_n$ are also jointly Gaussian. Furthermore,
\(M_n(x,z\,|B)\) has variance $h(x,z)~\sigma_n^{2\beta }$, 
for some function $h$ exhibiting at most polynomial growth, 
and $\Delta_n(x,z\,|B)$ is a random variable satisfying
 \[
 \iint_{\Reals^2}\Expect{\left|\Delta_n(x,z\,|B)\right|^{8m}}\nu(x,z){\d}x{\d}z
 \quad\precsim\quad
 \sigma_n^{8m\beta +8m\gamma}\,.
 \]

\item\label{assumption:AsymptoticWeakDependence}
\textbf{Asymptotic Weak Dependence:}
There exist sets $\auxset_n\subset\Reals^4$ taking
up increasingly larger parts of the space,
specifically $\int_{\auxset_n^c}\nu(x_1,z_1)\nu(x_2,z_2){\d}x_1 {\d}z_1 {\d}x_2 {\d}z_2\precsim\sigma_n^{1/2}$,
and fixed polynomials \(g_1, g_2\),
 such that for  $(x_1,z_1,x_2,z_2)\in\auxset_n$
\[
\left|\Expect{M_n(x_1,z_1\,|B)M_n(x_2,z_2\,|B)}\right|
\quad\leq\quad
 g_1(x_1,z_1,x_2,z_2)\times\sigma_n^{2\beta+\gamma}\,
\]
while, for all real $x_1,z_1,x_2$,
\[
\left|\Expect{M_n(x_1,z_1\,|B)K(x_2\,|B)}\right|
\quad\leq\quad
 g_2(x_1,z_1,x_2)\times\sigma_n^{\beta+\gamma}\,.
\]
\end{enumerate}
\end{AssumptionFramework}

Sections~\ref{sec:RWM} and \ref{sec:MALA} respectively give concrete examples of anomalous RWM and MALA algorithms in random environment that can be cast in terms of the above framework, that is they satisfy Assumptions~\ref{assumption:NormalPotentialFluctuations}-\ref{assumption:AsymptoticWeakDependence}.

Assumptions \ref{assumption:NormalPotentialFluctuations}-\ref{assumption:AsymptoticWeakDependence}
allow the
approximation of functionals $\mathcal{I}_n(B)$ by 
progressively simpler functionals. Initially,
consider
\begin{equation}\label{eq:tildeI_n}
\tilde{\mathcal{I}}_n(B)
\quad=\quad
\iint_{\Reals^2}\rho_n^2(x,x+\sigma_nz\,|B)\exp(K(x\,|B))\nu(x,z){\d}x{\d}z\,.
\end{equation}
We prove a quantitative result which will imply almost sure decay at the same speed as $\mathcal{I}_n(B)$.
\begin{lem}\label{lem:ItotildeI}
Let Assumptions
\ref{assumption:NormalPotentialFluctuations},
\ref{assumption:AsymptoticProposal} 
and \ref{assumption:ApproximateLMHRNormality}
be satisfied. Then 
\[
\Expect{\left|\mathcal{I}_n(B)-\tilde{\mathcal{I}}_n(B)\right|^{m}}
\quad\precsim\quad
\sigma_n^{2m\beta+m\gamma}\,.
\]
\end{lem}

\begin{proof}
Writing \(\nu(x,z)=\nu_1(x)\nu_2(z)\) for convenience,
the expectation can be rewritten using Assumptions \ref{assumption:NormalPotentialFluctuations} and \ref{assumption:AsymptoticProposal} of Framework \ref{theassumption} and then bounded by a combination of Jensen's inequality and double usage of
Cauchy-Schwarz inequality (all with respect to $\nu(x,z)\d x\d z\d\mathbb{P}$) to give
\begin{multline*}
\Expect{\left|\mathcal{I}_n(B)-\tilde{\mathcal{I}}_n(B)\right|^{m}}
\quad\leq\quad
\\
\Expect{\iint_{\Reals^2}\rho_n^{2m}(x,x+\sigma_nz\,|B)\exp(mK(x\,|B))\left|L_n(x,z\,|B)-1\right|^{m}\nu(x,z){\d}x{\d}z}
\\
\quad\leq\quad
\left(\iint_{\Reals^2}\Expect{\exp(2mK(x\,|B))}\nu(x,z){\d}x{\d}z\right)^{1/2}
\times\left(\iint_{\Reals^2}\Expect{\left(L_n(x,z\,|B)-1\right)^{4m}}\nu(x,z){\d}x{\d}z\right)^{1/4}
\\
\times
\left(\iint_{\Reals^2}\Expect{\rho_n^{8m}(x,x+\sigma_nz\,|B)}\nu(x,z){\d}x{\d}z\right)^{1/4}\,.
\end{multline*}
The first factor is bounded 
by application of Assumption \ref{assumption:NormalPotentialFluctuations} 
followed by marginalization over \(z\).
The second factor decays at least as $\sigma_n^{m\gamma}$ by Assumption \ref{assumption:AsymptoticProposal}.

The proof will be concluded once we establish the last factor decays at least as $\sigma_n^{2m\beta}$. Indeed
\begin{multline*}
\left(\iint_{\Reals^2}\Expect{\rho_n^{8m}(x,x+\sigma_nz\,|B)}\nu(x,z){\d}x{\d}z\right)^{1/4}
\quad=\quad
\\
\left(\Expect{\iint_{\Reals^2}
\left(M_n(x,z\,|B)+\Delta_n(x,z\,|B)\right)^{8m}\nu(x,z){\d}x{\d}z
}\right)^{1/4}
\\
\quad\leq\quad
2^{2m-1}\cdot
\Bigg(
\left(\Expect{\iint_{\Reals^2}
M_n(x,z\,|B)^{8m}\nu(x,z){\d}x{\d}z
}\right)^{1/4}
\quad+\quad
\left(\Expect{\iint_{\Reals^2}
\Delta_n(x,z\,|B)^{8m}\nu(x,z){\d}x{\d}z
}\right)^{1/4}
\Bigg)
\\
\quad\leq\quad
2^{2m-1}
\cdot
\Bigg(
\left(\Expect{N(0,1)^{8m}}\iint_{\Reals^2}
h(x,z)^{4m}\nu(x,z){\d}x{\d}z
\right)^{1/4}
\sigma_n^{2m\beta}
\quad+\quad
C\sigma_n^{2m\beta+2m\gamma}
\Bigg)
\quad\precsim\quad
\sigma_n^{2m\beta}\,,
\end{multline*}
where $C$ is some positive constant.
The identity holds by Assumption \ref{assumption:ApproximateLMHRNormality}. The first inequality 
follows from the elementary bound $(a+b)^{2m}\leq 2^{2m-1}(a^{2m}+b^{2m})$ 
together with application of a triangle inequality in $L^4(\nu\times\mathbb{P})$ norm. 
The remainder follows from the Fubini-Tonelli theorem and the details of \ref{assumption:ApproximateLMHRNormality}.
\end{proof}

The functionals $\tilde{I}_n(B)$ can now be simplified further by approximating 
\(\rho_n(x,x+\sigma_nz\;|B)\approx M_n(x,z\,|B)\),
and controlling the approximation using
Assumptions
\ref{assumption:NormalPotentialFluctuations} and\ref{assumption:ApproximateLMHRNormality}. Let
\begin{equation}\label{eq:hatI_n}
\hat{\mathcal{I}}_n(B)=\iint_{\Reals^2}M^2_n(x,z\,|B)\exp(K(x\,|B))\nu(x,z){\d}x{\d}z\,.
\end{equation}
Again the functionals $\tilde{\mathcal{I}}_n(B)$ and $\hat{\mathcal{I}}_n(B)$ can be shown to be close to each other.

\begin{lem}\label{lem:tildeI_to_hatI}
Suppose
Assumptions
\ref{assumption:NormalPotentialFluctuations} and\ref{assumption:ApproximateLMHRNormality}
are satisfied. Then
\[
\Expect{\left|\tilde{\mathcal{I}}_n(B)-\hat{\mathcal{I}}_n(B)\right|^m}\quad\precsim\quad \sigma_n^{2m\beta+m\gamma}\,.
\]
\end{lem}

\begin{proof}
Arguing as in Lemma \ref{lem:ItotildeI}.
Jensen's inequality yields
\[
\left|\tilde{\mathcal{I}}_n(B)-\hat{\mathcal{I}}_n(B)\right|^{m}
\quad\leq\quad
 \iint_{\Reals^2}\left|\rho^2_n(x,x+\sigma_nz\,|B)-M^2_n(x,z\,|B)\right|^{m}\exp(mK(x\,|B))\nu(x,z){\d}x{\d}z\,.
\]
Recall that by Assumption~\ref{assumption:ApproximateLMHRNormality}
\[
\rho^2_n(x,x+\sigma_nz\,|B)-M^2_n(x,z\,|B)
\quad=\quad
\Delta_n(x,z\,|B)\left(2M^2_n(x,z\,|B)+\Delta_n(x,z\,|B)\right).
\]
Exchanging the
expectation with the double integral using the Fubini-Tonelli theorem,
and then applying the Cauchy-Schwarz inequality twice over,
\begin{multline*}
\Expect{\left|\tilde{\mathcal{I}}_n(B)-\hat{\mathcal{I}}_n(B)\right|^{m}}
\quad\leq\quad
\left(\iint_{\Reals^2}\Expect{\exp\left(2mK(x\,|B)\right)}\nu(x,z){\d}x{\d}z)\right)^{1/2}
\\
\times\quad
\left(\iint_{\Reals^2}\Expect{\Delta_n(x,z\,|B)^{4m}}\nu(x,z){\d}x{\d}z\right)^{1/4} 
\\
\times\quad
\left(\iint_{\Reals^2}\Expect{\left(2M_n(x,z\,|B)+\Delta_n(x,z\,|B)\right)^{4m}}\nu(x,z){\d}x{\d}z\right)^{1/4}\,.
\end{multline*}
As in the proof of Lemma \ref{lem:ItotildeI},
Assumption \ref{assumption:NormalPotentialFluctuations}
implies the first factor is bounded and Assumption~\ref{assumption:ApproximateLMHRNormality} guarantees
second factor decays at least as $\sigma_n^{m\beta +m\gamma}$ and the third as $\sigma_n^{m\beta }$.
\end{proof}

The final step is to consider the functional
obtained from $\hat{\mathcal{I}}_n(B)$ 
by replacing $\left(M_n(x,z\,|B)\right)^2$ by its expectation (see Assumptions \ref{assumption:NormalPotentialFluctuations} and \ref{assumption:ApproximateLMHRNormality}):
\begin{multline}\label{eq:J_n}
\mathcal{J}_n(B)\quad=\quad \iint_{\Reals^2}\Expect{\left(M_n(x,z\,|B)\right)^2}
\exp\left(K(x\,|B)\right)\nu(x,z){\d}x{\d}z.
\\
\quad=\quad
\sigma_n^{2\beta}\cdot\iint_{\Reals^2}h(x,z)
\upi(x\,|B)\nu_2(z){\d}x{\d}z
\\
\quad=\quad
\sigma_n^{2\beta}\cdot\iint_{\Reals^2}h(x,z)
\exp(K(x\,|B))\nu_1(x)\nu_2(z){\d}x{\d}z\,.
\end{multline}
Note that the double integral is almost surely finite:
this follows from the polynomial growth of \(h(x,z)\)
(Assumption \ref{assumption:ApproximateLMHRNormality}), Cauchy-Schwarz inequality,
the fact that the densities \(\nu_1\) and \(\nu_2\) have finite polynomial moments
(stipulated in the Framework \ref{theassumption}),
and the fact that \(\exp(k(x))\) is integrable with respect to \(\nu_1\)
(Assumption~\ref{assumption:NormalPotentialFluctuations}).

Again we need to establish that the functionals $\hat{\mathcal{I}}_n(B)$ and $\mathcal{J}_n(B)$ are close.
\begin{lem}\label{lem:hatI_to_J}
Let Assumptions 
\ref{assumption:NormalPotentialFluctuations},
\ref{assumption:ApproximateLMHRNormality}
and
\ref{assumption:AsymptoticWeakDependence}
be satisfied. Then
\[
\Expect{\left|\hat{\mathcal{I}}_n(B)-\mathcal{J}_n(B)\right|^{m}}
\quad\precsim\quad
\sigma_n^{2m\beta + m\min(\frac{\gamma}{2},\frac{1}{48})}\,.
\]
\end{lem}

\begin{proof}
It suffices to bound 
$\Expect{\left|\hat{\mathcal{I}}_n(B)-\mathcal{J}_n(B)\right|^{2m}}$,
since by Jensen's inequality
\[\Expect{\left|\hat{\mathcal{I}}_n(B)-\mathcal{J}_n(B)\right|^{m}}
\quad\leq\quad
\Expect{\left|\hat{\mathcal{I}}_n(B)-\mathcal{J}_n(B)\right|^{2m}}^{1/2}\,.\]

Formulae \eqref{eq:hatI_n} 
for \(\hat{\mathcal{I}}_n(B)\)
and \eqref{eq:J_n} 
for \(\mathcal{J}_n(B)\) together imply
\[
\hat{\mathcal{I}}_n(B)-\mathcal{J}_n(B)
\quad=\quad
\iint_{\Reals^2}\left(M_n(x,z\,|B)^2-\Expect{M_n(x,z\,|B)^2}\right) \exp\left(K(x\,|B)\right)\nu(x,z){\d}x{\d}z
\]
and consequently
\begin{multline*}
\left(\hat{\mathcal{I}}_n(B)-\mathcal{J}_n(B)\right)^{2m}
\quad=\quad
\\
\idotsint_{\Reals^{4m}}
\left(\prod_{i=1}^{2m}\left(M_n(x_i,z_i\,|B)^2-\Expect{M_n(x_i,z_i\,|B)^2}\right) 
\right)
\times
\exp\left(\sum_{i=1}^{2m}\left(K(x_i|B)\right)\right)
\times
\prod_{i=1}^{2m}\left(\nu(x_i,z_i){\d}x_i{\d}z_i\right)
\\
\quad=\quad
\idotsint_{\Reals^{4m}}
\left(\prod_{i=1}^{2m}\left(M^2_i-R_{ii}\right)\right)
\times \exp(\bar{K})
\times 
\left(\prod_{i=1}^{2m}\nu(x_i,z_i){\d}x_i{\d}z_i\right)
\,,
\end{multline*}
where we abbreviate notation by writing
$M_i=M_n(x_i,z_i\,|B)$, 
$R_{ii}=\Expect{M_i^2}$ 
and $\bar{K}=\sum_{i=1}^{2m}\left(K(x_i|B)\right)$.
Note that the various
\(M_i\) and \(K(x_j|B)\) are \emph{not}
necessarily independent, and typically will not be so.

Using the Fubini-Tonelli theorem 
to exchange the expectation
in \(\Expect{\left(\hat{\mathcal{I}}_n(B)-\mathcal{J}_n(B)\right)^{2m}}\)
with the implicit
multiple integrals,
we now obtain
\begin{equation}\label{eq:hatItoJ:MotherIntegral}
\Expect{\left(\hat{\mathcal{I}}_n(B)-\mathcal{J}_n(B)\right)^{2m}}
\quad=\quad	
\idotsint_{\Reals^{4m}}\Expect{\exp(\bar{K}) \prod_{i=1}^{2m}\left(M^2_i-R_{ii}\right) }
\prod_{i=1}^{2m}\left(\nu(x_i,z_i){\d}x_i{\d}z_i\right)
\,.
\end{equation}

By Lemma~\ref{lem:Isserlis_exponential} the expectation $\Expect{\exp(\bar{K}) \prod_{i=1}^{2m}\left(M^2_i-R_{ii}\right)}$ equals
\[
\exp\left(\half\Expect{\bar{K}^2}\right)		\sum_{\substack{A_1\cup A_2\cup A_3=\{1,2,\dots,2m\}	\\	A_1\cap A_2=A_1\cap A_3=A_2\cap A_3=\emptyset\\ |A_2|\text{ is even}}}
\sum_{p\in\mathcal{P}^*(2A_1\cup A_2)}
2^{|A_2|} \prod_{i\in A_2}\Expect{M_i\bar{K}}\times\prod_{i\in A_3}\Expect{M_i\bar{K}}^2	\times\prod_{\Lambda\in p}\Expect{M_{\lambda_1}M_{\lambda_2}}.
\]
Inserting the above into  \eqref{eq:hatItoJ:MotherIntegral}, we obtain
\begin{multline}\label{eq:hatItoJ:aux1}
\Expect{\left(\hat{\mathcal{I}}_n(B)-\mathcal{J}_n(B)\right)^{2m}}
\quad=\quad
\sum_{\substack{A_1\cup A_2\cup A_3=\{1,2,\dots,2m\}	\\	A_1\cap A_2=A_1\cap A_3=A_2\cap A_3=\emptyset\\ |A_2|\text{ is even}}}
\sum_{p\in\mathcal{P}^*(2A_1\cup A_2)}		2^{|A_2|}
\\
\times\quad
\idotsint_{\Reals^{4m}} \prod_{i\in A_2}\Expect{M_i\bar{K}}\cdot\prod_{i\in A_3}\Expect{M_i\bar{K}}^2	\cdot\prod_{\Lambda\in p}\Expect{M_{\lambda_1}M_{\lambda_2}}
\cdot \exp\left(\Expect{\tfrac12\bar{K}^2}\right)	\cdot\prod_{i=1}^{2m}\left(\nu(x_i,z_i){\d}x_i{\d}z_i\right)
\,.
\end{multline}

Now focus attention on a typical summand 
in the above sum. 
This corresponds to fixing
a partition $A_1,A_2,A_3$ with prescribed properties and a proper pairing $p$ of $2A_1\cup A_2$. 
Applying the
Cauchy-Schwarz inequality 
with respect to the measure 
$\prod_{i=1}^{2m}\left(\nu(x_i,z_i){\d}x_i{\d}z_i\right)$,
\begin{multline}\label{eq:hatItoJ:aux2}
\idotsint_{\Reals^{4m}} 
\prod_{i\in A_2}\Expect{M_i\bar{K}}
\cdot
\prod_{i\in A_3}\Expect{M_i\bar{K}}^2
\cdot
\prod_{\Lambda\in p}
\Expect{M_{\lambda_1}M_{\lambda_2}}
\cdot 
\exp\left(\half\Expect{\bar{K}^2}\right)
\cdot
\prod_{i=1}^{2m}\left(\nu(x_i,z_i){\d}x_i{\d}z_i\right)
\\
\quad\leq\quad
\left(
\idotsint_{\Reals^{4m}} 
\prod_{i\in A_2}\Expect{M_i\bar{K}}^2
\cdot
\prod_{i\in A_3}\Expect{M_i\bar{K}}^4
\cdot 
\exp\left(\Expect{\bar{K}^2}\right)
\cdot
\prod_{i=1}^{2m}\left(\nu(x_i,z_i){\d}x_i{\d}z_i\right)
\right)^{1/2}
\\
\times\quad
\left(
\idotsint_{\Reals^{4m}} 
\prod_{\Lambda\in p}\Expect{M_{\lambda_1}M_{\lambda_2}}^2
\cdot
\prod_{i=1}^{2m}\left(\nu(x_i,z_i){\d}x_i{\d}z_i\right)
\right)^{1/2}
\,.
\end{multline}

Consider the first factor.
We can bound
each $|\Expect{M_i\bar{K}}|$ by
$|\Expect{M_i\bar{K}}|\leq g(z_i,x_1,\dots,x_n)\sigma_n^{\beta+\gamma}$,
using a polynomial 
$g(z_i,x_1,\dots,x_n)=\sum_{j=1}^{2m} g_2(x_i,z_i,x_j)$,
generated from the second point of
Assumption \ref{assumption:AsymptoticWeakDependence}.

By the Cauchy-Schwarz inequality,
\[
 \left(\frac{1}{2m}\right)^2\bar{K}^2
 \quad=\quad
 \left(\frac{1}{2m}\sum_{i=1}^{2m}K(x_i|B)\right)^2
 \quad\leq\quad
 \frac{1}{2m}\sum_{i=1}^{2m}K(x_i|B)^2
\]
and so
\begin{equation*}
 \prod_{i\in A_2}\Expect{M_i\bar{K}}^2\cdot\prod_{i\in A_3}\Expect{M_i\bar{K}}^4\cdot \exp\left(\Expect{\bar{K}^2}\right)
 \quad\leq\quad
 \prod_{i\in A_2}\Expect{M_i\bar{K}}^2\cdot\prod_{i\in A_3}\Expect{M_i\bar{K}}^4\cdot \exp\left(2m\sum_{j=1}^{2m}\Expect{K(x_j|B)^2}\right)
 \,.
\end{equation*}

Hence, 
Assumptions \ref{assumption:NormalPotentialFluctuations} and \ref{assumption:AsymptoticWeakDependence} yield
\begin{multline}\label{eq:hatItoJ:aux3}
\prod_{i\in A_2}\Expect{M_i\bar{K}}^2\cdot\prod_{i\in A_3}\Expect{M_i\bar{K}}^4\cdot \exp\left(\Expect{\bar{K}^2}\right)
\\
\leq\quad
\exp\left(2m\sum_{j=1}^{2m}k(x_j)\right)
\prod_{i\in A_2}g(z_i,x_1,\dots,x_n)^2\prod_{i\in A_3}g(z_i,x_1,\dots,x_n)^4
~\times~\sigma_n^{(2|A_2|+4|A_3|)(\beta +\gamma)}\,.
\end{multline}
Application of the Cauchy-Schwarz inequality, and
the exponential integrability of 
\(2 m^2 k(x)\) (with respect to \(\nu_1(x){\d}x\))
assured by Assumption \ref{assumption:NormalPotentialFluctuations},
shows that this is
integrable with respect to the probability measure $\prod_{i=1}^{2m}\left(\nu(x_i,z_i){\d}x_i{\d}z_i\right)$. 
Consequently we obtain
\begin{equation}\label{eq:hatItoJ:aux4} 
\left(\idotsint_{\Reals^{4m}} \prod_{i\in A_2}\Expect{M_i\bar{K}}^2\cdot\prod_{i\in A_3}\Expect{M_i\bar{K}}^4\cdot \exp\left(\Expect{\bar{K}^2}\right)	\cdot\prod_{i=1}^{2m}\nu(x_i,z_i){\d}x_i{\d}z_i\right)^{1/2}
\quad\precsim\quad
\sigma_n^{(|A_2|+2|A_3|)(\beta  +\gamma)}\,.
\end{equation}

Consider now the second factor in \eqref{eq:hatItoJ:aux2}.
As $p$ is a proper pairing, Lemma \ref{lem:Isserlis_product_bound} asserts there is a
partition of $p$
into three sets of pairs $p_1,p_2,p_3$ 
of size at least $\lfloor|A_1|/3+|A_2|/6\rfloor$
so that all pairs within each $p_i$ are disjoint 
and moreover
\[
\prod_{\Lambda\in p}	\Expect{M_{\lambda_1}M_{\lambda_2}}^2
\leq 
\frac{1}{3}	\prod_{\Lambda\in p_1}	\Expect{M_{\lambda_1}M_{\lambda_2}}^6
+\frac{1}{3}	\prod_{\Lambda\in p_2}	\Expect{M_{\lambda_1}M_{\lambda_2}}^6
+\frac{1}{3}	\prod_{\Lambda\in p_3}	\Expect{M_{\lambda_1}M_{\lambda_2}}^6
\,.
\]

This allows us to split the integral over $\Reals^{4m}$ into a product of integrals over $\Reals^4$
\begin{multline}\label{eq:hatItoJ:aux5}
\idotsint_{\Reals^{4m}} \prod_{\Lambda\in p}\Expect{M_{\lambda_1}M_{\lambda_2}}^2\cdot\prod_{i=1}^{2m}\nu(x_i,z_i){\d}x_i{\d}z_i
\\
\leq\quad
 \frac{1}{3}\sum_{j=1}^3
\idotsint_{\Reals^{4m}}
\prod_{\Lambda\in p_j}\Expect{M_{\lambda_1}M_{\lambda_2}}^6
\prod_{i=1}^{2m}\nu(x_i,z_i){\d}x_i{\d}z_i
\\
=\quad
 \frac{1}{3}\sum_{j=1}^3
\prod_{\Lambda\in p_j} 
\int_{\Reals^{4}}\Expect{M_{\lambda_1}M_{\lambda_2}}^6\nu(x_{\lambda_1},z_{\lambda_1})\nu(x_{\lambda_2},z_{\lambda_2}){\d}x_{\lambda_1}{\d}z_{\lambda_1}{\d}x_{\lambda_2}{\d}z_{\lambda_2}.
\end{multline}
The last equality holds because pairs within each $p_j$ are by construction disjoint which imposes a product structure on the high-dimensional integral.

For each of the factors of \eqref{eq:hatItoJ:aux5}, the first bound of
Assumption \ref{assumption:AsymptoticWeakDependence} 
yields
\[
\int_{\auxset_n}\Expect{M_{\lambda_1}M_{\lambda_2}}^6\nu(x_{\lambda_1},z_{\lambda_1})
\nu(x_{\lambda_2},z_{\lambda_2}){\d}x_{\lambda_1}{\d}z_{\lambda_1}{\d}x_{\lambda_2}{\d}z_{\lambda_2}
\quad\precsim\quad
\sigma_n^{12\beta +6\gamma }\,.
\]
The Cauchy-Schwartz inequality,
together with
Assumptions \ref{assumption:ApproximateLMHRNormality} and \ref{assumption:AsymptoticWeakDependence} control the integral off the set $\auxset_n$,
\begin{multline*}
\int_{\auxset_n^c}\Expect{M_{\lambda_1}M_{\lambda_2}}^6\nu(x_{\lambda_1},z_{\lambda_1})\nu(x_{\lambda_2},z_{\lambda_2}){\d}x_{\lambda_1}{\d}z_{\lambda_1}{\d}x_{\lambda_2}{\d}z_{\lambda_2}\\
\leq\quad
\int_{\Reals^4}\Expect{M^2_{\lambda_1}}^3\Expect{M^2_{\lambda_2}}^31_{\auxset^c_n}(x_{\lambda_1},z_{\lambda_1},x_{\lambda_2},z_{\lambda_2})\nu(x_{\lambda_1},z_{\lambda_1})\nu(x_{\lambda_2},z_{\lambda_2}){\d}x_{\lambda_1}{\d}z_{\lambda_1}{\d}x_{\lambda_2}{\d}z_{\lambda_2}
\\
\quad\precsim\quad
\sigma_n^{12\beta }
\left(\int_{\auxset_n^c}\nu(x_{\lambda_1},z_{\lambda_1})\nu(x_{\lambda_2},z_{\lambda_2}){\d}x_{\lambda_1}{\d}z_{\lambda_1}{\d}x_{\lambda_2}{\d}z_{\lambda_2}\right)^{1/2}
\quad\precsim\quad
\sigma_n^{12\beta +\frac{1}{4}}\,.
\end{multline*}
Together the above bounds give 
\begin{equation*}
\int_{\Reals^4}\Expect{M_{\lambda_1}M_{\lambda_2}}^6\nu(x_{\lambda_1},z_{\lambda_1})\nu(x_{\lambda_2},z_{\lambda_2}){\d}x_{\lambda_1}{\d}z_{\lambda_1}{\d}x_{\lambda_2}{\d}z_{\lambda_2}
\quad\precsim\quad
\sigma_n^{12\beta +6\min(\gamma ,\tfrac{1}{24})}.
\end{equation*}

Since Lemma~\ref{lem:Isserlis_product_bound}
asserts that each set of pairs $p_j$ contains 
at least $\lfloor|A_1|/3+|A_2|/6\rfloor$ pairs,
the above together with \eqref{eq:hatItoJ:aux5} gives 
\begin{equation}\label{eq:hatItoJ:aux6}
\left(\idotsint_{\Reals^{4m}} \prod_{\Lambda\in p}\Expect{M_{\lambda_1}M_{\lambda_2}}^2\cdot\prod_{i=1}^{2m}\nu(x_i,z_i){\d}x_i{\d}z_i\right)^{1/2}
\quad\precsim\quad 
\sigma_n^{\left(6\beta +3\min(\gamma ,\tfrac{1}{24})\right)\left\lfloor|A_1|/3+|A_2|/6\right\rfloor}\,.
\end{equation}

Combining \eqref{eq:hatItoJ:aux6} with \eqref{eq:hatItoJ:aux2}
and \eqref{eq:hatItoJ:aux4},
we obtain
the following bound for each fixed partition:
\begin{align*}
&\left|\idotsint_{\Reals^{4m}} \prod_{i\in A_2}\Expect{M_i\bar{K}}\cdot\prod_{i\in A_3}\Expect{M_i\bar{K}}^2	\cdot\prod_{\Lambda\in p}\Expect{M_{\lambda_1}M_{\lambda_2}}
\cdot \exp\left(\Expect{\half\bar{K}^2}\right)	\cdot\prod_{i=1}^{2m}\nu(x_i,z_i){\d}x_i{\d}z_i\right|.
\\
&\precsim\quad
\sigma_n^{(|A_2|+2|A_3|)\left(\beta  +\gamma \right)}
\quad\times\quad
\sigma_n^{\lfloor|A_1|/3+|A_2|/6\rfloor\times (6\beta  +3\min(\gamma ,\tfrac{1}{24}))}.
\\
&\precsim \quad
\sigma_n^{2\beta (|A_1|+|A_2|+|A_3|)}
\quad\times\quad
\sigma_n^{2\gamma |A_3|~+~(\gamma +\tfrac{1}{2}\min\left(\gamma ,\frac{1}{24})\right)|A_2|~+~\min(\gamma ,\tfrac{1}{24})|A_1|
\quad-\quad 6\beta ~-~3\min(\gamma ,\tfrac{1}{24})}
\\
&\leq\quad
\sigma_n^{4m\beta }\times \sigma_n^{2m\min(\gamma ,\tfrac{1}{24})}\times \sigma_n^{-6\beta  -3\min(\gamma ,\tfrac{1}{24})}
\quad\leq\quad
\sigma_n^{4m\beta }\times \sigma_n^{m\min(\gamma ,\tfrac{1}{24})}\,.
\end{align*}
The argument for this uses $|A_1|+|A_2|+|A_3|=2m$
together with crude bounds to reduce
coefficients of remaining \(A_1\), \(A_2\), \(A_3\) 
to \(\min(\gamma ,\tfrac{1}{24})\) 
and then employs $m>3+\frac{144\beta}{\min(24\gamma,1)}$ as stipulated in
the Framework \ref{theassumption}.

The above bound no longer depends on the choice of partition $A_1,A_2,A_3$ and so can be used in \eqref{eq:hatItoJ:aux1} to achieve a bound of
\[
 \left(\hat{\mathcal{I}}_n(B)-\mathcal{J}_n(B)\right)^{2m}
 \quad\leq\quad
 \text{constant} \times \sigma_n^{4m\beta +m\min(\gamma ,\tfrac{1}{24})}\,,
\]
where the constant depends on \(m\) but not on \(n\).
As noted at the start of the proof, this establishes the lemma.
\end{proof}

We now require the following application of the Borel-Cantelli lemma.
\begin{proposition}\label{prop:BorelCantelli}
Let $U_1, U_2, \ldots$ and $\tilde{U}_1, \tilde{U}_2, \ldots$ 
be sequences of random variables, 
let $\delta_1, \delta_2, \ldots$ be a positive sequence converging to zero,
and suppose $\kappa$ is a positive constant. 
Assume there is a constant $C>0$ and an integer $m>\frac{1}{\kappa}$ 
such that the inequality 
$\Expect{|U_n-\tilde{U}_n|^m}\leq C\delta_n^m n^{-m \kappa}$ is satisfied for every $n$. 
Then $\Prob{\delta_n^{-1}(U_n-\tilde{U}_n)\xrightarrow{n\to\infty}0}=1$.
\end{proposition}

\begin{proof}
Take an arbitrary $\epsilon>0$. By Markov's inequality
\[
\Prob{\left|U_n-\tilde{U}_n\right|>\epsilon \delta_n}
\quad\leq\quad
\frac{1}{\epsilon^m \delta^{m}_n}\Expect{\left|U_n-\tilde{U}_n\right|^m}
\quad\leq\quad
\frac{C}{\epsilon^m} n^{-m \kappa}\,.
\]
Summing over \(n=1, 2,\ldots\), and noting that $m\kappa>1$,
\[
\sum_{n=1}^\infty\Prob{\left|U_n-\tilde{U}_n\right|>\epsilon \delta_n}
\quad\leq\quad
\frac{C}{\epsilon^m}\sum_{n=1}^\infty n^{-m\kappa}
\quad<\quad\infty\,.
\]
It now follows from the Borel-Cantelli lemma that
$\Prob{\left|U_n-\tilde{U}_n\right|>\epsilon \delta_n\quad\text{i.o.}}=0$. Since $\epsilon>0$ was arbitrary, the result follows.
\end{proof}

This enables us to show that the functionals $\mathcal{I}_n(B),~\tilde{\mathcal{I}}_n(B),~\hat{\mathcal{I}}_n(B)$ and $\mathcal{J}_n(B)$ indeed decay with the same speed almost surely (for almost all realisations of the random environment $B$) and thus identify the almost sure decay of $\mathcal{I}_n(B)$.

\begin{proposition}\label{prop:Indecay}
Let the assumptions
of Framework \ref{theassumption} be satisfied. Then (almost surely in the random environment $B$)
\[
\sigma_n^{-2\beta }\mathcal{I}_n(B)
\quad\xrightarrow{n\to\infty}\quad
\iint_{\Reals^2}h(x,z)\upi(x\,|B))\nu_2(z){\d}x{\d}z\,.
\]
\end{proposition}
So in this case $\mathcal{I}_n(B)$ almost surely decays as $\sigma_n^{2\beta }$.

\begin{proof}
Note that the Framework
\ref{theassumption} includes a stipulation
that \(\sigma_n=\ell n^{-\tfrac1{2\beta}}\), as well as a requirement 
that \(m>3+\frac{144\beta}{\min(24\gamma,1)}\).

Apply Proposition~\ref{prop:BorelCantelli} 
together with Lemma~\ref{lem:ItotildeI} 
in the case that $U_n=\mathcal{I}_n$, $\tilde{U}_n=\tilde{\mathcal{I}}_n$, 
$\delta_n=\sigma^{2\beta}_n$ and $\kappa=\tfrac{\gamma}{2\beta}$. 
Since \(m\kappa=m\frac{\gamma}{2\beta}\geq 3>1\),
it follows that the difference $\left|\mathcal{I}_n(B)-\tilde{\mathcal{I}}_n(B)\right|$ 
almost surely decays faster than $\sigma_n^{2\beta}$. 

Similarly, 
apply
Proposition~\ref{prop:BorelCantelli} 
together with Lemma~\ref{lem:tildeI_to_hatI} 
in the case that
$U_n=\tilde{\mathcal{I}}_n$, $\tilde{U}_n=\hat{\mathcal{I}}_n$, $\delta_n=\sigma^{2\beta}_n$ 
and $\kappa=\tfrac{\gamma}{2\beta}$.
Since again $m\kappa=m\frac{\gamma}{2\beta}\geq 3>1$,
it follows that the difference $\left|\tilde{\mathcal{I}}_n(B)-\hat{\mathcal{I}}_n(B)\right|$
almost surely decays faster than $\sigma_n^{2\beta}$. 

Finally, apply Proposition~\ref{prop:BorelCantelli} 
together with Lemma~\ref{lem:hatI_to_J} 
in the case that $U_n=\hat{\mathcal{I}}_n$, $\tilde{U}_n=\mathcal{J}_n$, 
$\delta_n=\sigma^{2\beta}_n$ and $\kappa=\frac{\min(24\gamma,1)}{96\beta }$.
Now 
\(m\kappa=m\frac{\min(24\gamma,1)}{96\beta }>\frac{144\beta}{\min(24\gamma,1)}\frac{\min(24\gamma,1)}{96\beta }=\frac{3}{2}>1\),
and so
the difference $\left|\hat{\mathcal{I}}_n(B)-\mathcal{J}_n(B)\right|$ 
almost surely decays faster than $\sigma_n^{2\beta}$. 

Consequently the difference $\left|I_n(B)-\mathcal{J}_n(B)\right|$ 
almost surely decays faster than $\sigma_n^{2\beta}$. 
But $\mathcal{J}_n(B)$ is calculated exactly in \eqref{eq:J_n},
and demonstrably almost surely decays exactly as $\sigma_n^{2\beta}$. 
Consequently the same must hold for $I_n(B)$ and so the 
proposition follows. 
\end{proof}

Since the random targets $\upi(\cdot\,|B))$ are almost surely integrable and independent of $n$
(Assumption \ref{assumption:NormalPotentialFluctuations}), the following corollary follows by normalisation.
\begin{cor}\label{cor:rho2decay}
Let the assumptions of
Framework \ref{theassumption} be satisfied. Then (almost surely 
in the random environment $B$)
\[
\sigma_n^{-2\beta }\Expectp{\pi(\cdot\,|B), q_n(\cdot\,|B)}{\rho^2_n\,|B}
\quad\xrightarrow{n\to\infty}\quad
\iint_{\Reals^2}h(x,z)\pi(x\,|B)\nu_2(z){\d}x{\d}z \,.
\]
\end{cor}
Thus $\Expectp{\pi(\cdot\,|B), q_n(\cdot\,|B)}{\rho^2_n\,|B}$ 
almost surely decays as $\sigma_n^{2\beta }$.

The final task is to
show that a Lindeberg-type condition holds almost surely.

\begin{lem}\label{lem:Lindeberg}
Let the assumptions of
Framework \ref{theassumption} be satisfied.
Then almost surely (for almost every realisation of the random environment $B$)
\[
\Expect{\Expectp{\pi(\cdot\,|B), q_n(\cdot\,|B)}{\rho^2_n1_{\rho^2_n>\sigma^\beta_n}\,|B}^m}
\quad\precsim\quad
\sigma_n^{3m\beta}.
\]
\end{lem}

\begin{proof}
By a combination of the Cauchy-Schwarz and Markov inequalities,
for almost every realisation of $B$,
\begin{align*}
\Expectp{\pi(\cdot\,|B), q_n(\cdot\,|B)}{\rho^2_n1_{\rho^2_n>\sigma^\beta_n}\,|B}
\quad&\leq\quad
\Expectp{\pi(\cdot\,|B), q_n(\cdot\,|B)}{\rho^4_n\,|B}^{1/2}
\quad\cdot\quad
\mathbb{P}_{\pi(\cdot\,|B), q_n(\cdot\,|B)}\left[\LMHR_n^2>\sigma^\beta_n\,|B\right]^{1/2}
\\
&\leq\quad
\sigma^{-\beta}_n~\Expectp{\pi(\cdot\,|B), q_n(\cdot\,|B)}{\rho^4_n\,|B}\,.
\end{align*}
Hence, by Jensen's inequality
\[
\Expect{\Expectp{\pi(\cdot\,|B), q_n(\cdot\,|B)}{\rho^2_n1_{\rho^2_n>\sigma^\beta_n}\,|B}^m}
\quad\leq\quad 
\sigma^{-ma}_n~\Expect{\Expectp{\pi(\cdot\,|B), q_n(\cdot\,|B)}{\rho^{4m}_n\,|B}}\,.
\]

As in the case of $\mathcal{I}_n(B)$ (see \eqref{eq:I_n}) the random functional $\Expectp{\pi(\cdot\,|B), q_n(\cdot\,|B)}{\rho^{4m}_n\,|B}$ differs from
\[
\iint_{\Reals^2}\rho^{4m}_n(x,x+\sigma_nz\,|B)
\upi(x\,|B)\,\sigma_nq_n(x,x+\sigma_nz\,|B){\d}x{\d}z
\]
just by a normalising constant.
Using Assumptions \ref{assumption:NormalPotentialFluctuations} and \ref{assumption:AsymptoticProposal}, the Fubini-Tonelli theorem,
and the Cauchy-Schwarz inequality twice over,
\begin{multline*}
\Expect{\Expectp{\pi(\cdot\,|B), q_n(\cdot\,|B)}{\rho^{4m}_n\,|B}}
\quad=\quad
\Expect{\iint_{\Reals^2}\rho^{4m}_n(x,x+\sigma_nz\,|B)\exp(K(x\,|B)L_n(x,z\,|B)\nu(x,z){\d}x{\d}z}\\
\quad\leq\quad
\Expect{\iint_{\Reals^2}\rho^{8m}_n(x,x+\sigma_nz\,|B)\nu(x,z){\d}x{\d}z}^{1/2}
\times\Expect{\iint_{\Reals^2}\exp(4K(x\,|B)\nu(x,z){\d}x{\d}z}^{1/4}\\
\times \quad\Expect{\iint_{\Reals^2}L^4_n(x,z\,|B)\nu(x,z){\d}x{\d}z}^{1/4}\,.
\end{multline*}
The first factor is finite
and decays as $\sigma_n^{4\beta m}$ by 
Assumption \ref{assumption:ApproximateLMHRNormality}, 
the second is bounded because of Assumption
\ref{assumption:NormalPotentialFluctuations}
and the third is bounded because of Assumption
\ref{assumption:AsymptoticProposal}. Hence, the result follows.
\end{proof}

\begin{cor}\label{cor:LindeberAS}
Let the assumptions of Framework
\ref{theassumption} be satisfied. Then almost surely (for almost every realisation of the random environment $B$)
\[
\sigma_n^{-2\beta }\Expectp{\pi(\cdot\,|B), q_n(\cdot\,|B)}{\rho^2_n1_{\rho^2_n>\sigma^\beta_n}\,|B}
\quad\xrightarrow{n\to\infty}\quad
0 \,.
\]
That is $\Expectp{\pi(\cdot\,|B), q_n(\cdot\,|B)}{\rho^2_n1_{\rho^2_n>\sigma^\beta_n}\,|B}$ almost surely decays faster than $\sigma_n^{2\beta }$.
\end{cor}

\begin{proof}
Note that all integrands are strictly positive and use Lemma~\ref{lem:Lindeberg} together with Proposition~\ref{prop:BorelCantelli} for $U_n=\Expectp{\pi(\cdot\,|B), q_n(\cdot\,|B)}{\rho^2_n1_{\rho^2_n>\sigma^\beta_n}\,|B}$, $\tilde{U}_n=0$, $\delta_n=\sigma^{2\beta}_n$ and $\kappa:=\frac{1}{2}$. Note that $m\kappa=\frac{m}{2}>\frac{3}{2}>1$ by Assumption~\ref{theassumption}.
\end{proof}

\begin{thm}\label{thm:scaling_general}
Let the assumptions of Framework
\ref{theassumption} be satisfied. 
For \(n=1, 2, \ldots\),
and for each $\bar{x}=(x_1,\dots,x_n)\in\Reals^n$,
let $\Pi_n(\bar{x}\,|B)=\prod_{i=1}^n\pi(x_i\,|B)$ and
$Q_n(\bar{x},\d\bar{y}\,|B)=\prod_{i=1}^nq_n(x_i,y_i\,|B){\d}y_i$
be respectively a target and a proposal on $\Reals^n$, 
both depending on a random environment $B$. 
If $X^{(n)}(B)\sim \Pi_n(\cdot\,|B)$ 
and $Y^{(n)}(B)\sim Q_n(X^{(n)},\d\bar{y}\,|B)$ then there is $\sigma^2>0$
such that
the Metropolis-Hastings acceptance probabilities (conditional on the underlying $B$) satisfy
\[
\alpha\left(X^{(n)}(B),Y^{(n)}(B)\right)
\quad\xrightarrow{w}\quad
\left(1\wedge\exp\right)\left(N\left(-\half\sigma^2,\sigma^2\right)\right)
\hfill\text{ as }n\to\infty\,,
\]
almost surely (almost surely in the random environment $B$).
Moreover, we may take 
\[
\sigma^2
\quad=\quad
\ell^{2\beta}\iint_{\Reals^2}h(x,z)\pi(x\,|B)\nu_2(z){\d}x{\d}z\,.
 \]
\end{thm}

\begin{proof}
We restrict ourselves to the almost sure event 
of realisations of the random environment such that 
the conclusions of Corollary~\ref{cor:rho2decay} 
and  Corollary~\ref{cor:LindeberAS} both hold simultaneously. 
For notational convenience we fix an arbitrary realisation of the random environment $B$ satisfying this event and condition on this realization, 
and in the remainder of the proof 
we omit all reference to the random environment.

The $i$-th coordinates $X^{(n)}_i$ and $Y^{(n)}_i$ of $X^{(n)}$ and $Y^{(n)}$ are jointly distributed according to
the product probability measure $\pi(x)q_n(x,y){\d}x{\d}y$.
The product structure implies 
\[
\Psi\left(X^{(n)},Y^{(n)}\right)
\quad\colon=\quad
\log\left(\frac{\Pi_n\left(Y^{(n)}\right)	Q_n\left(Y^{(n)},X^{(n)}\right)}{\Pi_n\left(X^{(n)}\right)		Q_n\left(X^{(n)},Y^{(n)}\right)}\right)
\quad=\quad
\sum_{i=1}^n\rho_n\left(X^{(n)}_i,Y^{(n)}_i\right)
\,.
\]

Because of Corollary~\ref{cor:rho2decay}, 
if we set $\sigma_n=\ell n^{-\frac{1}{2\beta}}$ then
\begin{multline*}
\Expectp{\pi q_n}{\sum_{i=1}^n\rho^2_n\left(X^{(n)}_i,Y^{(n)}_i\right)}
\quad=\quad
n\Expectp{\pi q_n}{\rho^2_n}\\
\quad=\quad
\ell^{2\beta}\sigma_n^{-2\beta}\Expectp{\pi q_n}{\rho^2_n}
\quad\xrightarrow{n\to\infty}\quad
\sigma^2 \;=\;
\ell^{2\beta}\iint_{\Reals^2}h(x,z)\pi(x)\nu_2(z){\d}x{\d}z
\,.
\end{multline*}

Moreover Corollary~\ref{cor:rho2decay} and
Corollary~\ref{cor:LindeberAS} imply that for each coordinate
$\Expectp{\pi q_n}{\rho^2_n\left(X^{(n)}_i,Y^{(n)}_i\right)}=
\Expectp{\pi q_n}{\rho^2_n}$
decays as $\sigma^{2\beta}_n$,
and 
$\Expectp{\pi q_n}{\rho^2_n\left(X^{(n)}_i,Y^{(n)}_i\right)1_{\rho_n^2\left(X^{(n)}_i,Y^{(n)}\right)_i>\sigma_n^\beta}}
=\Expectp{\pi q_n}{\rho^2_n1_{\rho^2_n>\sigma^\beta_n}}$ 
decays faster than $\sigma^{2\beta}_n$.

It is therefore a consequence of Theorem~\ref{thm:unconditional-clt} that,
as $n\to\infty$,
\[
\Psi\left(X^{(n)},Y^{(n)}\right)
\quad\xrightarrow{w}\quad
N\left(-\frac{1}{2}\sigma^2,\sigma^2\right)\,.
\]
It is immediate from the definition
of weak convergence
that
the desired result for acceptance probabilities follows,
since $1\wedge\exp$ is a bounded Lipschitz 
(hence continuous) function.
\end{proof}

%
\section{Application to Random walk Metropolis algorithms}\label{sec:RWM}
In this section we show that the 
Anomalous Scaling Framework 
\ref{theassumption} of Section 
\ref{sec:anomalous_scaling} holds for
the Random walk Metropolis algorithm
based on centered Normal proposals
when applied to a suitably perturbed
product target.
The perturbation is applied to the marginal log-density
and corresponds to addition of
a typical fBM path.

The random environment is given by a 
typical path of a two sided fBM $\fBM$ path 
with Hurst index $H\in(0,1)$.
This is
a stationary centred Gaussian process 
with covariance function given by Equation \eqref{eq:fBM} 
and with paths that are almost surely 
$\gamma$-H\"older continuous everywhere, for $0<\gamma<H$. 
In particular $\fBM\) is a continuous Gaussian process with stationary increments.

As stipulated by Theorem~\ref{thm:CLT_RWMvanilla},
the RWM proposal is symmetric multivariate normal, with
marginal probability density
given by the kernel
\(q_n(x,dz)=\tfrac{1}{\sqrt{2\pi\sigma^2_n}}
\exp\left(-\tfrac{|z-x|^2}{2\sigma^2_n}\right){\d}z\),
where $\sigma^2_n=\ell^2n^{-\frac{1}{2H}}$ for some positive constant $\ell$. The 
reference measures \(\nu_1(x)\) and 
\(\nu_2(z)\)  of the 
Framework \ref{theassumption}
are both taken to be standard Normal densities,
so that
$\nu(x,z)=\nu_1(x)\nu_2(z)=\frac{1}{2\pi}e^{-\frac{1}{2}(x^2+z^2)}$.
We will also take $\beta =H$ and $\gamma=\min(H,1-H)$.
If it can be established that the assumptions
listed in the Framework \ref{theassumption} 
all hold,
then  Theorem~\ref{thm:CLT_RWMvanilla} 
will be an immediate consequence of Theorem~\ref{thm:scaling_general}. 

The first task is to control the fluctuations
of the potential given by the random environment
\(\fBM\). As indicated above, 
we consider
\begin{equation}\label{eq:RWM:K}
K(x\,|\fBM)	\quad=\quad		\fBM_x\,.	
\end{equation}

\begin{lem}\label{lem:RWM:NormalPotentialFluctuations}
Assumption \ref{assumption:NormalPotentialFluctuations}
is satisfied.
\end{lem}
\begin{proof}
Evidently, \((K(x|\fBM):x\in\Reals)\)
is a centered Gaussian process, since it is simply
fractional Brownian motion.
Moreover its variance function is $k(x)=|x|^{2H}$ 
(defined for every real $x$).
Assumption \ref{assumption:NormalPotentialFluctuations}
requires finiteness of 
$\Expect{\exp\left(mK(x\,|\fBM)\right)}$ 
for some suitable \(m\).

In fact for every real \(m\), for any real $x$,
\[
 \Expect{\exp\left(mK(x\,|\fBM)\right)}\quad=\quad
 \frac{1}{\sqrt{2\pi |x|^{2H}}}
 \int_\Reals \exp\left(-\frac{y^2}{2|x|^{2H}}
 + my\right){\d}y
 \quad=\quad
 e^{\frac{m^2|x|^{2H}}{2}}\,,
\]
and
\[
 \int_\Reals e^{\frac{m^2|x|^{2H}}{2}}\nu_1(x){\d}x
 \quad=\quad
 \frac{1}{\sqrt{2\pi}}
 \int_\Reals e^{\frac{m^2|x|^{2H}}{2}-\frac{x^2}{2}}{\d}x
 \quad<\quad\infty\,.
\]
%
%
\end{proof}
%
%

For the RWM case the
``asymptotic behaviour of proposal'' property follows directly.
\begin{lem}\label{lem:RWM:AsymptoticRWMproposal}
Assumption \ref{assumption:AsymptoticProposal}
is satisfied. 
\end{lem}
\begin{proof}
The RWM proposal 
is given by 
\[
q_n(x,dz)
\quad=\quad
\frac{1}{\sqrt{2\pi\sigma^2_n}}
\exp\left(-\tfrac{(z-x)^2}{2\sigma^2_n}\right){\d}z
\,,
\]
so \(\sigma_n q_n(x,x+\sigma_n z)=\nu_2(z)\)
identically.
Accordingly \(L_n(x,z|\fBM)=1\),
and thus, for all \(m\),
\[
\iint_{\Reals^2}
\Expect{\left|L_n(x,z\,|B)-1\right|^{4m}}\nu(x,z){\d}x{\d}z
\quad\equiv\quad
0\,.
\]
%
%
\end{proof}

To establish 
the
``approximate normality of LMHR''
property we need to define
\begin{equation}\label{eq:RWM:M}
M_n(x,z\,|\fBM)	\quad=\quad		\fBM_{x+\sigma_nz}-\fBM_x\,.	
\end{equation}

\begin{lem}\label{lem:RWM:ApproximateLMHRNormality}
Assumption \ref{assumption:ApproximateLMHRNormality} 
is satisfied if \(M_n\) is defined using Equation \eqref{eq:RWM:M}.
\end{lem}

\begin{proof}
It is immediate that \((M_n(x,z|\fBM):x\in\Reals)\)
is a centred Gaussian process, since it is a
linear transformation of fBM. Moreover
it follows directly from the fBM covariance as given in Equation \eqref{eq:fBM}
that the variance of \(M_n(x,z|\fBM)\) is given by  \(|z|^{2H} \sigma_n^{2\beta}\)
(bearing in mind that we have chosen \(\beta =H\));
and
certainly \(|z|^{2H}\) is a function of polynomial growth.

Consider \(\Delta_n(x,z\,|\fBM)\) determined 
for all real \(x\), \(z\) and all positive 
integers \(n\) by
\[
\rho_n(x,x+\sigma_nz\,|\fBM)\quad=\quad M_n(x,z\,|\fBM)\quad+\quad\Delta_n(x,z\,|\fBM)\,.
\]
Since
\begin{equation*}
\rho_n(x,x+\sigma_nz\,|\fBM)
\quad=\quad
\log\left(\frac{\upi(x+z\,|\fBM)}{\upi(x\,|\fBM)}\right)
\quad=\quad
\fBM_{x+\sigma_nz}+\frac{(x+\sigma_nz)^2}{2}~-~\fBM_x-\frac{x^2}{2}\,,
\end{equation*}
we obtain
\begin{equation*}
 \Delta_n(x,z\,|\fBM)
 \quad=\quad
 \tfrac{1}{2}\left(x^2-(x+\sigma_nz)^2\right)
 \quad=\quad
 -\sigma_n z (x+ \tfrac12\sigma_n z)\,.
\end{equation*}
Accordingly it follows that,
for some constant \(C_m\) depending only on \(m\),
\begin{multline*}
\iint_{\Reals^2}\Expect{\left|\Delta_n(x,z\,|B)\right|^{8m}}\nu(x,z){\d}x{\d}z
\quad=\quad
\frac{1}{2\pi}\iint_{\Reals^2}\left|\sigma_n z (x+ \tfrac12\sigma_n z)\right|^{8m}e^{-\half(x^2+z^2)}{\d}x{\d}z
\\
\quad\leq\quad
C_m \times \sigma_n^{8m}
\quad\precsim\quad
\sigma_n^{8m\beta +8m\gamma}\,,
\end{multline*}
since \(\beta +\gamma=H + \min(H,1-H)\leq 1\) 
for \(H\in(0,1)\).
\end{proof}

Finally, to demonstrate 
the ``asymptotic weak dependence'' property
we define the following subsets of $\Reals^4$:
\begin{equation}\label{eq:auxset}
\auxset_n
\quad:=\quad
\left\{(x_1,z_1,x_2,z_2)\in\Reals^4 \;:\; 
|x_1-x_2|>2\sigma_n^{\half}(|z_1|+|z_2|)\right\}\,.
\end{equation}

\begin{lem}\label{lem:RWM:AsymptoticWeakDependence}
Assumption \ref{assumption:AsymptoticWeakDependence}
is satisfied using the sets \(\auxset_n\):
\begin{enumerate}[(i)]
\item\label{lem:RWM:auxset}
 $\int_{\auxset_n^c}\nu(x_1,z_1)\nu(x_2,z_2){\d}x_1{\d}z_1{\d}x_2{\d}z_2\quad\precsim\quad\sigma_n^{\half}\,.$
\item\label{lem:RWM:auxsetCorrelation}
 For any $(x_1,x_2,z_1,z_2)\in\auxset_n$,
 noting that \(2\beta+\gamma=2H+\min(H,1-H)\leq1+H\)
 if \(H\in(0,1)\),
\[
\left|\Expect{M_n(x_1,z_1)M_n(x_2,z_2)}\right|
\quad\leq\quad
\frac{H\,|2H-1|\,}{2^{2-2H}} ~|z_1|^H|z_2|^H \sigma_n^{1+H}
\quad\leq\quad
\frac{H\,|2H-1|\,}{2^{2-2H}} ~|z_1|^H|z_2|^H 
\sigma_n^{2\beta+\gamma}
\,.
\]
\item\label{lem:RWM:otherCorrelation}
Noting again
that \(\beta +\gamma=H+\min(H,1-H)\) for 
\(H\in(0,1)\),
 there exists a polynomial $g_2(x_1,z_1,x_2)$ such that 
\[
\left|\Expect{M_n(x_1,z_1\,|\fBM)K(x_2\,|\fBM)}\right|	
\quad\leq\quad
 g_2(x_1,z_1,x_2) ~\sigma_n^H\cdot \sigma_n^{\min(H,1-H)}
 \quad=\quad
 g_2(x_1,z_1,x_2) ~\sigma_n^{\beta+\gamma}
 \,.
\]
\end{enumerate}
\end{lem}

\begin{proof}
Property \ref{lem:RWM:auxset} follows by
applying Lemma~\ref{lem:auxset_properties} 
in the Appendix, 
using the sequence $a_n=\sigma_n^{\half}$.

For \ref{lem:RWM:auxsetCorrelation} first note that by \eqref{eq:fBM} the expectation $\Expect{M_n(x_1,z_1)M_n(x_2,z_2)}$ can be rewritten as
\[
-\frac{1}{2}\left(|x_1-x_2+\sigma_n z_1-\sigma_n z_2|^{2H}-|x_1-x_2-\sigma_n z_2|^{2H}-|x_1-x_2+\sigma_n z_1|^{2H}+|x_1-x_2|^{2H}\right).
\]
Consider $(x_1,x_2,z_1,z_2)\in \auxset_n$,
and apply Lemma~\ref{lem:auxset_properties}
with \(u=v=\sigma_n^{\half}\)
(assuming \(n\) large enough that \(\sigma_n\leq1\)).
It follows that
\[
\max\left(|\sigma_n z_1-\sigma_n z_2|,~|\sigma_n z_1|,~|\sigma_n z_2)|\right)
\quad<\quad
\frac{|x_1-x_2|}{2}
\]
and
\[
\min\left(|x_1-x_2+\sigma_n z_1-\sigma_n z_2|,~
          |x_1-x_2+\sigma_n z_1|,~
          |x_1-x_2-\sigma_n z_2|,~|x_1-x_2|\right)
\quad>\quad
\frac{|x_1-x_2|}{2}
\,.
\]
Hence,
\(x_1-x_2+\sigma_n z_1-\sigma_n z_2\),
\(x_1-x_2+\sigma_n z_1\),
\(x_1-x_2-\sigma_n z_2\)
and \(x_1-x_2\) are either all positive or all negative.
Consequently the function $x\mapsto |x|^{2H}$ is smooth over any of the bounded intervals with
endpoints drawn from these four points,
and so we may use apply Lemma \ref{lem:taylor1} 
to argue:
\begin{multline*}
\left|\Expect{M_n(x_1,z_1\,|\fBM)M_n(x_2,z_2\,|\fBM)}\right|
\quad\leq\quad
\\
H\times|2H-1|\times
|z_1||z_2| \sigma^2_n	
\int_0^1\int_0^1
|x_1-x_2+u\sigma_nz_1-v\sigma_nz_2|^{2H-2}{\d}u{\d}v
\\
\quad\leq\quad	 
H\times|2H-1|\times
|z_1z_2|\left|\frac{x_1-x_2}{2}\right|^{2H-2}\sigma^2_n
\quad=\quad
H\times|2H-1|\times
|z_1|^H|z_2|^H\left(\frac{2\sqrt{|z_1||z_2|}}{|x_1-x_2|}\right)^{2-2H}\sigma^2_n\\
\leq\quad	 
H\times|2H-1|\times
|z_1|^H|z_2|^H\left(\frac{|z_1|+|z_2|}{|x_1-x_2|}\right)^{2-2H}\sigma^2_n
\quad\leq\quad
\frac{H\times|2H-1|}{2^{2-2H}}|z_1|^H|z_2|^H\sigma^{1+H}_n.
\end{multline*}
Here the last step follows because of the definition
of \(\auxset_n\):
if \((x_1,x_2,z_1,z_2)\in \auxset_n\)
then
\(2\sigma_n^{\half}(|z_1|+|z_2|)<|x_1-x_2|\).

For \ref{lem:RWM:otherCorrelation},
first observe that by \eqref{eq:fBM}
\begin{multline*}
\Expect{M_n(x_1,z_1\,|\fBM)K(x_2\,|\fBM)}
\quad=\quad
\Expect{\left(\fBM_{x_1+\sigma_nz_1}-\fBM_{x_1}\right)\fBM_{x_2}}
\\
\quad=\quad\frac{1}{2}\left(|x_1+\sigma_n z_1|^{2H}-|x_1|^{2H}-|x_1-x_2+\sigma_n z_1|^{2H}+|x_1-x_2|^{2H}\right)
\,.
\end{multline*}

We now need to distinguish between
the cases \(H\lesseqgtr\half\).
First, consider the case
$H\leq 1/2$, so that $2H\leq 1$,
\(H+\min(H,1-H)=2H\),
and $\left||a|^{2H}-|b|^{2H}\right|\leq |a-b|^{2H}$ holds for real $a,b$.
Consequently
\[
\left|\Expect{M_n(x_1,z_1\,|\fBM)K(x_2\,|\fBM)}\right|
\quad\leq\quad \sigma_n^{2H}|z_1|^{2H}
\quad=\quad |z_1|^{2H} \sigma_n^{H+\min(H,1-H)}
\quad=\quad |z_1|^{2H} \sigma_n^{\beta+\gamma}
\,.
\]

Second, consider the case $H>1/2$,
so that $1< 2H< 2$ and the function $x\mapsto |x|^{2H}$ is in $\mathcal{C}^1(\Reals)$. Then for any
real $a,b$,
using \(|a|^{2H-1}<1+|a|\),
\begin{multline*}
\left||a+b|^{2H}-|a|^{2H}\right|
\quad=\quad
2H \left|b \int_0^1|a+ub|^{2H-1}\sign(a+ub){\d}u\right|
\\
\leq\quad
2H |b| \int_0^1|a+ub|^{2H-1}{\d}u
\quad\leq\quad
 2H \, |b| \, (2+|a|+|b|)
 \,,
\end{multline*}
and hence
\[
\left|\Expect{M_n(x_1,z_1\,|\fBM)K(x_2\,|\fBM)}\right|
\quad\leq\quad
H\sigma_n|z_1|~\left(4+|x_1|+|x_1-x_2|+2\sigma_n|z_1|\right).
\]
\end{proof}

The proof of Theorem~\ref{thm:CLT_RWMvanilla}
is now immediate:
\begin{proof}[Proof of Theorem~\ref{thm:CLT_RWMvanilla}]
Lemmas~\ref{lem:RWM:NormalPotentialFluctuations},~\ref{lem:RWM:AsymptoticRWMproposal},~\ref{lem:RWM:ApproximateLMHRNormality}~and~\ref{lem:RWM:AsymptoticWeakDependence} together show that the Anomalous Scaling Framework \ref{theassumption} holds for the RWM algorithm as 
described at the head of this section
and as stipulated by Theorem~\ref{thm:CLT_RWMvanilla}. 
Consequently Theorem~\ref{thm:CLT_RWMvanilla} 
is a direct consequence of Theorem~\ref{thm:scaling_general}.
\end{proof}

%
\section{Application to Metropolis adjusted Langevin algorithms}\label{sec:MALA}
In this section we show that the 
Anomalous Scaling Framework 
\ref{theassumption} of Section 
\ref{sec:anomalous_scaling} holds for
the Metropolis adjusted Langevin algorithm
based on Normal proposals
when applied to a suitably perturbed product target.
The perturbation is applied at the level of the second derivative
of the log-density of the marginal target,
adding a typical fBM path multiplied by 
a non-random localization term.

Again 
the random environment is given by a 
typical path of a two sided fBM $\fBM$ process 
with Hurst index $H\in(0,1)$,
a continuous Gaussian process with stationary increments.

As stipulated by Theorem~\ref{thm:CLT_MALAvanilla} ,
the MALA proposal 
has probability density \(q_n(x,dy\,|\fBM,c)\)
given by the multivariate normal density 
\begin{equation}\label{eq:MALAprop}
N\left(x+\frac{\sigma_n^2}{2}\nabla\left(\log\pi(x\,|\fBM;c)\right),\,\sigma^2_n\cdot I_n\right)
\end{equation} 
where $\sigma^2_n=\ell^2n^{-\frac{1}{4+2H}}$ 
for some positive constant $\ell$. 
Here \(c\) refers to the constant used
for the definition of localization in
Equation \eqref{eq:localisation}.
Again the 
reference measures \(\nu_1(x)\) and 
\(\nu_2(z)\)  of the 
Framework \ref{theassumption}
are both taken to be standard Normal densities,
so that 
$\nu(x,z)=\nu_1(x)\nu_2(z)=\frac{1}{2\pi}e^{-\frac{1}{2}(x^2+z^2)}$.
For MALA we take $\beta =2+H$ and $\gamma=\min(H,1-H)$.
Once again we need to establish that the
Anomalous Scaling Framework \ref{theassumption} holds; Theorem~\ref{thm:CLT_MALAvanilla} will then follow 
using Theorem~\ref{thm:scaling_general}.

We begin by showing that the 
log-target density has Normal fluctuations. To that end define
\begin{equation}\label{eq:MALA:K}
K(x\,|\fBM;c)	\quad=\quad		
x^2\int_0^1\fBM_{xs}\qu(xs)(1-s){\d}s
\quad=\quad
\int_0^x\fBM_{u}\qu(u)(x-u){\d}u
\,,
\end{equation}
where \(\qu(x)=\min\left\{1,~ c^{\frac{3}{2H}}\;|x|^{-3}\right\}\) is the localisation function
introduced in Section \ref{sec:results} 
by Equation \eqref{eq:localisation}. 
(The last expression above is obtained by using 
the substitution $u=sx$.)
It is convenient to focus
on potentials (the log-marginal target probability densities),
which are given by
\begin{equation}\label{eq:potential}
\V{x}
\quad=\quad
\log\left(\pi(x\,|\fBM;c)\right)
\quad=\quad -\log\left(\int_\Reals\upi(u\,|\fBM;c){\d}u\right)-\frac{x^2}{2}+\int_0^x\fBM_{u}\qu(u)(x-u){\d}u\,.
\end{equation}
Repeated differentiation yields formulae for first and second derivatives
of the potential:
\begin{align}\label{eq:potentialdot}
\Vd{x} \quad&=\quad -x+\int_0^x\fBM_{u}\qu(u){\d}u \,,
\\
\label{eq:potentialdotdot}
\Vdd{x} \quad&=\quad -1+\fBM_{x}\qu(x)\,.
\end{align}

We first establish some basic properties for
the localisation function 
\(\qu(x)\).
\begin{lem}\label{lem:localisation_properties}
The localisation function \(\qu(x)=\min\left\{1,~ c^{\frac{3}{2H}}\;|x|^{-3}\right\}\) satisfies the following:
\begin{enumerate}[(i)]
\item
\label{prop:qu-boundedness}
$\qu(x)\leq 1$ for all $x\in\Reals$.
\item
\label{prop:qu-growth}
$|x|^{2H}\qu(x)\leq c$ for all $x\in\Reals$.
\item
\label{prop:qu-lipschitz}
$\qu$ is Lipschitz with constant $3c^{-\frac{1}{2H}}$.
\end{enumerate}
\end{lem}
\begin{proof}
 Property \ref{prop:qu-boundedness} follows immediately from the
 definition. Property \ref{prop:qu-growth} follows by arguing
 separately for \(|x|\leq c^{\frac{1}{2H}}\), when 
 \(|x|^{2H}\qu(x)=|x|^{2H}\leq c\),
 and \(|x|> c^{\frac{1}{2H}}\), when 
 \(|x|^{2H}\qu(x)<|x|^{2H-3}\; c^{\frac{3}{2H}}<c^{1-3/(2H)}\; c^{\frac{3}{2H}}=c\) (note that \(H\in(0,1)\)).
 Property \ref{prop:qu-lipschitz} follows from considering the derivative
 of the continuous function \(\qu(x)\) away from the gradient discontinuities
 at \(x=\pm c^{\frac{1}{2H}}\).
\end{proof}

We first consider the ``mixed Gaussian perturbation'' property.
\begin{lem}\label{lem:MALA:NormalPotentialFluctuations}
For any positive integer $m$, Assumption \ref{assumption:NormalPotentialFluctuations}
is satisfied for all sufficiently small localisation parameters $c>0$:
\begin{enumerate}[(i)]
\item\label{prop:Kdef}
For every real $x$, set
\[
\upi(x\,|\fBM;c)	\quad=\quad		\exp(K(x\,|\fBM;c))\nu_1(x)\,.
\]
Then
 $K(x\,|\fBM;c)$ 
is a centred 
 Normal random variable 
with 
variance 
$k(x)\leq  3(1+\frac{1}{2-2H})c^{1+\frac{1}{H}}\cdot x^2$.
\item \label{prop:pot-small-c}
For any real $x$ it is the case that
$\exp\left(2m^2k(x)\right)\leq e^{ 6m^2(1+\frac{1}{2-2H})c^{1+\frac{1}{H}}\cdot x^2}$
and this 
is integrable with respect to $\nu_1$ for all sufficiently small $c>0$.
\end{enumerate}
\end{lem}

\begin{proof}
Normality in point~\ref{prop:Kdef} follows immediately from \eqref{eq:MALA:K} and
the observation that \(\fBM\)
is a zero-mean Gaussian process. 
The rest of property~\ref{prop:Kdef} 
is trivially true if $x=0$,
since 
\[
\V{0}\quad=\quad-\log\left(\int_\Reals\upi(u\,|\fBM;c){\d}u\right)
\]
is just the log of the normalizing constant, 
so we need only deal with $x\neq 0$. 
Note that the inequality $|\Gamma^{(H)}(x,y)|\leq |x|^{2H}+|y|^{2H}$ (see \eqref{eq:fBM}) implies
\begin{multline}\label{eq:potential_properties_1}
\Expect{\left(K(x\,|\fBM;c)\right)^2}	
\quad=\quad
x^4\int_0^1\int_0^1\Gamma^{(H)}(sx,tx)\qu(tx)\qu(sx)(1-s)(1-t){\d}s{\d}t
\\
\quad\leq\quad	
2x^4\cdot\int_0^1\qu(tx)(1-t){\d}t\cdot\int_0^1|sx|^{2H}\qu(sx)(1-s){\d}s\,.
\end{multline}

The definition \eqref{eq:localisation} of the localisation function 
permits the bound
\begin{multline}\label{eq:potential_properties_2}
\int_0^1\qu(tx)(1-t){\d}t	
\quad=\quad	
\int_0^{c^{\frac{1}{2H}}|x|^{-1}}\qu(tx)(1-t){\d}t
+
\int_{c^{\frac{1}{2H}}|x|^{-1}}^1\qu(tx)(1-t){\d}t
\\
\quad\leq\quad
\int_0^{c^{\frac{1}{2H}}|x|^{-1}}{\d}t
+
c^{\frac{3}{2H}}|x|^{-3}\int_{c^{\frac{1}{2H}}|x|^{-1}}^\infty t^{-3}{\d}t
\quad\leq\quad
\frac{3}{2}c^{\frac{1}{2H}}|x|^{-1}\,.
\end{multline}

Splitting the integral
\(\int_0^1|sx|^{2H}\qu(xs)(1-s){\d}s\) and
employing Lemma~\ref{lem:localisation_properties}\ref{prop:qu-growth}
and \eqref{eq:localisation} respectively to the two parts
(and noting again that \(H\in(0,1)\)),
we obtain
\begin{multline}\label{eq:potential_properties_3}
\int_0^1|sx|^{2H}\qu(xs)(1-s){\d}s
\quad=\quad	
\int_0^{c^{\frac{1}{2H}}|x|^{-1}} |sx|^{2H}\qu(sx)(1-s){\d}s
+
\int_{c^{\frac{1}{2H}}|x|^{-1}}^1 |sx|^{2H}\qu(sx)(1-s){\d}s
\\
\quad\leq\quad \left(1+\frac{1}{2-2H}\right)c^{1+\frac{1}{2H}}|x|^{-1}\,.
\end{multline}
%
%
The remainder of property \ref{prop:Kdef} is 
now established by substituting
\eqref{eq:potential_properties_2} and \eqref{eq:potential_properties_3} into \eqref{eq:potential_properties_1}.

Finally, Property \ref{prop:pot-small-c} 
is established by applying
property \ref{prop:Kdef}
to bound
\[
\exp\left(2m^2k(x)\right)
\quad\leq\quad
\exp\left(6m^2\cdot\left(1+\frac{1}{2-2H}\right)c^{1+\frac{1}{H}}x^2\right)
\,.
\]
Thus Property \ref{prop:pot-small-c} holds when 
$6m^2\left(1+\frac{1}{2-2H}\right)c^{1+\frac{1}{H}}< \frac{1}{2}$, which is to say when
$c< \left(12m^2\left(1+\frac{1}{2-2H}\right)\right)^{-\frac{H}{1+H}}$.
\end{proof}

We now establish the ``asymptotic behaviour of proposal'' property.
We begin by considering the variance and exponential 
moments of the first derivative of the potential.
\begin{lem}\label{lem:potentialdot_properties}
The following statements hold:
\begin{enumerate}[(i)]
\item
\label{prop:diffpot-moment2}
$\Vd{x}+x$ is a centred Normal random variable with
variance controlled for every real $x$ by
\[\Expect{\left(\Vd{x}+x\right)^2}
\quad\leq\quad
 3\left(1+\frac{1}{2-2H}\right)c^{1+\frac{1}{H}}\,.
\]
\item
\label{prop:diffpot-mgf1}
For every real $x,z$
\begin{multline*}
\Expect{\exp\left(8m\sigma_n\left|z\Vd{x}\right|\right)}\\
\quad\leq\quad
2\exp\left(4m\sigma_n x^2\right)\exp\left(4m\sigma_n z^2\right)\exp\left(32 m^2\sigma^2_nz^2\Expect{\left(\Vd{x}+x\right)^2}\right)\,.
\end{multline*}
Furthermore there is a convenient bound for all sufficiently large \(n\):
\[
\Expect{\exp\left(8m\sigma_n\left|z\Vd{x}\right|\right)}	
\quad\leq\quad
2\exp\left(\frac{2x^2}{3}\right)\exp\left(\frac{4z^2}{3}\right)\,.
\]
\item
\label{prop:diffpot-mgf2}
For all sufficiently large $n$, 
$\Expect{\exp\left(m\sigma^2_n\Vd{x}^2\right)}\leq \sqrt{2}\exp(\frac{2x^2}{3})$ for all real $x$.
\end{enumerate}
\end{lem}

\begin{proof}
Normality in property \ref{prop:diffpot-moment2} 
follows immediately from \eqref{eq:potentialdot} and the observation that \(\fBM\)
is a zero-mean Gaussian process.
The proof of the bound is entirely analogous to the proof of Lemma~\ref{lem:MALA:NormalPotentialFluctuations}\ref{prop:Kdef}.
%
%
Proof of property \ref{prop:diffpot-mgf1}:
this uses property \ref{prop:diffpot-moment2},
the bounds \(e^{|au|}\leq e^{|a|u} + e^{-|a|u}\)
and $|zx|\leq \frac{1}{2}(x^2+z^2)$,
and the fact that \(\Vd{x}+x\) is a centred Normal random variable
and therefore has zero mean.
\begin{multline*}
\Expect{\exp\left(8m\sigma_n\left|z\Vd{x}\right|\right)}\quad\leq\quad
\exp\left(8m\sigma_n|zx|\right)\Expect{\exp\left(8m\sigma_n|z|\left|\Vd{x}+x\right|\right)}
\\
\quad\leq\quad
2\exp\left(8m\sigma_n|zx|\right)\Expect{\exp\left(8m\sigma_n|z|\left(\Vd{x}+x\right)\right)}
\\
\quad\leq\quad
2\exp\left(4m\sigma_nx^2\right)\exp\left(4m\sigma_nz^2\right)\exp\left(32m^2\sigma^2_nz^2\Expect{\left(\Vd{x}+x\right)^2}\right)
\,.
\end{multline*}
Here the second step uses  \(e^{|au|}\leq e^{|a|u} + e^{-|a|u}\)
and the symmetry of the random variable \(\Vd{x}+x\),
while
the last step also employs the formula for the moment generating function of a centred Gaussian random variable.

The rest of property \ref{prop:diffpot-mgf1} follows 
by using property \ref{prop:diffpot-moment2}
for sufficiently large \(n\).

Proof of property \ref{prop:diffpot-mgf2}:
Take $n$ large enough (noting that \(\sigma_n\to0\))
so that $24m\sigma^2_{n}(1+\frac{1}{2-2H})c^{1+\frac{1}{H}}\leq \frac{1}{2}$ and also $4m\sigma^2_{n}<\frac{2}{3}$.
Using $a^2\leq 2b^2+2(a-b)^2$,
and bearing in mind the bound of property \ref{prop:diffpot-moment2},
\begin{multline*}
\Expect{\exp\left(2m\sigma^2_n\Vd{x}^2\right)} \quad\leq\quad
\exp(4m\sigma^2_nx^2)\Expect{\exp\left(4m\sigma^2_n(\Vd{x}+x)^2\right)}
\\
\quad=\quad
\exp(4m\sigma^2_nx^2)\left(1-8m\sigma_n^2\Expect{\left(\Vd{x}+x\right)^2}\right)^{-1/2}\\
\leq\quad
\exp(4m\sigma^2_nx^2)\left(1-24m\sigma^2_{n}\left(1+\frac{1}{2-2H}\right)c^{1+\frac{1}{H}}\right)^{-1/2}
\quad\leq\quad 
\sqrt{2}\exp\left(\frac{2x^2}{3}\right)
\,,
\end{multline*}
where the last line
uses 
the evaluation \(\Expect{e^{\lambda^2 N^2}}=(1-2\lambda^2)^{-\frac12}\)
for \(2\lambda^2<1\) when \(N\) is a standard Normal random variable.
\end{proof}

\begin{lem}\label{lem:MALA:AsymptoticRWMproposal}
Assumption \ref{assumption:AsymptoticProposal} 
is satisfied. 
\begin{enumerate}[(i)]
\item\label{prop:Ldef}
For every real $x,z$ and every positive
integer $n$,
\[
\sigma_nq_n(x,x+\sigma_n z)		\quad=\quad		L_n(x,z\,|\fBM;c)\nu_2(z)\,,
\]
where $L_n(x,z\,|\fBM;c)=\exp\left(\frac{z\sigma_n}{2}\Vd{x}-\frac{\sigma_n^2}{8}\Vd{x}^2\right)$.
\item\label{prop:Lbound}
Recall that we have stipulated $\gamma=\min(H,1-H)$. 
The random variable $L_n(x,z\,|\fBM;c)$ satisfies
\[
\iint_{\Reals^2}\Expect{\left|L_n(x,z\,|\fBM;c)-1\right|^{4m}}\nu(x,z)dxdz
\quad\precsim\quad
\sigma_n^{4m}
\quad\precsim\quad \sigma_n^{4m\gamma}\,.
\]
\end{enumerate}
\end{lem}

\begin{proof}
Property~\ref{prop:Ldef} holds by definition, since
\begin{multline*}
\sigma_nq_n(x,x+\sigma_nz)
\quad=\quad
\frac{\sigma_n}{\sigma_n\sqrt{2\pi}}\exp\left(-\frac{1}{2\sigma_n^2}\left(\sigma_nz-\frac{\sigma^2_n}{2}\Vd{x}\right)^2\right)
\\	\quad=\quad
\frac{1}{\sqrt{2\pi}}\exp\left(-\frac{1}{2}\left(z-\frac{\sigma	_n}{2}\Vd{x}\right)^2\right)
\\
\quad=\quad\frac{1}{\sqrt{2\pi}}e^{-\frac{z^2}{2}}~\cdot~\exp\left(\frac{z\sigma_n}{2}\Vd{x}-\frac{\sigma^2_n}{8}\Vd{x}^2\right)\,.
\end{multline*} 

To see \ref{prop:Lbound} note that $|e^t-1|\leq |t|e^{|t|}$ holds for all $t\in\Reals$. Using this together with a
repeated application of the Cauchy-Schwarz inequality,
note that for all large enough $n$
\begin{multline*}
\Expect{\left|L_n(x,z\,|\fBM;c)-1\right|^{4m}}
\quad=\quad
\Expect{\left|\exp\left(z\frac{\sigma_n}{2}\Vd{x}-\frac{\sigma^2_n}{8}\Vd{x}^2\right)-1\right|^{4m}}
\\
\quad\leq\quad 
\mathbb{E}\Bigg[
\left(
z\frac{\sigma_n}{2}\Vd{x}-\frac{\sigma^2_n}{8}\Vd{x}^2
\right)^{4m} \times
\hfill
\\
\times\quad \exp\left(2m\sigma_n\left|z\Vd{x}\right|\right)
\exp\left(\frac{m\sigma^2_n}{2}\Vd{x}^2\right)
\Bigg]
\\
\quad\leq\quad
 \sigma^{4m}_n~\times~\Expect{\left(\frac{z}{2}\Vd{x}-\frac{\sigma_n}{8}\Vd{x}^2\right)^{8m}}^{1/2}\times\Expect{\exp\left(8m\sigma_n\left|z\Vd{x}\right|\right)}^{1/4}\\
 \quad\times\quad\Expect{\exp\left(2m\sigma^2_n\Vd{x}^2\right)}^{1/4}\,.
\end{multline*} 
By Lemma \ref{lem:potentialdot_properties} $\Vd{x}$ is a Normal with mean $x$ and bounded variance, hence the first expectation can be controlled by a polynomial $g(x,z)$.
The second expectation is bounded by $2^{1/4}e^{\frac{x^2}{6}}e^{\frac{z^2}{3}}$ by
Lemma \ref{lem:potentialdot_properties}\ref{prop:diffpot-mgf1},
and the third is bounded by $2^{1/8}e^{\frac{x^2}{6}}$ 
by Lemma \ref{lem:potentialdot_properties}\ref{prop:diffpot-mgf2}. 
The proof is completed by observing 
that
$g(x,z)e^{\frac{x^2}{3}}e^{\frac{z^2}{3}}$ is integrable with respect to the 
reference density $\nu$,
since \(\nu(x,z)=\tfrac{1}{2\pi}e^{-\half(x^2+z^2)}\).
\end{proof}

To establish the ``approximate LMHR normality''
property,
set 
\begin{equation}\label{eq:MALA:M_n}
M_n(x,z\,|\fBM;c)
\quad=\quad \sigma^2_nz^2\int_0^1\fBM_{x+t\sigma_nz}\qu(x)(1-2t){\d}t\,.
\end{equation}

\begin{lem}\label{lem:MALA:ApproximateLMHRNormality}
Assumption \ref{assumption:ApproximateLMHRNormality} 
is satisfied
\begin{enumerate}[(i)]
\item\label{lem:MALA:Mdef} For all real $x,z$ and 
positive integers $n$
\[
\rho_n(x,x+\sigma_nz\,|\fBM;c)\quad=\quad M_n(x,z\,|\fBM;c)\quad+\quad\Delta_n(x,z\,|\fBM;c)
\]
for $\Delta_n(x,z\,|\fBM;c)=\Delta^{(1)}_n(x,z\,|\fBM;c)-\Delta^{(2)}_n(x,z\,|\fBM;c)$, where
\begin{align*}
\Delta^{(1)}_n(x,z\,|\fBM;c)
\quad&=\quad
\frac{\sigma^2_nz^2}{2}\int_0^1\fBM_{x+t\sigma_nz}\left(\qu(x+t\sigma_nz)-\qu(x)\right)(1-2t){\d}t\,,\\
\Delta^{(2)}_n(x,z\,|\fBM;c)
\quad&=\quad
\frac{\sigma^3_nz}{4} \int_0^1\Vd{x+t\sigma_nz}\Vdd{x+t\sigma_nz}{\d}t
\,.
\end{align*}

\item\label{lem:MALA:Mvariance} 
$(M_n(x,z\,|\fBM;c):x,z\in\Reals)$ 
is a centred 
Gaussian process
with 
one-point
variance 
$h(x,z)\sigma^{4+2H}_n$ where
\[
h(x,z)
\quad=\quad
\frac{1}{2}\frac{H}{2+7H+7H^2+2H^3}\qu^2(x)|z|^{4+2H}\,.
\]
\item\label{lem:MALA:Mdelta} Finally, recall $\beta=2+H$ and $\gamma=\min(H,1-H)$ implying $8\beta+8\gamma\leq 24$ and
\[
\iint_{\Reals^2}\left|\Delta^{(1)}_n(x,z\,|\fBM;c)\right|^{8m}\nu(x,z){\d}x{\d}z
\quad\precsim\quad \sigma^{24m}_n
\quad\precsim\quad \sigma^{8m\beta+8m\gamma}
\]
and
\[
\iint_{\Reals^2}\left|\Delta^{(2)}_n(x,z\,|\fBM;c)\right|^{8m}\nu(x,z){\d}x{\d}z
\quad\precsim\quad \sigma^{24m}_n
\quad\precsim\quad \sigma^{8m\beta+8m\gamma}\,.
\]
\end{enumerate}
\end{lem}
\begin{proof}
We know that by formulae \eqref{eq:potential} and \eqref{eq:MALAprop} and $a^2-b^2=(a+b)(a-b)$
\begin{multline*}
\LMHR_{n}(x,x+\sigma_nz\,|\fBM;c)
\quad=\quad \log(\pi(x+\sigma_nz\,|\fBM;c))~-~\log(\pi(x\,|\fBM;c))\\
+\quad\log(q_n(x+\sigma_nz,x\,|\fBM;c))~-~\log(q_n(x,x+\sigma_nz\,|\fBM;c))\\
=\quad
\V{x+\sigma_nz}-\V{x}\\
-\quad\frac{1}{2\sigma_n^2}\left(-\sigma_nz-\frac{\sigma_n^2}{2}\Vd{x+\sigma_nz}\right)^2+\frac{1}{2\sigma^2_n}\left(\sigma_nz-\frac{\sigma^2_n}{2}\Vd{x}\right)^2
\\
=\quad
\V{x+\sigma_nz}-\V{x}\\
-\quad\frac{\sigma_nz}{2}\left(\Vd{x}+\Vd{x+\sigma_nz}\right)
-\frac{\sigma^2_n}{8}\left(\Vd{x+\sigma_nz}^2-\Vd{x}^2\right)\,.
\end{multline*}

By Lemma~\ref{lem:taylor2} 
from the Appendix (using \(\delta=\sigma_n z\)),
and \eqref{eq:potentialdotdot},
bearing in mind that $\int_0^1(1-2t){\d}t=0$,
\begin{multline*}
\V{x+\sigma_nz}-\V{x}~-~\frac{\sigma_nz}{2}\left(\Vd{x}+\Vd{x+\sigma_nz}\right)\\
\quad=\quad
\frac{\sigma^2_nz^2}{2}\int_0^1(1-2t)\Vdd{x+t\sigma_nz}{\d}t
\quad=\quad
\frac{\sigma^2_nz^2}{2}\int_0^1\qu(x+t\sigma_nz)\fBM_{x+t\sigma_nz}(1-2t){\d}t\\
\quad=\quad
M_n(x,z\,|\fBM;c)~+~\Delta^{(1)}_n(x,z\,|\fBM;c)\,.
\end{multline*}
On the other hand,
noting that
$\frac{\partial}{\partial t}(\dot{V}^2(t))=2\dot{V}(t)\ddot{V}(t)$,
by the fundamental theorem of calculus
\begin{multline*}
\frac{\sigma^2_n}{8}\left(\Vd{x+\sigma_nz}^2-\Vd{x}^2\right)\\
\quad=\quad
\frac{\sigma^3_nz}{4}\int_0^1\Vdd{x+t\sigma_nz}\Vd{x+t\sigma_nz}{\d}t
\quad=\quad  \Delta^{(2)}_n(x,z\,|\fBM;c)
\,.
\end{multline*}

Property \ref{lem:MALA:Mvariance}:
the centred Gaussian distribution property follows from the fact that fBM
is a centred Gaussian process.
Moreover
\[
\Expect{M_n(x,z\,|\fBM;c)^2}\quad=\quad
\qu(x)^2\sigma_n^4z^4\int_0^1\int_0^1
\Gamma^{(H)}(x+t\sigma_nz,x+s\sigma_nz)(1-2t)(1-2s){\d}t{\d}s\,.
\]
Recall the formula for the covariance of fBM in ~\eqref{eq:fBM} and note that all the terms that do not depend on both $t$ and $s$ 
must vanish when integrated with respect to $(1-2t)(1-2s){\d}t{\d}s$. 
Hence
\[
\Expect{M_n(x,z\,|\fBM;c)^2}\quad=\quad
-\frac{\qu(x)^2}{2}|z|^{4+2H}\sigma^{4+2H}_n\int_0^1\int_0^1\left|t-s\right|^{2H}
(1-2t)(1-2s){\d}t{\d}s
\,.
\]
The result is now obtained by noting that the
last integral equals $-\frac{H}{2+7H+7H^2+2H^3}$ 
(non-zero and bounded for $H\in(0,1)$).
%
%

Property \ref{lem:MALA:Mdelta}:
The random variable $\Delta^{(1)}_n(x,z\,|\fBM;c)$ is centred Normal:
this again follows from the fact that fBM
is a centred Gaussian process.
Also note that by the Cauchy-Schwarz inequality,
and the quantified Lipschitz property for \(\qu\) described in 
Lemma~\ref{lem:localisation_properties}~\ref{prop:qu-lipschitz},
\begin{multline*}
\Expect{\Delta^{(1)}_n(x,z\,|\fBM;c)^2}
\quad\leq\quad
\frac{\sigma^4_nz^4}{4}\int_0^1\Expect{\left(\fBM_{x+t\sigma_nz}\right)^2}\left(\qu(x+t\sigma_nz)-\qu(x)\right)^2{\d}t\\
\quad\leq\quad
\frac{9 \sigma^6_nz^6}{4} c^{\tfrac1{H}}
\int_0^1\Expect{\left(\fBM_{x+t\sigma_nz}\right)^2}t^2{\d}t
\quad\leq\quad
\frac{9}{4}c^{-\frac{1}{H}}z^6(|x|^H+|\sigma_n z|^H)^2\times\sigma_n^{6}\,,
\end{multline*}
This yields the required
control of \(\iint_{\Reals^2}\left|\Delta^{(1)}_n(x,z\,|\fBM;c)\right|^{8m}\nu(x,z){\d}x{\d}z\),
using the fact that the summand \(\Delta^{(1)}_n(x,z\,|\fBM;c)\) is centred
Normal while the polynomial moments of \(\nu\) are all bounded.

For $\Delta^{(2)}_n(x,z\,|\fBM;c)$ consider the bound
\begin{multline*}
\Expect{\Delta^{(2)}_n(x,z\,|\fBM;c)^{8m}}
\quad\leq\quad
\frac{\sigma^{24m}_nz^{8m}}{4^{8m}}\int_0^1\Expect{\left(\Vd{x+t\sigma_nz}\Vdd{x+t\sigma_nz}\right)^{8m}}{\d}t\\
\quad\leq\quad
\frac{\sigma^{24m}_nz^{8m}}{4^{8m}}\int_0^1\Expect{\Vd{x+t\sigma_nz}^{16m}}^{1/2}\Expect{\Vdd{x+t\sigma_nz}^{16m}}^{1/2}{\d}t\,.
\end{multline*}
Both expectation can be bounded above with polynomials in $x$ and $z$, since they are expectations of powers of Normal random variables whose means and variances can be bounded by polynomials (see Lemma~\ref{lem:MALA:AsymptoticRWMproposal} and equations~\eqref{eq:potentialdot}, \eqref{eq:potentialdotdot}). 
\end{proof}

Finally, to demonstrate 
the ``asymptotic weak dependence'' property
we need to define suitable subsets of 
$\Reals^4$.
The definition is based on that of  \eqref{eq:auxset} 
but using different proposal variances $\sigma_n$)
\begin{equation}\label{eq:auxsetMALA}
\auxset_n
\quad:=\quad
\left\{(x_1,z_1,x_2,z_2)\in\Reals^4 \;:\; 
|x_1-x_2|>2\sigma_n^{1/2}(|z_1|+|z_2|)\right\}\,.
\end{equation}

\begin{lem}\label{lem:MALA:AsymptoticWeakDependence}
Assumption \ref{assumption:AsymptoticWeakDependence}
is satisfied:
\begin{enumerate}[(i)]
\item\label{lem:MALA:auxset}
 $\int_{\auxset_n^c}\nu(x_1,z_1)\nu(x_2,z_2){\d}x_1{\d}z_1{\d}x_2{\d}z_2\quad\precsim\quad\sigma_n^{1/2}\,.$
\item\label{lem:MALA:auxsetCorrelation}
 For any $(x_1,x_2,z_1,z_2)\in\auxset_n$ we have 
\begin{multline*}
\left|\Expect{M_n(x_1,z_1\,|\fBM;c)M_n(x_2,z_2\,|\fBM;c)}\right|
\quad\leq\quad
\frac{H\,|2H-1|\,}{2^{2-2H}} ~|z_1|^{2+H}|z_2|^{2+H}\sigma_n^{5+H}\\
\quad\precsim\quad |z_1|^{2+H}|z_2|^{2+H}\sigma_n^{2\beta+\gamma}\,.
\end{multline*}
\item\label{lem:MALA:otherCorrelation} There exists a polynomial $g_2(x_1,z_1,x_2)$ such that 
\[
\left|\Expect{M_n(x_1,z_1\,|\fBM;c)K(x_2\,|\fBM;c)}\right|	
\quad\leq\quad
 g_2(x_1,z_1,x_2) ~\sigma_n^{2+H+\min(H,1-H)}
 \quad=\quad g_2(x_1,z_1,x_2) ~\sigma_n^{\beta+\gamma}\,.
\]
\end{enumerate}
\end{lem}

\begin{proof}
Property \ref{lem:MALA:auxset} follows by
applying Lemma~\ref{lem:auxset_properties} 
in the Appendix, 
using the sequence $a_n=\sigma_n^{1/2}$.

Property \ref{lem:MALA:auxsetCorrelation}:
Using the formula \eqref{eq:fBM} for the 
covariance function of fBM,
\begin{multline*}
\Expect{M_n(x_1,z_1\,|\fBM;c)M_n(x_2,z_2\,|\fBM;c)}\\
\quad=\quad
z_1^2z_2^2\sigma_n^4\qu(x_1)\qu(x_2)\int_0^1\int_0^1\Gamma^{(H)}(x_1+t\sigma_nz_1,x_2+s\sigma_nz_2)(1-2t)(1-2s){\d}t{\d}s\,.
\end{multline*}
Again,
all terms not depending on both $t$ and $s$ vanish 
when integrated with respect to $(1-2t)(1-2s){\d}t{\d}s$. 
Hence, in the expression above we can swap
$\Gamma^{(H)}(x_1+t\sigma_nz_1,x_2+s\sigma_nz_2)$ for 
$$-\frac{1}{2}\left(|x_1-x_2|^{2H}-|x_1-x_2+\sigma_ntz_1|^{2H}-|x_1-x_2-\sigma_nsz_2|^{2H}+|x_1-x_2+\sigma_n(tz_1-sz_2)|^{2H}\right).$$

Using Lemma~\ref{lem:auxset_properties}(ii) 
from the Appendix with $u=\sigma_n^{1/2}t$, $v=\sigma_n^{1/2}s$ (assuming $n$ large enough that $\sigma_n\leq 1$), and 
the details of construction 
of the set $\auxset_n$ in \eqref{eq:auxsetMALA},
if it is the case that \((x_1,z_1,x_2,z_2)\in\auxset_n)\) then
it must be the case that $|x_1-x_2+\sigma_n(tz_1-sz_2)|>\frac{|x_1-x_2|}{2}>0$ and $|\sigma_n(tz_1-sz_2)|<\frac{|x_1-x_2|}{2}$ for each $t,s\in(0,1)$. So $x_1-x_2+\sigma_n(tz_1-sz_2)$ is at least a positive distance away from zero and of the same sign for all $t,s\in[0,1]$. Hence, since the function $x\mapsto |x|^{2H}$ is smooth away from zero, Lemma~\ref{lem:taylor1} from the Appendix implies that
\[
\Gamma^{(H)}(x_1+t\sigma_nz_1,x_2+s\sigma_nz_2)
\quad=\quad
H(2H-1)ts\sigma^2_nz_1z_2\int_0^1\int_0^1|x_1-x_2+\sigma_n(utz_1-vsz_2)|^{2H-2}{\d}u{\d}v\,.
\]
Therefore, by Lemma~\ref{lem:auxset_properties}(ii)
from the Appendix,
and construction of $\auxset_n$ in \eqref{eq:auxsetMALA}:
\begin{multline*}
\left|\Expect{M_n(x_1,z_1\,|\fBM;c)M_n(x_2,z_2\,|\fBM;c)}\right|	
\quad\leq\quad 
H\,|2H-1|\,|z_1z_2|^3\left|\frac{x_1-x_2}{2}\right|^{2H-2}\sigma^6_n\\
\quad=\quad
H\,|2H-1|\,|z_1|^{2+H}|z_2|^{2+H}\left(\frac{2\sqrt{|z_1||z_2|}}{|x_1-x_2|}\right)^{2-2H}\sigma^6_n
\\
\quad\leq\quad 
H\,|2H-1|\,|z_1|^{2+H}|z_2|^{2+H}\left(\frac{|z_1|+|z_2|}{|x_1-x_2|}\right)^{2-2H}\sigma^6_n\\
\quad\leq\quad
\frac{H\,|2H-1|\,}{2^{2-2H}}|z_1|^{2+H}|z_2|^{2+H}\sigma^{5+H}_n\,.
\end{multline*}

Property~\ref{lem:MALA:otherCorrelation}:
We now need to distinguish between
the cases \(H\lesseqgtr\half\).
First, consider the case
$H\leq 1/2$, so that $2H\leq 1$
Observe that by \eqref{eq:MALA:K} and \eqref{eq:MALA:M_n}
\begin{multline*}
\Expect{M_n(x_1,z_1\,|\fBM;c)K(x_2\,|\fBM;c)}	
\quad=\quad
\\
 \sigma_n^2z_1^2x^2_2	\int_0^1\int_0^1\qu(x_1)\qu(sx_2)\Gamma^{(H)}(x_1+t\sigma_nz_1,sx_2)(1-2t)(1-s){\d}t{\d}s
\quad=\quad
\\
\frac{1}{2} \sigma^2_nz_1^2x^2_2		\int_0^1\int_0^1\qu(x_1)\qu(sx_2)\left(|x_1+t\sigma_nz_1|^{2H}-|x_1|^{2H}-|x_1-sx_2+t\sigma_nz_1|^{2H}+|x_1-sx_2|^{2H}\right)
\\(1-2t)(1-s){\d}t{\d}s
\,.
\end{multline*}
The second equality holds since the difference of integrands does not depend on $t$ and thus integrates to zero. 
Since $2H\leq 1$ and $\left||x_1+t\sigma_nz_1|^{2H}-|x_1|^{2H}\right|\leq \sigma_n^{2H}|z_1|^{2H}$ (similarly for the other term),
we obtain $\left|\Expect{M_n(x_1,z_1\,|\fBM;c)K(x_2\,|\fBM;c)}\right|\leq \sigma^{2+2H}_n|z_1|^{2+2H}x^2_2$. 

Second, consider the case $H>1/2$, 
so that $0<2H-1<1$ and the function $x\mapsto |x|^{2H}$ has a continuous derivative $x\mapsto 2H\, \sign(x)\,|x|^{2H-1}$. The Fundamental theorem of calculus then implies
\begin{multline*}
\left||x_1+t\sigma_nz_1|^{2H}-|x_1|^{2H}\right|
\quad\leq\quad
2H\left|t\sigma_nz_1\int_0^1\sign(x_1+ut\sigma_nz_1)|x_1+ut\sigma_nz_1|^{2H-1}\d u\right|\\
\quad\leq\quad
2H\sigma_n|z_1|\int_0^1|x_1+ut\sigma_nz_1|^{2H-1}\d u
\quad\leq\quad
2H\sigma_n|z_1|\left(|x_1|^{2H-1}+\sigma_n^{2H-1}|z_1|^{2H-1}\right)\,.
\end{multline*}
An analogous bound holds for $\left||x_1-sx_2+t\sigma_nz_1|^{2H}-|x_1-sx_2|^{2H}\right|$ and together
\[
\left|\Expect{M_n(x_1,z_1\,|\fBM;c)K(x_2\,|\fBM;c)}	\right|
\quad\leq\quad
H\,|z_1|^3x_2^2\left(2|x_1|^{2H-1}+|x_2|+2\sigma_n^{2H-1}|z_1|^{2H-1}\right)\times \sigma_n^3\,.
\]
\end{proof}
 
Now we are in a position to prove Theorem~\ref{thm:CLT_MALAvanilla}.

\begin{proof}[Proof of Theorem~\ref{thm:CLT_MALAvanilla}]
Lemmas~\ref{lem:MALA:NormalPotentialFluctuations},~\ref{lem:MALA:AsymptoticRWMproposal},~\ref{lem:MALA:ApproximateLMHRNormality}~and~\ref{lem:MALA:AsymptoticWeakDependence} together show that the Anomalous Scaling Framework \ref{theassumption} holds for the MALA algorithm as 
described at the head of this section
and as stipulated by Theorem~\ref{thm:CLT_MALAvanilla}. 
Consequently Theorem~\ref{thm:CLT_MALAvanilla} 
is a direct consequence of Theorem~\ref{thm:scaling_general}.
\end{proof}

%
\section{Discussion}\label{sec:conclusion}
   
In this concluding section we discuss how our results relate
to considerations of Expected Squared Jump Distance, 
further research possibilities, and some practical considerations
concerning how our results might relate to questions of
practical Markov chain Monte Carlo.

\subsection{Expected squared jump distance}\label{sec:ESJD}
In the setting of either Theorem~\ref{thm:CLT_RWMvanilla} or Theorem~\ref{thm:CLT_MALAvanilla}, 
in particular when the $X_i\sim \pi(\cdot\,|\fBM)$ are 
conditionally independent and identically distributed
and $\sigma_n=\ell n^{-\frac{1}{2\beta}}$,
and given a positive sequence $\vartheta_1$, $\vartheta_2$ \ldots
decaying to zero,
we define proposals $Y_i^{(n),\vartheta}\sim q^{\vartheta}_n(X_i,dy)\sim N(X_i,\vartheta_n^2\cdot I_n)$ in the RWM case and $Y_i^{(n),\vartheta}\sim q^{\vartheta}_n(X_i,dy)\sim N(X_i+\frac{\vartheta^2_n}{2}\Vd{X_i},\vartheta_n^2\cdot I_n)$ in the MALA case. 
We also define random variables 
which measure the growth/decay rate of the
Expected Squared Jump Distance (ESJD)
relative to $\sigma^2_n$ for different scalings of proposal variance;
these are conditional expectations given \(\fBM\) as follows:
\[
\ESJD_n(\fBM,\vartheta_n)\quad=\quad
n^{\frac{1}{\beta}}\times\Expect{\left(Y^{(n),\vartheta}_1-X_1\right)^2(1\wedge \exp)\left(\sum_{i=1}^n\LMHR(X_i,Y^{(n),\vartheta}_i)\right)\,\Big|\fBM}
\,.
\]
From either Theorem~\ref{thm:CLT_RWMvanilla} or Theorem~\ref{thm:CLT_MALAvanilla} we can deduce that almost surely (when conditioned on $\fBM$)
\begin{equation}\label{eq:ESJD1}
\ESJD_n(\fBM,\sigma_n)\quad\xrightarrow{\fBM \text{ a.s.}} \quad W(\ell)\,:=\,2\ell^2\Phi\left(-\frac{\ell^\beta\theta}{2}\right)
\end{equation}
for an appropriate positive random variable $\theta$ that is $\fBM$-measurable (see Theorem~\ref{thm:CLT_RWMvanilla} or Theorem~\ref{thm:CLT_MALAvanilla} and \citet[Proposition~2.4]{RobertsGelmanGilks-1997}).
This can be shown by adopting the method of proof of Corollary 18 in \cite{ZanellaKendallBedard-2016}, where we realise all the $X_i$, $Y^{(n)}_i$ on the same probability space and use the tower property.

%
%

We seek to show that the rate of \(\ESJD_n(\fBM,\vartheta_n)\) is optimal when 
$\vartheta_n\sim \sigma_n$. More precisely, we must show
that the rate converges to zero almost surely for 
$\vartheta_n$ with decay rate differing
asymptotically from the decay
rate of \(\sigma_n\). If $\frac{\vartheta_n}{\sigma_n}\to 0$, it is straightforward to show $\ESJD_n(\fBM,\vartheta_n)\to 0$ almost surely.
Indeed, we simply note the acceptance rate is bounded above (by \(1\))
and argue that
\begin{multline*}
\limsup_{n\to\infty}n^{\frac{1}{\beta}}\times\Expect{\left(Y^{(n),\vartheta}_1-X_1\right)^2(1\wedge \exp)\left(\sum_{i=1}^n\LMHR(X_i,Y^{(n),\vartheta}_i)\right)\;\Big|\fBM}
\\
\quad\precsim\quad
\lim_{n\to\infty}\sigma_n^{-2}\times\Expect{\left(Y^{(n),\vartheta}_1-X_1\right)^2\;\Big|\fBM}
\quad=\quad 
\lim_{n\to\infty}\frac{\vartheta_n^2}{\sigma_n^{2}}
\quad \to \quad 0\,.
\end{multline*}
Unfortunately, when  $\frac{\vartheta_n}{\sigma_n}\to \infty$ we can only show convergence in probability
\begin{equation}\label{eq:ESJD:Pconv}
\ESJD_n(\fBM,\vartheta_n)\quad\xrightarrow{\text{P}} \quad 0\,.
\end{equation}

The reason is that, even though Lemmas~\ref{lem:ItotildeI},~\ref{lem:tildeI_to_hatI}~and~\ref{lem:hatI_to_J} as well as Lemma~\ref{lem:Lindeberg} all remain valid even if we use proposal variances $\vartheta^2_n$ instead of $\sigma^2_n$ (we only require $\vartheta_n\to 0$), we cannot recover a result analogous to 
the Borel-Cantelli argument of Proposition~\ref{prop:BorelCantelli} 
for arbitrary decay rates of proposal variances,
only for rates $\vartheta_n$ ``close'' to $\sigma_n$. 
In effect, it can thus be shown that 
the decay rate is ``locally optimal'', but not 
necessarily ``globally optimal''. 
This is an intrinsic issue for the Anomalous Scaling Framework method
\ref{theassumption}
described in Section \ref{sec:anomalous_scaling}: 
even with better bounds or a different random environment construction it will always be possible to construct decay rates $\vartheta_n$
that are slow enough to ensure that Borel-Cantelli arguments fail,
as a result of certain series not being summable.

The need to restrict to
convergence in probability suggests 
that the setting of \cite{RobertsGelmanGilks-1997} 
will not apply to the setting of Theorems~\ref{thm:CLT_RWMvanilla}~and~\ref{thm:CLT_MALAvanilla}, since 
of course subsequence arguments will then imply
existence of sub-sequences
of increasing dimension along which classical scaling is not optimal.
It seems unlikely that almost sure convergence would 
ever not hold, but a proof of this in the case
$\frac{\vartheta_n}{\sigma_n}\to \infty$ would have to deal with varying and very different decay rates of the proposal variance 
 $\vartheta^2_n$.

Nevertheless, for any $\frac{\vartheta_n}{\sigma_n}\to \infty$ 
we can recover weaker versions of Corollaries~\ref{cor:rho2decay}~and~\ref{cor:LindeberAS}: 
\begin{equation}\label{eq:ESJD2}
\vartheta_n^{-2\beta}\Expect{\LMHR_n^2(X,Y^{(n),\vartheta}_i\,|\fBM)\,|\fBM}\quad\xrightarrow{\text{P}}\quad \theta
\end{equation}
and
\begin{equation}\label{eq:ESJD3}
\vartheta_n^{-2\beta}\Expect{\LMHR_n^2(X,Y^{(n),\vartheta}\,|\fBM)1_{\rho^2_n(X,Y^{(n),\vartheta})>\vartheta_n^\beta}\,|\fBM}\quad\xrightarrow{\text{P}}\quad 0\,.
\end{equation}
This is enough to establish our objective, 
Equation \eqref{eq:ESJD:Pconv}. 
And for this it suffices to show that
the almost sure versions of \eqref{eq:ESJD2} and \eqref{eq:ESJD3}
imply $\ESJD_n(\fBM,\vartheta_n)\to 0$ almost surely.
Proof of convergence in probability then follows
using the celebrated
characterisation of convergence in probability 
as holding whenever sub-sequences all have 
almost surely convergent sub-sub-sequences.

\begin{lem} Assume $\frac{\vartheta_n}{\sigma_n}\to\infty$.
Almost sure versions of \eqref{eq:ESJD2} and \eqref{eq:ESJD3} imply $\ESJD_n(\fBM,\vartheta_n)\to 0$ almost surely.
\end{lem}

\begin{proof}
Theorem~\ref{thm:mean+half-variance-is-small}
(``mean plus half-variance is asymptotically negligible'')
together with almost sure versions of \eqref{eq:ESJD2} and \eqref{eq:ESJD3} implies
\begin{equation}\label{eq:ESJD4}
\vartheta^{-2\beta}_n\left(\Expect{\LMHR_n(X,Y^{(n),\vartheta}\,|\fBM)}~+~\frac{1}{2}\Var{\LMHR_n(X,Y^{(n),\vartheta}\,|\fBM)\,|\fBM}\right)
\quad\xrightarrow{\text{a.s.}}\quad	0\,.
\end{equation}
Write $\rho_i^{(n),\vartheta}=\LMHR_n(X_i,Y^{(n),\vartheta}_i\,|\fBM)$. The $\rho_i^{(n),\vartheta}$ are independent given $\fBM$, and so \eqref{eq:ESJD4} yields
\begin{equation}\label{eq:ESJD5}
\frac{1}{n\vartheta_n^{2\beta}}\Var{\sum_{i=1}^n\rho_i^{(n),\vartheta}\,|\fBM}
\quad=\quad
\vartheta_n^{-2\beta}\Var{\rho_1^{(n),\vartheta}\,|\fBM}
\quad\xrightarrow{\text{a.s.}}\quad \theta
\end{equation}
and
\[
\frac{1}{n\vartheta_n^{2\beta}}\Expect{\sum_{i=1}^n\rho_i^{(n),\vartheta}\,|\fBM}
\quad=\quad
\vartheta_n^{-2\beta}\Expect{\rho_1^{(n),\vartheta}\,|\fBM}
\quad\xrightarrow{\text{a.s.}}\quad -\frac{theta}{2}\,.
\]
Consequently, for all large enough $n$,
\begin{multline}\label{eq:ESJD7}
\Prob{\sum_{i=1}^n\rho_i^{(n),\vartheta}~>~-\frac{1}{4}\Var{\sum_{i=1}^n\rho_i^{(n),\vartheta}\,|\fBM}\,|\fBM}\\
\quad\leq\quad
\Prob{\left|\sum_{i=1}^n\left(\rho_i^{(n),\vartheta}-\Expect{\rho_i^{(n),\vartheta}\,|\fBM}\right)\right|~>~\frac{1}{5}\Var{\sum_{i=1}^n\rho_{n,i}\,|\fBM}\,|\fBM}\\
=\quad
\Prob{\left(\sum_{i=1}^n\left(\rho_i^{(n),\vartheta}-\Expect{\rho_i^{(n),\vartheta}\,|\fBM}\right)\right)^{4m}~>~\frac{1}{5^{4m}}\Var{\sum_{i=1}^n\rho_{n,i}\,|\fBM}^{4m}\,\big|\fBM}\,,
\end{multline}
for $m$ 
chosen to satisfy the requirements
of Framework~\ref{theassumption}. From the proof of Lemma~\ref{lem:Lindeberg} we may conclude
\begin{equation}\label{eq:ESJD8}
\Expect{\left(\rho_i^{(n),\vartheta}-\Expect{\rho_i^{(n),\vartheta}\,|\fBM}\right)^{4m}\,|\fBM}
\quad \leq \quad
\Expect{\left(\rho_i^{(n),\vartheta}\right)^{4m}\,|\fBM}
\quad\precsim\quad
\vartheta^{4m\beta}_n\,.
\end{equation}
Following the argument
of \citet[Proposition~26]{MijatovicVogrinc-2017},  
given centred IID random variables $A_1,\dots A_n\sim A$ satisfying $\Expect{A^{4m}}<\infty$
we have
\begin{multline}\label{eq:ImprovedMarkov}
\Expect{\left(A_1+\dots +A_n\right)^{4m}}
\quad=\quad
\mathop{\sum\cdots\sum}_{\substack{s_1+\dots +s_n=4m \\ s_i\in\{0,2,3,\dots\}}}\Expect{A_1^{s_1}\cdots A_n^{s_n}}
\\
\quad\leq\quad
\mathop{\sum\cdots\sum}_{\substack{s_1+\dots +s_n=4m \\ s_i\in\{0,2,3,\dots\}}}
\Expect{|A_1|^{s_1}}
\cdots 
\Expect{|A_n|^{s_n}}
\quad\leq\quad
\mathop{\sum\cdots\sum}_{\substack{s_1+\dots +s_n=4m \\ s_i\in\{0,2,3,\dots\}}}
\Expect{|A_1|^{4m}}^{\tfrac{s_1}{4m}}
\cdots 
\Expect{|A_n|^{4m}}^{\tfrac{s_n}{4m}}
\\
\quad=\quad
\left(\mathop{\sum\cdots\sum}_{\substack{s_1+\dots +s_n=4m \\ s_i\in\{0,2,3,\dots\}}}
1\right)
\Expect{|A_1|^{4m}}
\quad\leq\quad
3^{2m}
\left(\mathop{\sum\cdots\sum}_{t_1+\dots +t_n=2m}
1\right)
\Expect{|A_1|^{4m}}
\quad=\quad
3^{2m} n^{2m} \Expect{|A_1|^{4m}}\,.
\end{multline}
Here the first equality holds because all the terms containing exactly one copy of any of the $A_i$ vanish due to $A_i$ being centred and independent; the third inequality arises
from Jensen's inequality;
the fourth inequality is obtained
by mapping each tuple
\((s_1,\ldots,s_n)\) 
to \((t_1,\ldots,t_n)\) by dividing \(s_i\) by \(2\) if \(s_i\) is even, otherwise alternately
increasing or decreasing \(s_i\) by \(1\)
then dividing by \(2\). Each resulting 
tuple \((t_1,\ldots,t_n)\) sums to \(2m\)
and derives from no more than \(3^{2m}\)
of the
\((s_1,\ldots,s_n)\) tuples.

Using Markov's inequality on \eqref{eq:ESJD7} and then using \eqref{eq:ESJD8} together with \eqref{eq:ImprovedMarkov} shows
\begin{multline}\label{eq:ESJD6}
\Prob{\sum_{i=1}^n\rho_i^{(n),\vartheta}~>~-\frac{1}{4}\Var{\sum_{i=1}^n\rho_i^{(n),\vartheta}\;|\fBM}\;|\fBM}
\quad\precsim\quad
n^{2m}\times \frac{\vartheta_n^{4m\beta}}{\Var{\sum_{i=1}^n\rho_i^{(n),\vartheta}\;|\fBM}^{4m}}\\
\quad\precsim\quad
n^{2m}\times \frac{\vartheta_n^{4m\beta}}{n^{4m}\vartheta_n^{8m\beta}\theta^{4m}}
\quad=\quad
\frac{\sigma^{4m\beta}_n}{\vartheta_n^{4m\beta}\ell^{4m\beta} \theta^{4m}}\,,
\end{multline}
where we have used \eqref{eq:ESJD5} (changing the constant) for the second bound and 
have used $n\sigma_n^{2\beta}=\ell^{2\beta}$ 
(as stipulated in Framework \ref{theassumption})
for the final equality.

In both RWM and MALA case we have 
\begin{equation}\label{eq:ESJD Jump Control}
\Expect{\left(Y^{(n),\vartheta}_1-X_1\right)^4\,|\fBM}
\quad\precsim \quad 
\vartheta_n^4\,.
\end{equation}
The RWM case is trivial, while in the MALA case
\[
\Expect{\left(Y^{(n),\vartheta}_1-X_1\right)^4\,|\fBM}
\quad=\quad 
\vartheta^4_n \iint_{\Reals^2}\left(z+\frac{\vartheta_n}{2}\Vd{x}\right)^4\pi(x\,|\fBM;c)\nu_1(z)\d x\d z
\]
equals $3\vartheta_n^4$ plus a sum of higher powers of $\vartheta_n$ 
multiplied by
random variables that depend only on $\fBM$ and are almost surely finite by Assumption~\ref{assumption:NormalPotentialFluctuations} and Lemma~\ref{lem:potentialdot_properties}\ref{prop:diffpot-moment2}.

So now consider: under the constraint
$\sum_{i=1}^n\rho_i^{(n),\vartheta}~\leq~-\frac{1}{4}\Var{\sum_{i=1}^n\rho_i^{(n),\vartheta}\;|\fBM}$
the resulting upper bound on the acceptance rate leads, together with \eqref{eq:ESJD Jump Control}
and \eqref{eq:ESJD5}, to
\begin{multline*}
n^{\frac{1}{\beta}}\times
\Expect{\left(Y^{(n),\vartheta}_1-X_1\right)^2(1\wedge \exp)\left(\sum_{i=1}^n\LMHR(X_i,Y^{(n),\vartheta}_i\,|\fBM)\right)
\;;\;\sum_{i=1}^n\rho_i^{(n),\vartheta}~\leq~-\frac{1}{4}\Var{\sum_{i=1}^n\rho_i^{(n),\vartheta}\;|\fBM}\;|\fBM}
\\
\quad\precsim\quad
\frac{\vartheta^2_n}{\sigma_n^2}\exp\left(-\frac{1}{4}
\Var{\sum_{i=1}^n\rho_i^{(n),\vartheta}\;|\fBM}\right)
\quad\precsim\quad
\frac{\vartheta^2_n}{\sigma_n^2}\exp\left(-\frac{\ell^{2\beta}\theta}{3}\frac{\vartheta^{2\beta}_n}{\sigma_n^{2\beta}}\right)
\quad\xrightarrow{\text{a.s.}}\quad0\,.
\end{multline*}
(Here we reduce the denominator in the final exponent from \(4\)
to \(3\) to control fluctuations in the limit
for the scaled variance expressed by Equation 
\eqref{eq:ESJD5}.

Alternatively,
under the constraint
$\sum_{i=1}^n\rho_i^{(n),\vartheta}~>~-\frac{1}{4}\Var{\sum_{i=1}^n\rho_i^{(n),\vartheta}\;|\fBM}$
we can apply
the Cauchy-Schwarz inequality
together with the limits \eqref{eq:ESJD6} and \eqref{eq:ESJD Jump Control}.
Using $n\sigma_n^{2\beta}=\ell^{2\beta}$ again,
\begin{multline*}
n^{\frac{1}{\beta}}\times\Expect{\left(Y^{(n),\vartheta}_1-X_1\right)^2(1\wedge \exp)\left(\sum_{i=1}^n\LMHR(X_i,Y^{(n),\vartheta}_i\,|\fBM)\right)
\;;\;{\sum_{i=1}^n\rho_i^{(n),\vartheta}~>~-\frac{1}{4}\Var{\sum_{i=1}^n\rho_i^{(n),\vartheta}\;|\fBM}}\;|\fBM}
\\
\quad\precsim\quad
\sigma_n^{-2}\times
\Expect{\left(Y^{(n),\vartheta}_1-X_1\right)^4\;|\fBM}^{1/2}
\times
\Prob{\sum_{i=1}^n\rho_i^{(n),\vartheta}~>~-\frac{1}{4}\Var{\sum_{i=1}^n\rho_i^{(n),\vartheta}\;|\fBM}\;|\fBM}^{1/2}
\\
\quad\precsim\quad
\frac{\vartheta_n^2}{\sigma^2_n}\times
\frac{1}{\ell^{2m\beta}\theta^{2m}}\times \frac{\sigma_n^{2m\beta}}{\vartheta_n^{2m\beta}}
\quad=\quad
\frac{1}{\ell^{2m\beta}\theta^{2m}}\times \frac{\sigma_n^{2m\beta-2}}{\vartheta_n^{2m\beta-2}}
\end{multline*}
which almost surely converges to zero provided $m\beta>1$, since $\frac{\vartheta_n}{\sigma_n}\to\infty$.
In the case of MALA $\beta=2+H$ 
we need only use $m=1$;
however in the RWM case $\beta=H$, 
so we need to choose $m$ to be
sufficiently large 
(recall that all polynomial moments of $ \rho_i^{(n),\vartheta}$ are finite). 
Together the above bounds give 
\[
\ESJD(\fBM,\vartheta_n)\quad\xrightarrow{\text{a.s.}}\quad 0\,.
\]
\end{proof}


Accepting that
$\sigma^2_n$ is the optimal decay rate for proposal variances, we turn our attention to choosing the $\ell$ that maximizes the ESJD, 
equivalently (as it will turn out) 
determining
the optimal average acceptance rate. Revisiting \eqref{eq:ESJD1},
$W(\ell)$ takes the form
\[
W(\ell)\quad=\quad
2\ell^2\times\Phi\left(-\frac{\sigma(\ell)}{2}\right)
\quad=\quad
2\ell^2\times\Phi\left(-\frac{\ell^{\beta}\theta}{2}\right)
\]
for $\sigma(\ell)$ as in Theorems~\ref{thm:CLT_RWMvanilla} and \ref{thm:CLT_MALAvanilla}, 
for $\theta$ a positive constant that depends only on $\fBM$ 
and for $\Phi$ the cumulative distribution function of a standard Normal random variable. (To obtain the first equality integrate $1\wedge \exp$ with respect to the standard Normal density as in \citealp[Proposition 2.4]{RobertsGelmanGilks-1997}.)

Clearly, $W(\ell)$ is smooth, positive and converges to zero when either $\ell\to 0$ or $\ell \to \infty$. 
Its maximum is therefore achieved at a stationary point.
Taking derivatives and 
substituting
$a=\frac{\ell^{2\beta}\theta}{2}$ leads to the equation
\[
2\Phi(-a)-\beta a\varphi(-a)\quad=\quad 0
\,,
\]
where $\varphi$ is the standard Normal density function. This equation has a unique solution for positive $a$ (because $a\mapsto a\frac{\varphi(-a)}{\Phi(-a)}$ is strictly increasing) and the average acceptance rate at the optimal $a$ (and optimal $\ell$) is then given by $2\Phi(-a)=2\Phi(-\frac{\ell^{\beta}\theta}{2})$. We can solve the above equation numerically
for various $\beta(H)=H$ for RWM and $\beta(H)=2+H$ for MALA to obtain the associated optimal acceptance rates. The numerical results for both RWM and MALA are presented in Figure~\ref{fig:figure1}. Since left and right side of Figure~\ref{fig:figure1} are both obtained by numerically solving the same equation over different disjoint ranges of parameter $\beta$,
it is tempting to speculate that when using MALA for targets of smoothness class between $1$ and $2$ the optimal acceptance rates interpolate between the plots of Figure~\ref{fig:figure1} and attain values between $23\%$ and $45\%$.

\subsection{Further work and open questions}\label{sec:D&D}
\begin{enumerate}[(a)]
\item The following question remains: does there exists a "Langevin diffusion" limit result analogous to the main weak convergence results in \cite{RobertsGelmanGilks-1997}  (see \eqref{eq:Langevin})
and \cite{RobertsRosenthal-1998}? 
We do not pursue this question here as it does not fundamentally contribute to the force of the counterexamples. Note that it is not a trivial question as the gradient of the marginal does not exist in the RWM case. Hence, we can talk about an associated Langevin diffusion in terms of its Dirichlet form but not as a strong solution of an SDE with Lipschitz coefficients.

However, we expect soon to be able to obtain a positive answer, namely that
it will prove possible to show that
RWM and MALA chains (with targets and proposals as specified respectively in Theorems~\ref{thm:CLT_RWMvanilla} and \ref{thm:CLT_MALAvanilla}) converge weakly
\[
 X^{\text{RWM},(n)}_{\lfloor n^{1/H}\cdot t\rfloor,1}
\quad \xrightarrow{w} \quad U_t
\qquad\text{and}\qquad 
 X^{\text{MALA},(n)}_{\lfloor n^{1/(2+H)}\cdot t\rfloor,1}
\quad \xrightarrow{w} \quad U_t
\,.
\]
to a "Langevin Diffusion" $U$ with a speed parameter
\[
h(\ell)\quad=\quad W(\ell)\quad=\quad
2\ell^2\times\Phi\left(-\frac{\sigma(\ell)}{2}\right)
\quad=\quad
2\ell^2\times\Phi\left(-\frac{\ell^{\beta}\theta}{2}\right)
\]
where $\sigma(\ell)$ and $\theta$ are compatible with Section~\ref{sec:ESJD} above and determined by Theorems~\ref{thm:CLT_RWMvanilla} and \ref{thm:CLT_MALAvanilla}. 
(Of course this also leads to the optimal acceptance rate heuristics as noted above
at the end of Section \ref{sec:ESJD}.)

To be specific,
we plan to adapt the Dirichlet form methodology
of \citet{ZanellaKendallBedard-2016} to deliver these anomalous scaling results
at the level of weak convergence.
With the same methodology we also expect to recover the MALA results of \citet{RobertsRosenthal-1998}
with smoothness assumptions only slightly stronger than $\mathcal{C}^3(\Reals)$.
We shall report on this more general picture as part of
an upcoming review paper that will demonstrate the use 
of Dirichlet forms to provide a general framework for proving various results on optimal MCMC scaling.

\item  We note an obvious question that expands this line of thought,
namely, how much the random environment approach to optimal scaling can be generalised 
and can anything be gained from doing so? 
For example, is it feasible to take a different realisation of $\fBM$ in each coordinate 
of the product structure?
Can the realisations of the random environment be sampled for each $n$? Can we instead deal with perturbing a deterministic product target 
by a Gaussian process indexed by $\Reals^n$? These questions are challenging but attractive for further study, since this line of thinking offers a new way of expanding optimal scaling results beyond the product case.

A possibly fruitful
extension of the random environment approach
might lie in the investigation of
MCMC smoothness requirements
for boundaries.
We also note that 
random environments could be used 
to generate further kinds of counterexamples in MCMC 
(not connected to roughness of the target) or to study properties of MCMC methods
when averaged over a random environment in contexts 
where actual properties resist direct investigation.

\item Despite presenting only very particular examples we conjecture that the type of anomalous MCMC behaviour presented here happens in substantial generality and may indeed be typical when dealing with rough targets. 
One possible approach to support this conjecture 
would be to explore the actual analytical properties
provided by the random environment when arguing for
anomalous scaling results. 
In particular it would be most interesting if one could establish
that anomalous scaling was typical within a certain class of functions 
in the sense of Baire category: compare the development
of sparsity results for contours,
moving from \citet{Kendall-1980} to \citet{Kendall-1982}.

\item Another line of work that may be relevant is presented in \citet{NealRobertsYuen-2012}.
They also deal with badly behaved targets for RWM. 
They consider \emph{discontinuous} product targets, such that the one dimensional marginals are $\mathcal{C}^2$ on $[0,1]$ and zero outside.
They establish optimal scaling rate for the proposal variance $n^{-2}$ for dimension \(n\), coinciding with the case $H=\tfrac12$ in our setting.
However optimal acceptance rates differ because of different constructions of the Langevin diffusion.
Is there a link between the behaviours captured in their paper and in ours? It is natural to wonder whether both phenomena could be explained within a common framework.

\item Understanding the behaviour of MCMC methods not initiated in stationarity is very important for practical applications. Theoretically this has been studied together with optimal scaling results for instance in \citet{Christensen-2005,Jourdain-2014,Jourdain-2015,Kuntz-2018,Kuntz-2019}. It is demonstrated that (for MALA) not starting in stationarity can worsen the optimal scaling rate for 
some initial configurations, particularly some
chosen close to the mode of the target.

We did not theoretically study this question in our setting. However, numerically the RWM chains on rough targets introduced in Theorem~\ref{thm:CLT_RWMvanilla} seem to behave as predicted by the theorem despite not initiated in stationarity (see Section \ref{sec:UsefullHeuristics}\ref{item:roughRWM}).

Appropriate modification of the random environment approach could potentially be used to identify further examples of MCMC in a non-stationary phase exhibiting worse than expected scaling behaviour.

\end{enumerate}

\subsection{Heuristics for use in applications}\label{sec:UsefullHeuristics}

\begin{enumerate}[(a)]

\item\label{item:roughRWM} Let us first numerically verify what theoretical results predict. Consider an $n=200$ dimensional RWM example with $H=\half$. We pre-simulate a Brownian motion path at a very fine resolution (at $2\cdot 10^5$ equally spaced points between $-9$ and $9$) and use linear interpolation in between grid points to evaluate the target. All computation below was done with the same fixed pre-simulated Brownian motion path. Additionally, we start the RWM according to a standard Normal and use a large burn in to achieve approximate stationarity.

Even with these numerical
imperfections, the results still echo what the theory predicts. Set the variance of the marginal proposal to be $\ell^2/n^2$ for different values of $\ell$. At $\ell=5.5$ the average acceptance rate is $23.7\%$, 
with ESJD of $3.57\cdot 10^{-2}$, 
while at $\ell=13.0$ the average acceptance rate is $7.1\%$
with ESJD of $5.89\cdot 10^{-3}$, 
which appears to be near optimal.
Some other average acceptance rates and ESJD are reported in Table~\ref{tableAnomalousRWM}
(all numbers are based on a single RWM run of 
length $10^5$).
The top left image of Figure~\ref{fig2:figure2} depicts the marginal target density. The top right compares the autocorrelation of the
first coordinate of RWM algorithms 
for proposal variance tuned 
on the one hand 
to accept around $23\%$ of the proposals ($\ell=5.5$, dashed line)
and on the other hand
to accept $7\%$ of the proposals ($\ell=13$, solid line) and to attain near optimal ESJD value. The bottom picture depicts $10^4$ steps of the first coordinate of the same RWM algorithms.

\begin{Figure}
\begin{subfigure}{0.5\textwidth}
\begin{center}
\includegraphics[width=0.7\linewidth]{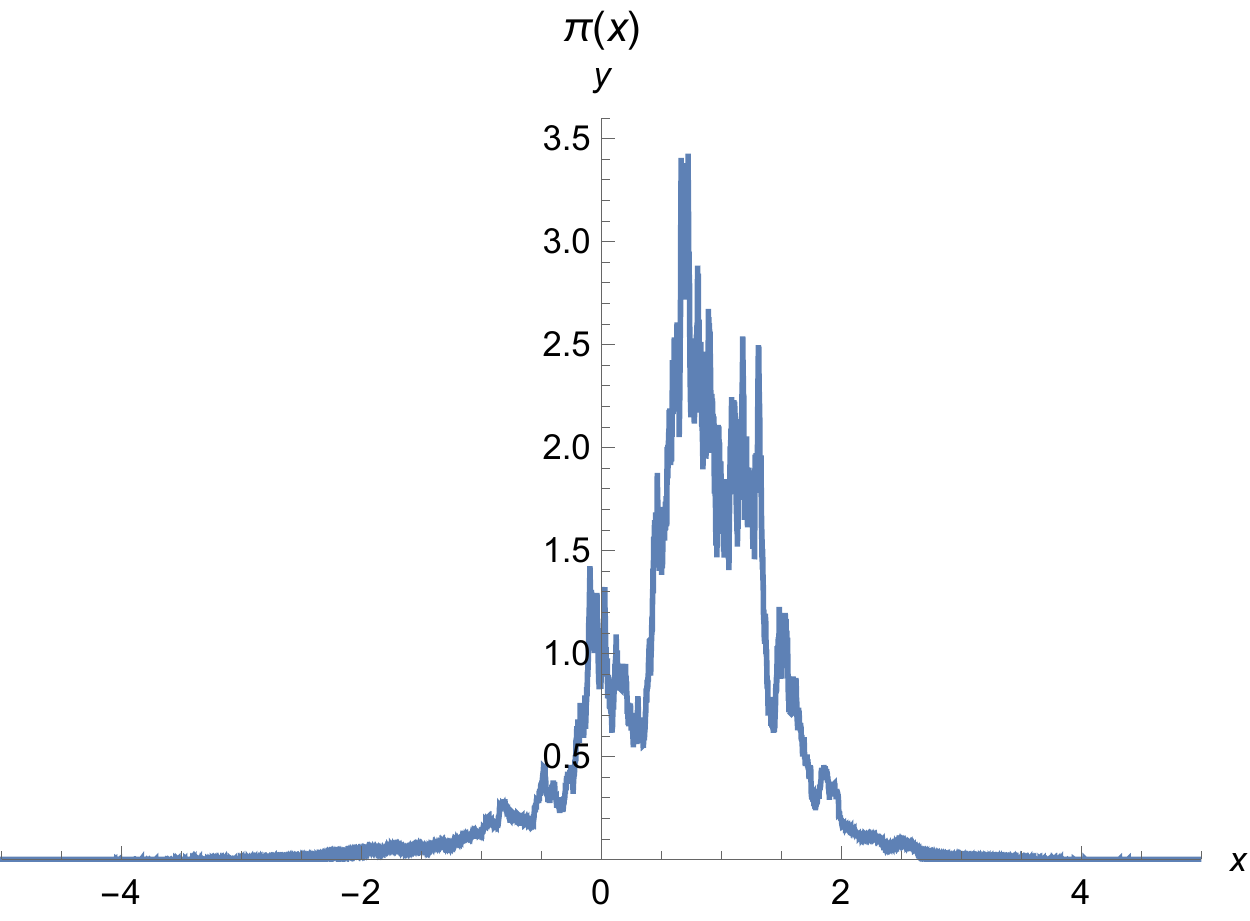}
\end{center}
\end{subfigure}%
\begin{subfigure}{0.5\textwidth}
\begin{center}
\includegraphics[width=0.7\linewidth]{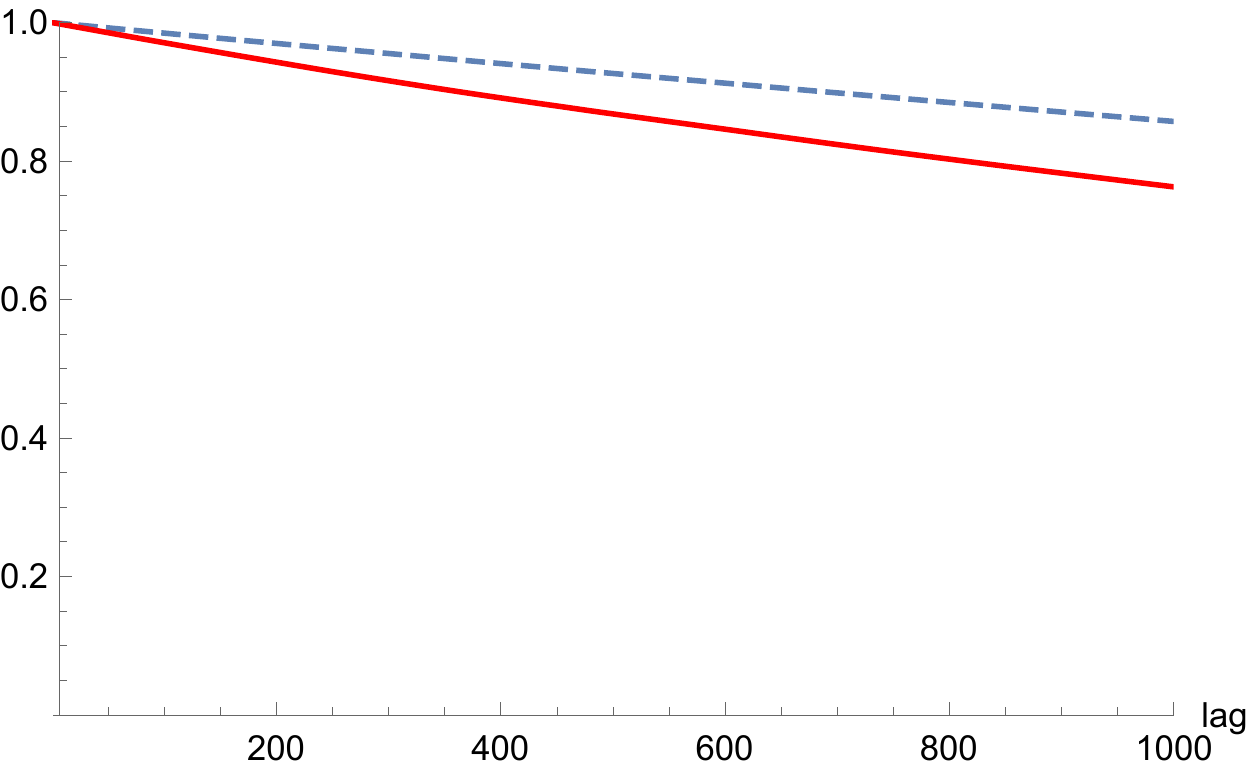}
\end{center}
\end{subfigure}\\
\vspace{10pt}

\begin{subfigure}{\textwidth}
\begin{center}
\includegraphics[width=0.7\linewidth]{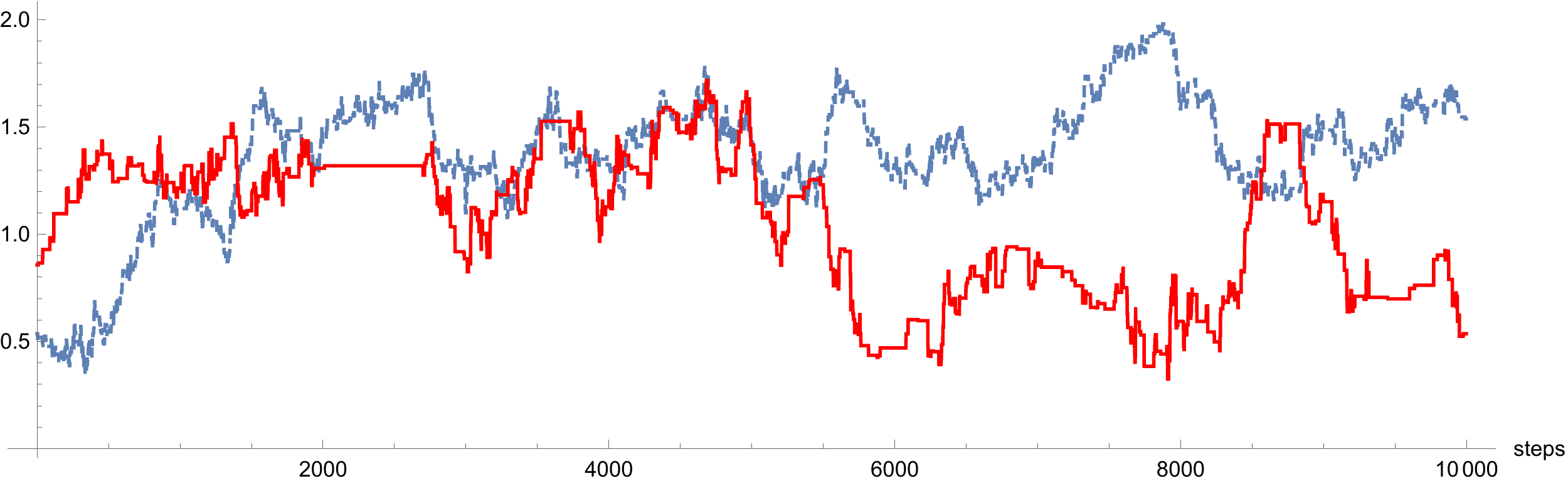}
\end{center}
\end{subfigure}
\caption{Marginal target density $\pi$ on the top left. Autocorrelation (top right) and trace (bottom) plots of first coordinate of RWM at $\ell=5.5$ tuned to accept $23\%$ of proposals (dashed) and at $\ell=13.0$ with near optimal ESJD (solid).
The near optimal ESJD scaling leads to behaviour of 
getting trapped in local modes}\label{fig2:figure2}
\end{Figure}

\begin{table}[h!]
\begin{center}
\begin{tabular}{|c|cccccc|}
\hline
$\ell$ & 5.0 & 5.5 & 11.0 & 12.0 & 13.0 & 14.0  \\ \hline
$\alpha$ & 26.2\% &  23.7\%  & 9.9\% & 7.9\% & 7.1\% & 5.9\%    \\ 
ESJD  & $3.26\cdot 10^{-2}$ & $3.57\cdot 10^{-2}$  & $5.69\cdot 10^{-2}$ & $5.66\cdot 10^{-2}$ & $5.89\cdot 10^{-2}$ & $5.69\cdot 10^{-2}$  \\
\hline
\end{tabular}
\end{center}
\caption{Average acceptance rates $\alpha$ and ESJD for different values of $\ell$ for a RWM example with rough target.}
\label{tableAnomalousRWM}
\end{table}

The observations of the acceptance rates, autocorrelation and ESJD are somewhat noisy, but it is clear that the ESJD values are the highest closer to the average acceptance rates of $7\%$ rather than to $23\%$. The autocorrelation also decays faster at $7\%$ than at $23\%$ of accepted proposals. Even the optimally tuned RWM exhibits slow mixing for the rough target.

An interesting feature we wish to note is the behaviour of the RWM path with 
near-optimal ESJD. It tends to occasionally get trapped in local modes for long periods of time, not accepting any proposal out of hundreds (see bottom image of Figure~\ref{fig2:figure2} around step $2000$).

\item \label{item:Oscillatory}
Consideration of the theoretical counterexamples presented in \ref{sec:UsefullHeuristics}\ref{item:roughRWM}
suggests that MCMC methods can get stuck in regions of high roughness in 
a manner similar to the way in which they can get stuck in local modes. Furthermore according to the MALA counterexample, these rough patches may only manifest at the level of the target derivatives, and hence may
not be immediately detectable from the plot of the target, while still slowing down mixing.  
We would also expect problems in practice
with application of RWM and MALA methods to finite dimensional targets
falling in the regimes described in \cite{RobertsGelmanGilks-1997} and \cite{RobertsRosenthal-1998}, but 
possessing regions of high local oscillations (at second order for MALA).
In such cases one might expect to need to tune acceptance rate to a lower value
than conventionally indicated. 

Indeed, consider the following toy numerical example. 
Take an $n=100$ dimensional RWM chain 
with a product target defined by the requirement
that the potential of the one-dimensional marginal is
\[
\log(\pi(x))\quad=\quad-\frac{x^2}{2}~+~a\cos(bx)
\]
for constants $a=0.25$, $b=30$. 
Further take the proposal variance of marginal proposal to be equal to $\ell n^{-1}$.
At $\ell=0.65$ the average acceptance rate is $23.3\%$, 
with ESJD of $9.77\cdot 10^{-4}$, 
while at $\ell=2.55$ the average acceptance rate is $7.7\%$
with ESJD of $4.88\cdot 10^{-3}$, 
which appears to be close to optimal.
Some other average acceptance rates and ESJD are reported in Table~\ref{tableRWM}
(all numbers are based on a single RWM run of 
length $10^6$ started in stationarity).
Again the top left image of Figure~\ref{fig3:figure3} depicts the marginal target density. The top right compares the autocorrelation of the
first coordinate of RWM algorithms 
for proposal variance tuned 
on the one hand 
to accept around $23\%$ of the proposals ($\ell=0.65$, dashed line)
and one the other hand
to attain near optimal ESJD value ($\ell=2.55$, solid line). The bottom picture depicts $10^4$ steps of the first coordinates of the same RWM algorithms.

\begin{Figure}
\begin{subfigure}{0.5\textwidth}
\begin{center}
\includegraphics[width=0.7\linewidth]{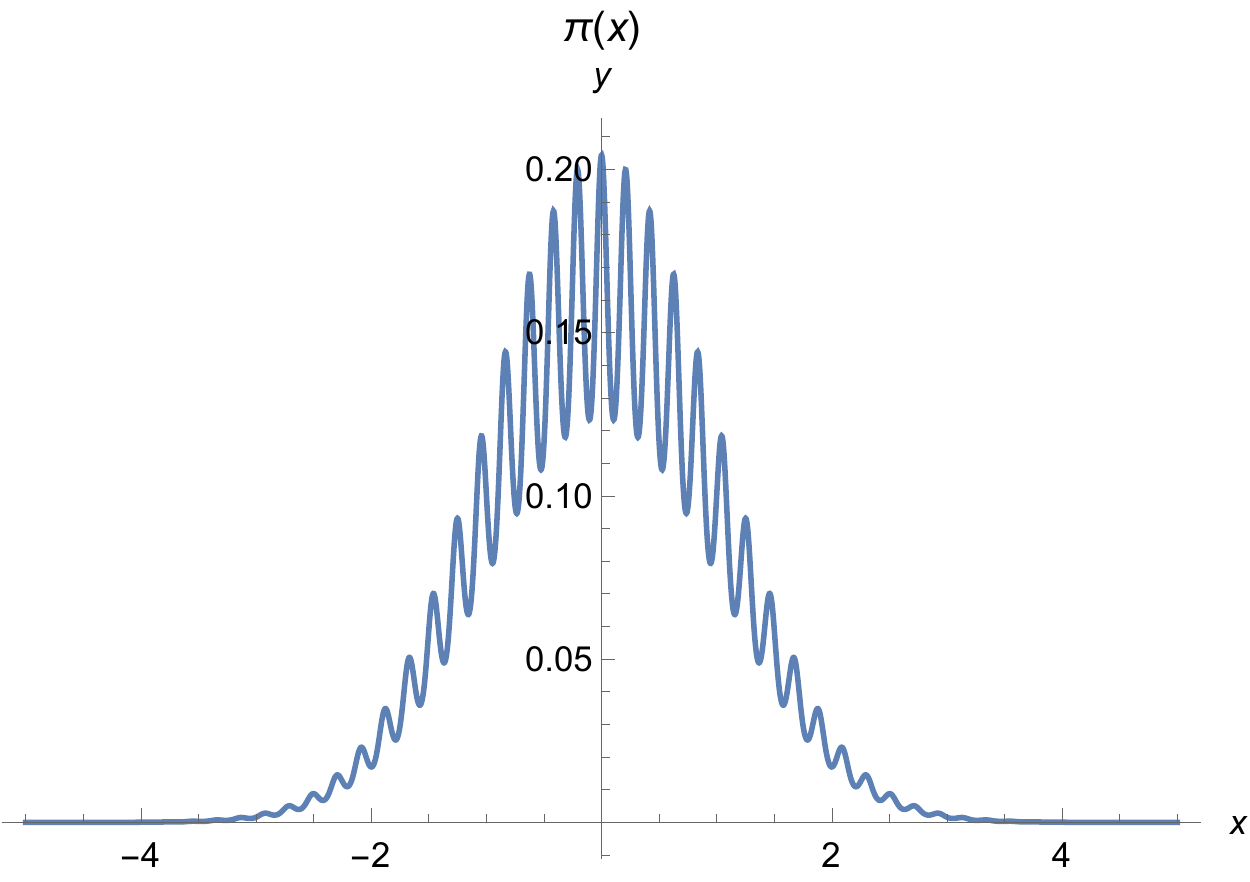}
\end{center}
\end{subfigure}%
\begin{subfigure}{0.5\textwidth}
\begin{center}
\includegraphics[width=0.7\linewidth]{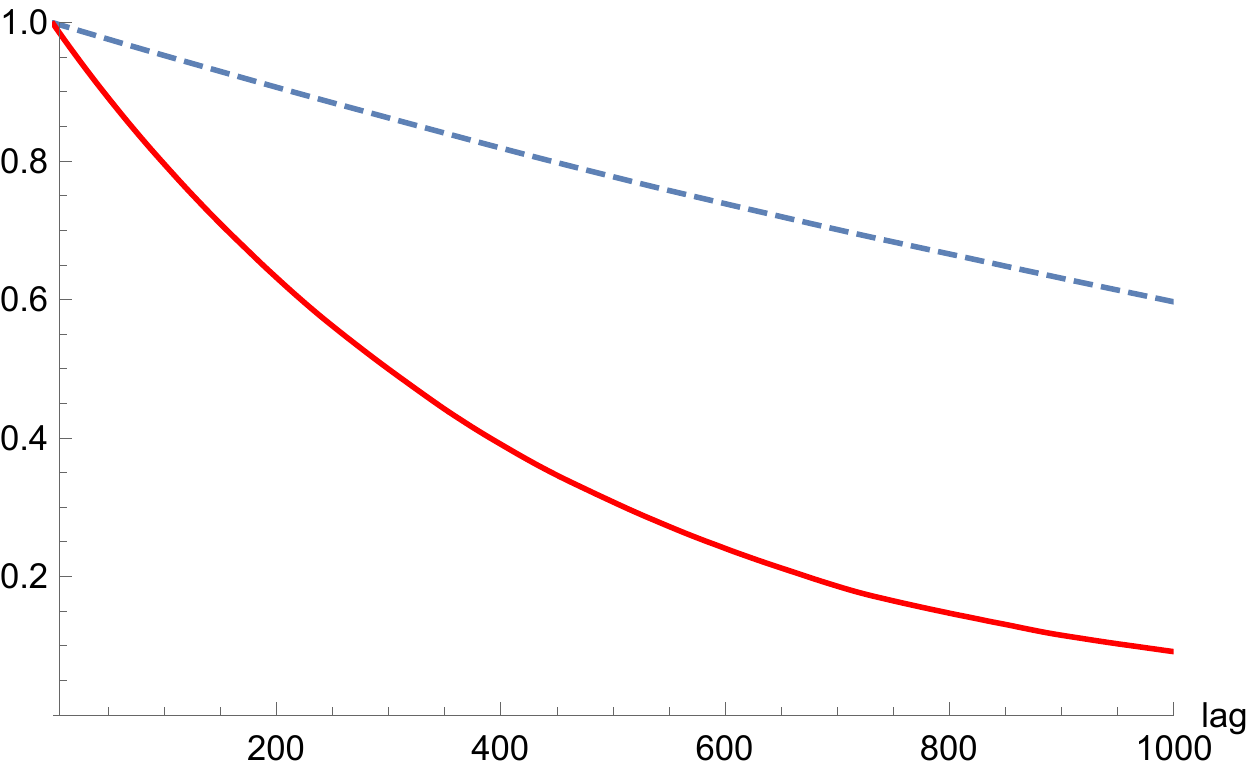}
\end{center}
\end{subfigure}\\
\vspace{10pt}

\begin{subfigure}{\textwidth}
\begin{center}
\includegraphics[width=0.7\linewidth]{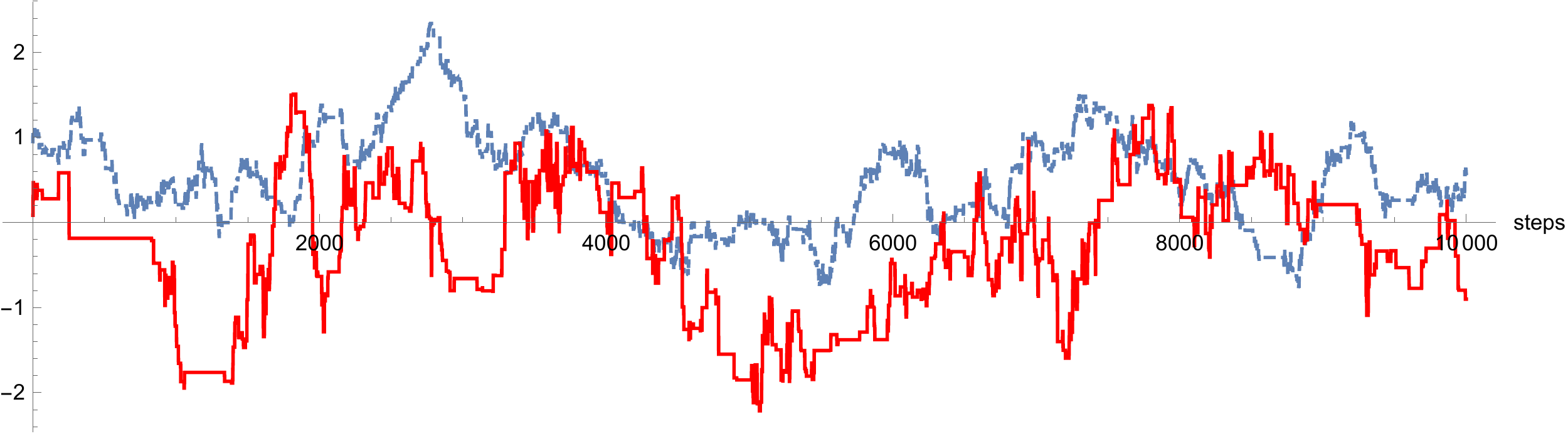}
\end{center}
\end{subfigure}
\caption{Marginal target density $\pi$ on the top left. Autocorrelation (top right) and trace (bottom) plots of first coordinate of RWM at $\ell=0.65$ tuned to accept $23\%$ of proposals (dashed) and at $\ell=2.55$ with near optimal ESJD (solid).
Again the near-optimal ESJD option leads to
behaviour of getting trapped in local modes
}\label{fig3:figure3}
\end{Figure}

\begin{table}[h!]
\begin{center}
\begin{tabular}{|c|cccccc|}
\hline
$\ell$ & 0.5 & 0.65 & 1.5 & 2 & 2.55 & 3  \\ \hline
$\alpha$ & 29.3\% &  23.3\%  & 14.7\% & 11.1\% & 7.7\% & 5.2\%    \\ 
ESJD  & $7.25\cdot 10^{-4}$ & $9.77\cdot 10^{-4}$  & $3.28\cdot 10^{-3}$ & $4.37\cdot 10^{-3}$ & $4.88\cdot 10^{-3}$ & $4.58\cdot 10^{-3}$  \\
\hline
\end{tabular}
\end{center}
\caption{Average acceptance rates $\alpha$ and ESJD for different values of $\ell$ for a smooth RWM example.}
\label{tableRWM}
\end{table}
Again the ESJD values and the autocorrelation plot suggest that RWM tuned to accept $7\%$ of proposals outperforms the RWM tuned to accept $23\%$ of proposals. The mixing is considerably faster than in the rough example \ref{sec:UsefullHeuristics}\ref{item:roughRWM} but is still slow.

We observe the same phenomenon as in \ref{sec:UsefullHeuristics}\ref{item:roughRWM}.
The solid line graph, 
corresponding to the optimally tuned proposal in terms of ESJD, has low acceptance rate and spends very long periods of time in particular states with high target density value.

Is this behaviour simply due to apparent multi-modality of the target? 
We do agree it is not unrelated, after all roughness and local oscillations are both in some sense extreme cases of local multi-modality. 
Note however, that the work of 
\cite{RobertsGelmanGilks-1997}
assures us that for considerably larger \(n\) 
we will see standard optimal scaling, despite the distance between neighbouring nodes relative to the proposal size not growing and modes becoming more pronounced due to multiplication of the marginal densities.

A natural question arises: 
can fixed deterministic marginal target densities 
of this kind be associated with an ``appropriate''
H\"older exponent?
In this case we obtain the same acceptance rate as 
in Theorem~\ref{thm:CLT_RWMvanilla} for $H\approx 0.5$, 
but it would be preferable to establish a link without having to optimize ESJD beforehand. 
If such a link can be established, 
can it be used together with the insights of Theorem~\ref{thm:CLT_RWMvanilla} to develop heuristics
on how to tune the proposal variance for 
sampling from locally oscillatory or multi-modal targets?

Consider now a similar example for MALA instead of RWM. Take an $n=100$ dimensional MALA chain 
with the potential of the one-dimensional marginal equal to
\[
\log(\pi(x))\quad=\quad-\frac{x^2}{2}~-~\frac{a}{b^2}\cos(bx)
\]
for $a=0.9$ and $b=5$. This target is log concave and looks very much like the standard Normal density and the oscillations only happen at the level of the second derivative of $\log(\pi)$: we therefore do not present a figure.
 Take the proposal variance of marginal proposal to be equal to $\ell n^{-\frac{1}{3}}$.
 Again, we can detect that the algorithm does not behave according to the theory \cite{RobertsRosenthal-1998} and has the best ESJD for lower acceptance rates.
At $\ell=1.51$ the average acceptance rate is $57.4\%$, 
with ESJD of $0.315$, 
while at $\ell=1.68$ the average acceptance rate is $47.5\%$
with ESJD of $0.331$, 
which appears to be near optimal.
 Some other average acceptance rates and ESJD are reported in Table~\ref{tableMALA}.

\begin{table}[h!]
\begin{center}
\begin{tabular}{|c|ccccccccc|}
\hline
$\ell$ & 1.4 & 1.51 & 1.6 & 1.67 & 1.68  & 1.7  & 1.72 & 1.73 & 1.8	\\ \hline
$\alpha$ & 62.9\% &  57.4\%  & 52.4\% & 48.1\% & 47.5\% & 46.3\%  & 45.2\% & 44.5\% &40.0 \\ 
ESJD  & 0.292 & 0.315  & 0.327 & 0.330 & 0.331 & 0.331  & 0.331	& 0.330 &	0.325	\\
\hline
\end{tabular}
\end{center}
\caption{Average acceptance rates $\alpha$ and ESJD for different values of $\ell$ for the MALA example.}
\label{tableMALA}
\end{table}

Results of Table~\ref{tableMALA} (each entry is again based on a single MALA run of length $10^6$ started at stationarity) are less precise then those of of Table~\ref{tableRWM} 
as 
in the MALA case
the ESJD do not vary so much over the range
of average acceptance rates $45.2\%-57.4\%$ permitted by 
the conditions of Theorem~\ref{thm:CLT_MALAvanilla}.
Moreover the numerical results suggest that, while
there is detectable deviation from classical results, 
nevertheless this does not have a
significant impact on the performance of the method.
Put differently, MALA tuned to accept anywhere between $45.2\%$ and $57.4\%$ of proposals works fine.
Worrying about the roughness of the second derivative does not seem fruitful in practice.

The above considerations and these numerical examples suggest it would be valuable to conduct a thorough numerical study (using a variety of locally oscillating targets in different dimensions) to investigate this further. A referee suggests that it would be very interesting to compare (theoretically and numerically) the performance of an MCMC algorithm on a rough target with the performance on an associated truncated Karhunen-Loève expansion of the random part of the target. This might shed light on the way in which different levels of oscillation  and roughness affect performance of MCMC algorithms.

\item
Sampling from targets with local oscillations is a matter of current scientific discussion, for instance applications in disordered media and soft matter \citep{Owhadi-2003,BenArous-2003,Pollak-2008,Duncan-2016,Hu-2018}. In light of our theoretical and numerical results it is sensible to argue that classic MCMC algorithms do not really work well for rough or oscillatory targets, and that one should seek appropriate modifications. Suggestions for such modifications also already exist \citep{Plechavc-2019}. Such modifications could usefully be assessed in terms 
of optimally tuned RWM and MALA algorithms 
providing benchmark MCMC methods.

It is also not impossible to imagine situations where one would want to sample rough targets, such as those in \ref{sec:UsefullHeuristics}\ref{item:roughRWM}. This could happen naturally if rough targets are  interpreted as noisy observations of a smooth target. This may be relevant for understanding pseudo-marginal Metropolis-Hastings algorithms \citep{Andrieu-2009,Beaumont-2003}. It is interesting to compare our results to the optimal scaling results for pseudo-marginal RWM obtained in \citet{Sherlock-2015}, who in case of stationary Gaussian noise (Section~3.2) obtain the standard scaling of proposal variance $n^{-1}$ but the exact same optimal acceptance rate $7.001\%$ as in our case for $H=\half$. Again it seems possible that one could develop common framework for studying noisy targets which would simultaneously explain both results.

Another more speculative usage of rough targets is when attempting to sample objects of fractal-like nature. 
Take Bayesian inference of ancestral trees as an example. 
There are a myriad ways in which an MCMC move on the space of trees can change the tree topology.
Combined with complex likelihood structure arising when modelling mutations this seems capable of resulting in a setting that in the limit (of say growing number of tree leaves) approaches a rough target.
In fact MCMC algorithms on trees do indeed suffer from very low acceptance rates 
when the proposal alters the tree topology, 
as reported for instance in \citet{Lakner-2008} and \citet{HohnaDrummond-2011}.
Further investigation is needed to determine if 
this can be accounted for by some kind of effective roughness or local oscillations of the target.

\end{enumerate}

\bigskip
\noindent \textbf{Acknowledgements.} 
This work was funded by the UK EPSRC under grant EP/R022100.
The second author acknowledges the support of 
the Alan Turing Institute under EPSRC grant EP/N510129.

We thank Jere Koskela, Anastasia Papavasiliou and Giacomo Zanella for useful debate about material presented in Discussion.

This is a theoretical research paper and, as such, no new data were created during this study. 
%
   \bibliographystyle{wchicago}
   \bibliography{VogrincKendall1}

%
\appendix 
%
\section{Auxiliary calculus results}\label{sec:Taylor}
This appendix establishes two simple lemmas 
concerning exact second-order Taylor expansions and a lemma establishing properties of a certain kind of set. All are used in the paper.
\begin{appendixlem}\label{lem:taylor1}
Let $f\in\mathcal{C}^2(\mathbb{I})$ for an interval $\mathbb{I}\subset \Reals$. The following identity holds, provided that $x$, $x+\delta_1$, $x+\delta_2$, and $x+\delta_1+\delta_2$ all belong to $\mathbb{I}$,
\[
f(x+\delta_1+\delta_2)-f(x+\delta_1)-f(x+\delta_2)+f(x)\quad=\quad\delta_1\delta_2\int_0^1\int_0^1f''(x+u\delta_1+v\delta_2){\d}u{\d}v\,.
\]
\end{appendixlem}

\begin{proof}
The fundamental theorem of calculus implies that 
$F(y)-F(x)=(y-x)\int_0^1F'(x+u(y-x)){\d}u$
holds
for every $F\in\mathcal{C}^1(\Reals)$ and all real $x,y$. 
This can be employed
once for $F_1(x)=f(x+\delta_1)-f(x)$ and $y=x+\delta_2$,
and then again for $F_u(x)=f'(x+u\delta_2)$ and $y=x+\delta_1$,
to yield:
\begin{multline*}
f(x+\delta_1+\delta_2)-f(x+\delta_1)-f(x+\delta_2)+f(x)
\quad=\quad
F_1(x+\delta_2)-F_1(x)
\quad=\quad
\delta_2\int_0^1F_1'(x+u\delta_2){\d}u
\\
\quad=\quad 
\delta_2\int_0^1f'(x+\delta_1+u\delta_2)-f'(x+u\delta_2){\d}u
\quad=\quad 
\delta_2\int_0^1F_u(x+\delta_1)-F_u(x){\d}u
\\
\quad=\quad 
\delta_1\delta_2\int_0^1\int_0^1F'_u(x+v\delta_1){\d}v{\d}u
\\
\quad=\quad 
\delta_1\delta_2\int_0^1\int_0^1f''(x+u\delta_1+v\delta_2){\d}u{\d}v
\,.
\end{multline*}
\end{proof}

\begin{appendixlem}\label{lem:taylor2}
Let $f\in\mathcal{C}^2(\Reals)$. The
following holds for all real $x,\delta$:
\[
f(x+\delta)-f(x)-\frac{\delta}{2}\left(f'(x)+f'(x+\delta)\right)
\quad=\quad
\frac{\delta^2}{2}\int_0^1(1-2t)f''(x+t\delta){\d}t\,.
\]
\end{appendixlem}

\begin{proof}
Consider exact second-order Taylor expansions of $f(x+\delta)$ around $x$ and 
of $f(x)$ around $x+\delta$:
\begin{align*}
f(x+\delta)&\quad=\quad
f(x)+\delta f'(x)+\int_x^{x+\delta} f''(u)(x+\delta-u){\d}u\,,
\\
f(x)&\quad=\quad
f(x+\delta)-\delta f'(x+\delta)+\int_{x+\delta}^x f''(v)(x-v){\d}v\,.
\end{align*}
These yield two different expansions for \(f(x+\delta)-f(x)\).
Averaging, we obtain
\begin{align*}
&f(x+\delta)-f(x)-\frac{\delta}{2}\left(f'(x)+f'(x+\delta)\right)
\quad=\quad
\frac{1}{2}\int_x^{x+\delta} f''(u)(x+\delta-u){\d}u
 -  \frac{1}{2}\int_{x+\delta}^x f''(v)(x-v){\d}v
 \\
&\quad=\quad\frac{\delta^2}{2}\left(\int_0^1 (1-t)f''(x+t\delta){\d}t
 -  \int_0^1 tf''(x+t\delta){\d}t\right)
\quad=\quad\frac{\delta^2}{2}\int_0^1(1-2t)f''(x+t\delta){\d}t\,,
\end{align*}
respectively using changes of variables $t=(u-x)/\delta$ and $t=(v-x)/\delta$.
\end{proof}

\begin{appendixlem}\label{lem:auxset_properties}
Let $\{a_n\}_{n\in\Numbers}$ be a strictly decreasing positive sequence and denote for each $n\in\Numbers$ the set
\[
\auxset_n
\quad:=\quad
\left\{(x_1,z_1,x_2,z_2)\in\Reals^4 \;:\; 
|x_1-x_2|>2a_n(|z_1|+|z_2|)\right\}\,.
\]
Then the following two statements hold:
\begin{enumerate}[(i)]
\item\label{prop:nu-aux-interaction}
$\frac{1}{4\pi^2}\int_{\auxset^c_n}e^{-\frac{1}{2}(x_1^2+z_1^2+x_2^2+z_2^2)}	{\d}x_1{\d}x_2{\d}z_1{\d}z_2
\quad\leq \quad
\frac{4}{\pi^{3/2}}\cdot a_n
$.
\item\label{prop:aux-geometry}
For all $(x_1,z_1,x_2,z_2)\in\auxset_n$ and $u,v \in[0,1]$,
\begin{equation*} 
 |u z_1-v z_2| a_n
 \quad<\quad  
                \tfrac12 |x_1-x_2| 
 \quad<\quad 
                 |x_1-x_2 - (u z_1-v z_2)a_n|\,.
\end{equation*}
\end{enumerate}
\end{appendixlem}

\begin{proof} 
Property \ref{prop:nu-aux-interaction}:
Consider the orthonormal change of variables given by
$y_1=\frac{1}{\sqrt{2}}(x_1-x_2)$, $y_2=\frac{1}{\sqrt{2}}(x_1+x_2)$.
This yields the following bound
using the simple-minded bound 
\(\int_{-a}^a e^{-y_1^2/2}{\d}y_1\leq 2a\):
\begin{multline*}
\frac{1}{4\pi^2}\int_{\auxset^c_n}e^{-\frac{1}{2}(x_1^2+z_1^2+x_2^2+z_2^2)}{\d}x_1{\d}x_2{\d}z_1{\d}z_2\\
\quad=\quad
\frac{1}{2\pi}\int_{\Reals^2}
e^{-\tfrac12(z^2_1+z^2_2)}
\left(\frac{1}{2\pi}\int_{|y_1|\leq \sqrt{2}a_n(|z_1|+|z_2|)}
e^{-\tfrac12(y^2_1+y^2_2)}{\d}y_1{\d}y_2\right)
{\d}z_1{\d}z_2
\\
\quad\leq\quad
\frac{\sqrt{2}a_n}{\pi}\cdot\frac{1}{2\pi}
 \iint_{\Reals^2}   (|z_1|+|z_2|)e^{-\tfrac12(z^2_1+z^2_2)}
{\d}z_1{\d}z_2
\\
\quad=\quad
\frac{2\sqrt{2}a_n}{\pi}\cdot\frac{1}{2\pi}
 \iint_{\Reals^2}   |z_1|e^{-\tfrac12(z^2_1+z^2_2)}
{\d}z_1{\d}z_2
\quad=\quad\frac{4}{\pi^{3/2}}\cdot a_n
\,.
\end{multline*}

Property \ref{prop:aux-geometry}:
Working with the definition of \(\auxset_n\),
we deduce
\[
 |u z_1-v z_2| a_n \quad\leq\quad
 (|z_1|+|z_2|) a_n
 \quad<\quad \tfrac12|x_1-x_2|\,.
\]
On the other hand,
\[
 \tfrac12|x_1-x_2| \;=\; |x_1-x_2|-\tfrac12|x_1-x_2|
 \;<\; |x_1-x_2|- |u z_1-v z_2|a_n
 \;\leq\; |x_1-x_2 - (u z_1-v z_2)a_n|\,.
\]
\end{proof}

\twocolumn
\pagenumbering{alph} 
\setcounter{page}{1}


\end{document}